\documentclass[10pt]{article}

\usepackage{enumerate,xspace}
\usepackage{amsmath,amssymb,wasysym,amsthm}
\usepackage[all]{xy}
\usepackage{proof}
\usepackage[svgnames]{xcolor}
\usepackage{pict2e} 
\usepackage{tikz}
\usepackage{stmaryrd} 
\usepackage{mathtools}
\usepackage{latexsym}
\usepackage{dsfont}
\usepackage{multicol}
\usepackage{cmll}
\usepackage{multirow}
\usepackage{longtable}

\usepackage{comment} 

\usepackage{tikz}
\usetikzlibrary{arrows,backgrounds}
\usetikzlibrary{decorations.pathmorphing}
\tikzset{node distance=2cm, auto}
\usepackage{tikz-cd}

\usepackage{lscape}
\usepackage{array}

\usepackage{hyperref} 
\hypersetup{
colorlinks,
citecolor=red,
filecolor=red,
linkcolor=blue,
urlcolor=red
}

\newtheorem{observation}{Remark}[subsection]
\newtheorem{lemma}[observation]{Lemma}  
\newtheorem{thm}[observation]{Theorem}
\newtheorem{definition}[observation]{Definition}
\newtheorem{proposition}[observation]{Proposition} 
\newtheorem{corollary}[observation]{Corollary}

\theoremstyle{definition}
\newtheorem{example}[observation]{Example}

\makeatletter

\newdimen\w@dth

\def\setw@dth#1#2{\setbox\z@\hbox{\scriptsize $#1$}\w@dth=\wd\z@
\setbox\@ne\hbox{\scriptsize $#2$}\ifnum\w@dth<\wd\@ne \w@dth=\wd\@ne \fi
\advance\w@dth by 1.2em}

\def\t@^#1_#2{\allowbreak\def\n@one{#1}\def\n@two{#2}\mathrel
{\setw@dth{#1}{#2}
\mathop{\hbox to \w@dth{\rightarrowfill}}\limits
\ifx\n@one\empty\else ^{\box\z@}\fi
\ifx\n@two\empty\else _{\box\@ne}\fi}}
\def\t@@^#1{\@ifnextchar_ {\t@^{#1}}{\t@^{#1}_{}}}

\def\t@left^#1_#2{\def\n@one{#1}\def\n@two{#2}\mathrel{\setw@dth{#1}{#2}
\mathop{\hbox to \w@dth{\leftarrowfill}}\limits
\ifx\n@one\empty\else ^{\box\z@}\fi
\ifx\n@two\empty\else _{\box\@ne}\fi}}
\def\t@@left^#1{\@ifnextchar_ {\t@left^{#1}}{\t@left^{#1}_{}}}

\def\two@^#1_#2{\def\n@one{#1}\def\n@two{#2}\mathrel{\setw@dth{#1}{#2}
\mathop{\vcenter{\hbox to \w@dth{\rightarrowfill}\kern-1.7ex
\hbox to \w@dth{\rightarrowfill}}%
}\limits
\ifx\n@one\empty\else ^{\box\z@}\fi
\ifx\n@two\empty\else _{\box\@ne}\fi}}
\def\tw@@^#1{\@ifnextchar_ {\two@^{#1}}{\two@^{#1}_{}}}

\def\tofr@^#1_#2{\def\n@one{#1}\def\n@two{#2}\mathrel{\setw@dth{#1}{#2}
\mathop{\vcenter{\hbox to \w@dth{\rightarrowfill}\kern-1.7ex
\hbox to \w@dth{\leftarrowfill}}%
}\limits
\ifx\n@one\empty\else ^{\box\z@}\fi
\ifx\n@two\empty\else _{\box\@ne}\fi}}
\def\t@fr@^#1{\@ifnextchar_ {\tofr@^{#1}}{\tofr@^{#1}_{}}}

\newdimen\W@dth
\def\setW@dth#1#2{\setbox\z@\hbox{$#1$}\W@dth=\wd\z@
\setbox\@ne\hbox{$#2$}\ifnum\W@dth<\wd\@ne \W@dth=\wd\@ne \fi
\advance\W@dth by 1.2em}

\def\T@^#1_#2{\allowbreak\def\N@one{#1}\def\N@two{#2}\mathrel
{\setW@dth{#1}{#2}
\mathop{\hbox to \W@dth{\rightarrowfill}}\limits
\ifx\N@one\empty\else ^{\box\z@}\fi
\ifx\N@two\empty\else _{\box\@ne}\fi}}
\def\T@@^#1{\@ifnextchar_ {\T@^{#1}}{\T@^{#1}_{}}}

\def\T@left^#1_#2{\def\N@one{#1}\def\N@two{#2}\mathrel{\setW@dth{#1}{#2}
\mathop{\hbox to \W@dth{\leftarrowfill}}\limits
\ifx\N@one\empty\else ^{\box\z@}\fi
\ifx\N@two\empty\else _{\box\@ne}\fi}}
\def\T@@left^#1{\@ifnextchar_ {\T@left^{#1}}{\T@left^{#1}_{}}}

\def\Tofr@^#1_#2{\def\N@one{#1}\def\N@two{#2}\mathrel{\setW@dth{#1}{#2}
\mathop{\vcenter{\hbox to \W@dth{\rightarrowfill}\kern-1.7ex
\hbox to \W@dth{\leftarrowfill}}%
}\limits
\ifx\N@one\empty\else ^{\box\z@}\fi
\ifx\N@two\empty\else _{\box\@ne}\fi}}
\def\T@fr@^#1{\@ifnextchar_ {\Tofr@^{#1}}{\Tofr@^{#1}_{}}}

\def\Two@^#1_#2{\def\N@one{#1}\def\N@two{#2}\mathrel{\setW@dth{#1}{#2}
\mathop{\vcenter{\hbox to \W@dth{\rightarrowfill}\kern-1.7ex
\hbox to \W@dth{\rightarrowfill}}%
}\limits
\ifx\N@one\empty\else ^{\box\z@}\fi
\ifx\N@two\empty\else _{\box\@ne}\fi}}
\def\Tw@@^#1{\@ifnextchar_ {\Two@^{#1}}{\Two@^{#1}_{}}}

\def\to{\@ifnextchar^ {\t@@}{\t@@^{}}}
\def\from{\@ifnextchar^ {\t@@left}{\t@@left^{}}}
\def\tofro{\@ifnextchar^ {\t@fr@}{\t@fr@^{}}}
\def\To{\@ifnextchar^ {\T@@}{\T@@^{}}}
\def\From{\@ifnextchar^ {\T@@left}{\T@@left^{}}}
\def\Two{\@ifnextchar^ {\Tw@@}{\Tw@@^{}}}
\def\Tofro{\@ifnextchar^ {\T@fr@}{\T@fr@^{}}}

\makeatother

\def\P{\mathcal P}

\def\Com{{\operatorname{Com}}}
\def\Ass{{\operatorname{Ass}}}
\def\Lie{{\operatorname{Lie}}}
\def\Der{{\operatorname{Der}}}

\def\Mod{\mathsf{MOD}}
\def\Alg{\mathsf{ALG}}

\def\Free{\mathsf{Free}}

\def\0{\mathbf 0}
\def\N{\mathbb N}

\def\Lieg{{\mathfrak{g}}}

\title{The Rosický Tangent Categories of Algebras over an Operad}
\author{Sacha Ikonicoff, Marcello Lanfranchi, Jean-Simon Pacaud Lemay}

\begin{document}

\maketitle

\begin{abstract}
Tangent categories provide a categorical axiomatization of the tangent bundle. There are many interesting examples and applications of tangent categories in a variety of areas such as differential geometry, algebraic geometry, algebra, and even computer science. The purpose of this paper is to expand the theory of tangent categories in a new direction: the theory of operads. The main result of this paper is that both the category of algebras of an operad and its opposite category are tangent categories. The tangent bundle for the category of algebras is given by the semi-direct product, while the tangent bundle for the opposite category of algebras is constructed using the module of K\"ahler differentials, and these tangent bundles are in fact adjoints of one another. To prove these results, we first prove that the category of algebras of a coCartesian differential monad is a tangent category. We then show that the monad associated to any operad is a coCartesian differential monad. This also implies that we can construct Cartesian differential categories from operads. Therefore, operads provide a bountiful source of examples of tangent categories and Cartesian differential categories, which both recaptures previously known examples and also yield new interesting examples. We also discuss how certain basic tangent category notions recapture well-known concepts in the theory of operads.
\end{abstract}


\subsection*{Acknowledgements}
The authors would first like to thank Martin Frankland for providing a very useful result about adjoint tangent structure. The authors would also like to thank Martin Frankland (again), Geoff Cruttwell and Dorette Pronk for many very useful discussions. For this research, the first named author was financially supported by a PIMS--CNRS Postdoctoral Fellowship and by the Fields Institute, and the third named author was financially supported by a JSPS Postdoctoral Fellowship, Award \#: P21746 and an ARC DECRA, Award \#: DE230100303.

\newpage
\tableofcontents

\section{Introduction}

Tangent categories provide a categorical description of the tangent bundle, one of the fundamental structures of differential geometry. Tangent categories were introduced by Rosick\'y in \cite{RosickyTangentCats}, and then, thirty years later, generalized and further developed by Cockett and Cruttwell in \cite{cockett2014differential}. Briefly, a tangent category (Definition \ref{definition:tangent-category}) is a category which comes equipped with an endofunctor $\mathsf{T}$ where for every object $A$, $\mathsf{T}(A)$ is interpreted as a generalized version of a tangent bundle over $A$. Furthermore, the functor $\mathsf{T}$ also comes equipped with five (or six) natural transformations that satisfy various axioms that capture the basic properties of the classical tangent bundle over a smooth manifold including natural projection, being a vector bundle, local triviality, linearity of the derivative, etc. 

Nowadays, the theory of tangent categories is a well-established area of research and fits into the broader story of differential categories. As expected, the theory of tangent categories and its applications are fundamentally linked to differential geometry. Many important concepts from differential geometry can be generalized in a tangent category, including vector fields \cite{cockett2015jacobi}, Euclidean spaces \cite{blute2009cartesian}, vector bundles \cite{cockettCruttwellDiffBundles,MacAdamVectorBundles}, connections \cite{cockettCruttwellConnections}, differential equations \cite{cockett2021differential}, differential forms and de Rham cohomology \cite{CruttwellLucychynCohomology}. 

Most well-known examples of tangent categories are based on differential geometry and synthetic differential geometry \cite{GARNER2018668}. Indeed, the canonical example of a tangent category is the category of smooth manifolds, where the tangent structure is induced by the classical tangent bundle of a smooth manifold. Recently, there has been an upswing on new interesting examples and novel applications of tangent categories beyond differential geometry such as in commutative algebra and algebraic geometry \cite{cruttwellLemay:AlgebraicGeometry}, and even in computer science, in particular in relation to differential linear logic \cite{cockett_et_al:LIPIcs:2020:11660} and the differential lambda calculus \cite{manzyuk2012tangent}. The objective of this paper is to further expand the theory and applications of tangent categories into a new direction: the theory of operads. 

Operad theory is a firmly established field of mathematics. Operads first originated as a useful tool in algebraic topology back in the late 1960s/early 1970s, in particular, to characterize iterated loop spaces \cite{may72}. The theory of operads went through a reinvention period in the 1990s, shifting from a topological point of view to a more algebraic one. Since then, operads have found applications in many mathematical domains, including homological algebra, category and higher category theory, combinatorics, and algebraic deformation theory. Operads have even found applications outside the realm of mathematics, where they appear notably in some aspects of mathematical physics, computer sciences, and biology. For an overview of applications of operads and a more detailed introduction, we invite the readers to see \cite{LV,Fresse-ModulesOperads,kriz95}. 

Naively, an operad $\P$ (Section \ref{sec:operads}) is a device that encodes a type of algebra structure on modules over a ring $R$ -- these operads are sometimes also referred to as algebraic operads. Every operad $\P$ has a canonical monad associated to it, and the algebras of said monad are what we call the algebras of the operad $\P$, or simply just the $\P$-algebras. Together, the algebras of an operad $\P$ form a category $\mathsf{ALG}_\P$. For many sorts of algebraic objects, there is an operad whose algebras are precisely those algebraic objects. For example, there is an operad $\Com$ where the $\Com$-algebras are the commutative $R$-algebras and so $\Alg_{\Com}$ is the category of commutative $R$-algebras (Example \ref{example:operadCom}). 

The main objective of this paper is to show that both $\mathsf{ALG}_\P$ and the opposite category $\mathsf{ALG}^{op}_\P$ are tangent categories (Theorem \ref{theorem:footT-semi-direct-product} and Theorem \ref{theorem:operadic-tangent-categories}), whose tangent bundle functors are adjoints to one another (Lemma  \ref{lem:operad-adj}). This is a generalization of the fact that $\Alg_{\Com}$ and $\Alg^{op}_{\Com}$ are both well-known examples of tangent categories, see for example \cite{cruttwellLemay:AlgebraicGeometry} for full details. Briefly, the tangent bundle of a commutative $R$-algebra $A$ in $\Alg_{\Com}$ is given by the algebra of dual numbers over $A$, by which we mean $A[\epsilon] := A[x]/(x^2)$ (Example \ref{example:tangent-structure-over-Alg-Com}):
\[\mathsf{T}(A) := A[\epsilon] \]
In fact, $\Alg_{\Com}$ was one of the main examples in Rosick\'y's original paper \cite[Example 2]{RosickyTangentCats}. On the other hand, the tangent bundle of a commutative algebra $A$ in $\Alg^{op}_{\Com}$ is given by the free symmetric algebra over $A$ of its module of K\"ahler differentials (Example \ref{example:otc-Com}): 
\[\mathsf{T}^\circ(A) := \mathsf{Sym}_A \left( \Omega_A \right)\]
The tangent category $\Alg^{op}_{\Com}$ is closely related to algebraic geometry, as explained in \cite{cruttwellLemay:AlgebraicGeometry}. Indeed, it is famously known that $\Alg^{op}_{\Com}$, the opposite category of commutative algebras, is equivalent to the category of affine schemes over $R$, the building blocks in algebraic geometry. Furthermore, Grothendieck himself calls $\mathsf{T}^\circ(A)$ the ``fibré tangent'' (french for tangent bundle) of $A$ \cite[Definition 16.5.12.I]{grothendieck1966elements}, while Jubin calls it the tangent algebra of $A$ \cite[Section 2.6]{jubin2014tangent}. These two tangent bundles are related since $\mathsf{T}$ and $\mathsf{T}^\circ$, viewed as endofunctors on $\Alg_{\Com}$, are mutual adjoints. This means that $\Alg_{\Com}$ is a tangent category with adjoint tangent structure (Section \ref{subsec:adjoint}). 

For any operad $\P$, by generalizing the constructions for commutative algebras, we are able to obtain tangent category structures for both $\mathsf{ALG}_\P$ and $\mathsf{ALG}^{op}_\P$. For a $\P$-algebra $A$, its tangent bundle in $\mathsf{ALG}_\P$ is given by the semi-direct product with itself (Section \ref{subsection:otc-concrete-description}):
\[ \mathsf{T}(A) := A \ltimes A \]
where the semi-direct product $\ltimes$ is a generalization of the dual numbers construction for $\P$-algebras \cite[Section 12.3.2]{LV}. On the other hand, the tangent bundle of a $\P$-algebra $A$ in $\mathsf{ALG}^{op}_\P$ requires more setup. Firstly, there is a notion of modules over a $\P$-algebra $A$, referred to as $A$-modules \cite[Section 12.3.1]{LV}, and in particular, a generalization of a module of K\"ahler differentials over $A$ \cite[Section 12.3.8]{LV}, denoted $\Omega_A$. Secondly, there is also a notion of $\P$-algebras over $A$, called simply $A$-algebras, and for any $A$-module $M$, there is a free $A$-algebra over $M$ \cite[Lemma~5.2]{Ginzburg-NCSymplecticGeometry}, denoted $\Free_A(M)$. As such, the tangent bundle of a $\P$-algebra $A$ in $\mathsf{ALG}^{op}_\P$ is defined as the free $A$-algebra of its module of K\"ahler differentials (Section \ref{subsec:adjoint-tan-operad}): 
\[ \mathsf{T}^\circ(A) := \Free_A(\Omega_A) \]
Furthermore, we also have that $\mathsf{ALG}_\P\left( \Free_A(\Omega_A) , A^\prime \right) \cong \mathsf{ALG}_\P(A, A^\prime \ltimes A^\prime)$. Therefore, the tangent bundles $\mathsf{T}$ and $\mathsf{T}^\circ$ are mutual adjoints, as desired. In particular, following the discussions of \cite{cruttwellLemay:AlgebraicGeometry}, we may interpret $\mathsf{ALG}^{op}_\P$ as a tangent category model of algebraic geometry relative to the operad $\P$. It is worth mentioning that, while the tangent bundle $\mathsf{T}$ is mostly the same for each operad, the adjoint tangent bundle $\mathsf{T}^\circ$ can vary quite drastically from operad to operad.

As a consequence, operads provide a large source of examples of tangent categories, including both previously known ones and many new ones, some of which may be exotic or very unexpected. Of course, by taking $\P = \Com$, we recapture precisely the tangent category of commutative algebras and the tangent category of affine schemes. However, we can take other operads $\P$ to obtain models of tangent categories that have not been previously considered. For example, we may take the operad $\Ass$, where $\Alg_\Ass$ is instead the category of (associative and unital) algebras (Example \ref{example:operadAss}). Therefore, $\Alg_\Ass$ is a tangent category model of non-commutative algebras (Example \ref{example:tangent-structure-over-Alg-Ass}), while $\Alg^{op}_\Ass$ is a tangent category model of non-commutative algebraic geometry (Example \ref{example:otc-Ass}). In particular, Ginzburg calls the tangent bundle $\mathsf{T}^\circ(A)$ the ``space of noncommutative differential forms of $A$'' \cite[Definition 10.2.3]{Ginzburg:notes-on-NCGeometry}. We can also take the operad $\Lie$, where $\Alg_\Lie$ is the category of Lie algebras (Example \ref{example:operadLie}). In this case, we get the surprising new examples of tangent categories of Lie algebras (Example \ref{example:tangent-structure-over-Alg-Lie} and Example \ref{example:otc-Lie}). Since operad theory has such a large range of relations to a variety of domains, we expect that the results of this paper will not only lead to a multitude of new examples of tangent categories but also greatly expand the reaches of the theory of tangent categories and its applications to new areas. 

We also discuss two notions coming from differential geometry in the setting of our tangent categories. The first is vector fields (Definition \ref{def:vecfield}), which, as the name suggests, generalizes the classical notion of vector fields from differential geometry. We will show that vector fields in both $\mathsf{ALG}_\P$ and $\mathsf{ALG}^{op}_\P$ correspond to derivations in the operadic sense \cite[Section 12.3.7]{LV}, which generalize the notion of algebraic derivations (Section \ref{sec:vf-operad}). The second is differential objects (Definition \ref{def:diffobj}), which provide analogues of Euclidean spaces in a tangent category. In particular, the tangent bundle of a differential object $A$ is just the product of $A$ with itself, $\mathsf{T}(A) \cong A \times A$. In $\mathsf{ALG}_\P$, this means that a differential object is a $\P$-algebra whose $\P$-algebra structure is essentially everywhere zero (Proposition \ref{prop:diffobj-alg}) -- which for some operads means that the only differential object is the zero algebra. On the other hand, the differential objects in $\mathsf{ALG}^{op}_\P$ are precisely the modules (in the operadic sense) over the $\P$-algebra $\P(0)$ (Theorem \ref{thm:diffobj-operad-op}). In future work, it would be interesting to study other tangent category notions in $\mathsf{ALG}_\P$ and $\mathsf{ALG}^{op}_\P$ such as connections, differential forms and their cohomology, differential bundles, differential equations, etc. 

To prove our main results, we will in fact prove a more general result, allowing us to avoid checking all the tangent category axioms for $\mathsf{ALG}_\P$ and $\mathsf{ALG}^{op}_\P$. Indeed, $\mathsf{ALG}_\P$ can also be described as the Eilenberg--Moore category of the monad associated to the operad $\P$. We investigate conditions under which the Eilenberg--Moore category of a monad is a tangent category (Section \ref{sec:ccdm}). It turns out that the solution to this problem is linked to a special class of tangent categories called Cartesian differential categories \cite{blute2009cartesian}. Briefly, a Cartesian differential category can be defined as a category with finite products and equipped with a differential combinator $\mathsf{D}$ which takes a map $f$ and produces its derivative $\mathsf{D}[f]$ (Definition \ref{def:CDC}). Every Cartesian differential category is a tangent category, where the tangent bundle is built using the differential combinator \cite[Proposition 4.7]{cockett2014differential}. In \cite{ikonicoff2021cartesian}, the first and third named authors introduced the notion of a coCartesian differential monad (Definition \ref{defi:cCDM}), which is precisely the kind of monad for which the opposite category of its Kleisli category is a Cartesian differential category. In \cite{ikonicoff2021cartesian}, the problem of identifying the structure of the Eilenberg--Moore category of a coCartesian differential monad was left open. In this paper, we will show that this category is a tangent category (Theorem \ref{thm:S-tan}), and that under mild assumptions, its opposite category is also a tangent category (Section \ref{subsection:adjoint-tangent-structure-cCD-monads}). 

We will then prove that the monad associated to any operad $\P$ is always a coCartesian differential monad (Theorem \ref{theorem:opdiffmonad}). This will immediately imply that $\mathsf{ALG}_\P$ is a tangent category (Theorem \ref{theorem:footT-semi-direct-product}). After some extra work, we will then also obtain that $\mathsf{ALG}^{op}_\P$ is a tangent category (Theorem \ref{theorem:operadic-tangent-categories}). Furthermore, since the monad associated to an operad $\P$ is a coCartesian differential monad, it also follows that the opposite category of its Kleisli category, $\mathsf{KL}^{op}_\P$, is a Cartesian differential category (Section \ref{subsec:CDC-operad}). Intuitively, the maps of $\mathsf{KL}^{op}_\P$ can be interpreted as special kinds of smooth functions, which we call $\P$-polynomials. In particular, the subcategory of finite dimensional $R$-module of $\mathsf{KL}^{op}_\P$ is the Lawvere theory for $\P$-polynomials, and is again a Cartesian differential category. Therefore, operads also give a source of examples of Cartesian differential categories, again both recapturing known examples, like classical polynomial differentiation (Example \ref{example:CDCCom}), and providing new unexpected examples, like the differentiation of Lie bracket polynomials (Example \ref{example:CDCLie}). Moreover, it is known that the subcategory of differential objects of a tangent category is a Cartesian differential category \cite[Theorem 4.11]{cockett2014differential}. We will show that every free $\P$-algebra is a differential object in $\mathsf{ALG}^{op}_\P$ (Lemma \ref{lem:freeP-diffobj}), and thus the Cartesian differential category $\mathsf{KL}^{op}_\P$ embeds into the Cartesian differential category of differential objects of $\mathsf{ALG}^{op}_\P$.

It is our hope that this paper is but the exciting start of a new unified theory for geometry for algebra structures obtained by applying the theory of tangent categories and Cartesian differential categories to the notion of operads. 

\textbf{Outline:} Section \ref{section:tancats} is a background section on the basics of tangent categories, where we introduce most of the terminology, notation, and constructions that we will use throughout the paper. Section \ref{sec:ccdm} is a general theory section on coCartesian differential monads, the results of this section are key to providing the main results of the following section. Section \ref{sec:operads} is the main section of this paper, where we study the tangent categories of algebras of an operad, and also discuss the Cartesian differential categories induced by an operad.  We conclude this paper in Section \ref{sec:future}, where we discuss future work that we hope to pursue that builds on the ideas presented in this paper. 

\textbf{Conventions:} We assume the reader is familiar with the basic notions of category theory such as categories, opposite categories, functors, natural transformations, and (co)limits like (co)products, pullbacks, pushouts, etc. In some cases, if only to introduce notation, we recall some of these concepts. In an arbitrary category $\mathbb{X}$, we denote identity maps as ${1_A: A \to A}$, and we use the classical notation for composition, $g \circ f$, as opposed to diagrammatic order which was used in other papers on tangent categories such as in \cite{cockett2014differential}. Finally, homsets in a category $\mathbb X$ of morphisms from an object $A$ to an object $B$ will be denoted by $\mathbb X(A,B)$. 

\section{Tangent Categories}
\label{section:tancats}

In this background section, we review tangent categories and their basic theory including adjoint tangent structure (where we also prove a new useful lemma), vector fields, differential objects, and Cartesian differential categories. 

\subsection{Tangent Categories}

We begin by recalling the necessary structural maps for a tangent category. We do not recall the full definition here and refer readers to see the full definition of a tangent category, including the axioms expressed as commutative diagrams, and intuitions in \cite{cockett2014differential,GARNER2018668}. The key difference between Rosick\'y's original definition of a tangent category \cite[Section 2]{RosickyTangentCats} and Cockett and Cruttwell's definition \cite{cockett2014differential} is that the former assumes an Abelian group structure on the fibres of the tangent bundle, while the latter generalizes to only a commutative monoid structure. As such, Rosick\'y's definition includes one extra natural transformation capturing the negatives in the tangent bundle. We will adopt the terminology used in \cite{cruttwellLemay:AlgebraicGeometry}, where a tangent structure with negatives will be called a \textit{Rosick\'y} tangent structure\footnote{We also choose this naming convention to clearly separate from the terminology used in \cite{harpazetal19b}, which studies a different notion of tangent categories for operads.}.  

\begin{definition}\label{definition:tangent-category} \cite[Definition 2.3 and Section 3.3]{cockett2014differential} A \textbf{(Rosick\'y) tangent structure} on a category $\mathbb{X}$ is a sextuple $\mathbb{T} := (\mathsf{T}, p, s, z, l, c)$ (resp. a septuple ${\mathbb{T} := (\mathsf{T}, p, s, z, l, c, n)}$) consisting of: 
\begin{enumerate}[{\em (i)}]
\item An endofunctor $\mathsf{T}: \mathbb{X} \to  \mathbb{X}$, called the \textbf{tangent bundle functor};
\item A natural transformation $p_A: \mathsf{T}(A) \to A$, called the \textbf{projection}, such that for each $n\in \mathbb{N}$, the $n$-fold pullback\footnote{By convention, $\mathsf{T}_0(A) = A$ and $\mathsf{T}_1(A) = \mathsf{T}(A)$} of $p_A$ exists, denoted as $\mathsf{T}_n(A)$ with projections $q_j: \mathsf{T}_n(A) \to \mathsf{T}(A)$, and such that for all $m \in \mathbb{N}$, $\mathsf{T}^m:=\mathsf{T}\circ\dots\circ\mathsf{T}$ preserves these pullbacks, that is, $\mathsf{T}^m( \mathsf{T}_n(A))$ is the $n$-fold pullback of $\mathsf{T}^m(p_A)$ with projections $\mathsf{T}^m(q_j)$;  
\item A natural transformation\footnote{Note that by the universal property of the pullback, it follows that we can define functors $\mathsf{T}_n: \mathbb{X} \to \mathbb{X}$.} $s_A: \mathsf{T}_2(A) \to \mathsf{T}(A)$, called the \textbf{sum};
\item A natural transformation $z_A: A \to \mathsf{T}(A)$, called the \textbf{zero map};
\item A natural transformation $l_A: \mathsf{T}(A) \to \mathsf{T}^2(A)$, called the \textbf{vertical lift};
\item A natural transformation $c_A: \mathsf{T}^2(A) \to \mathsf{T}^2(A)$, called the \textbf{canonical flip}; 
\item (And if Rosick\'y, a natural transformation ${n_A: \mathsf{T}(A) \to \mathsf{T}(A)}$, called the \textbf{negative map};)
\end{enumerate}
such that the equalities in \cite[Definition 2.3]{cockett2014differential} (and if Rosick\'y, also \cite[Definition 3.3]{cockett2014differential}) are satisfied. A \textbf{(Rosick\'y) tangent category} is a pair $(\mathbb{X}, \mathbb{T})$ consisting of a category $\mathbb{X}$ equipped with a (Rosick\'y) tangent structure $\mathbb{T}$ on $\mathbb{X}$.
\end{definition}

We can also ask our tangent categories to have finite products in such a way that the tangent bundle of a product is naturally isomorphic to the product of the tangent bundles, and that the tangent bundle of the terminal object is the terminal object. For a category with finite products, we denote $n$-ary products by $A_1 \times \hdots \times A_n$ with projections ${\pi_j: A_1 \times \hdots \times A_n \to A_j}$ and $\langle -, \hdots, - \rangle$ for the pairing operation, the terminal object as $\ast$, and for every object $A$, the unique map from $A$ to $\ast$ is denoted by $t_A: A \to \ast$.

\begin{definition}~\cite[Definition 2.8]{cockett2014differential} A \textbf{Cartesian (Rosick\'y) tangent category} is a (Rosick\'y) tangent category $(\mathbb{X}, \mathbb{T})$ such that $\mathbb{X}$ has finite products and the canonical maps:
\begin{align*}
\left \langle \mathsf{T}(\pi_1), \hdots, \mathsf{T}(\pi_n) \right \rangle: \mathsf{T}(A_1 \times \hdots \times A_n) \to \mathsf{T}(A_1) \times \hdots \times \mathsf{T}(A_n) && t_{\mathsf{T}(\ast)}: \mathsf{T}(\ast) \to \ast
\end{align*}
are isomorphisms: $\mathsf{T}(A_1 \times \hdots \times A_n) \cong \mathsf{T}(A_1) \times \hdots \times \mathsf{T}(A_n)$ and $\mathsf{T}(\ast) \cong \ast$. 
\end{definition}

See \cite[Example 2.2]{cockettCruttwellDiffBundles} and \cite[Example 2]{GARNER2018668} for lists of examples of tangent categories. Arguably the canonical example of a tangent category is the category of smooth manifolds, where the tangent structure is induced by the classical tangent bundle. This example provides a direct link between tangent categories and differential geometry. In Lemma \ref{lemma:biproduct}, we will review how every additive category is a tangent category, and in Section \ref{sec:CDC-diffobj}, we will review an important subclass of tangent categories: Cartesian differential categories. Furthermore, the main objective of this paper is to show that the category of algebras over an operad and its opposite category are both tangent categories. As such, more examples of tangent categories can be found in Section \ref{sec:operads}, including the (opposite) categories of algebras, commutative algebras, and Lie algebras.

\subsection{Adjoint Tangent Structure}\label{subsec:adjoint}

In \cite{cockett2014differential}, Cockett and Cruttwell introduce an important construction for this paper, called the dual tangent structure - not to be confused with the notion of cotangent structure. This construction allows one to build a tangent structure on the opposite category of a tangent category. This is possible when the tangent bundle functor admits a left adjoint, and in this case, said left adjoint becomes a tangent bundle functor on the opposite category. To avoid confusion, we will refer to this construction as the adjoint tangent structure. In particular, in Section \ref{subsec:adjoint-tan-operad} we will show that the category of algebras over an operad always has adjoint tangent structure, and therefore the opposite category of algebras over an operad is a tangent category. 

Recall that an adjunction between two categories $\mathbb{X}$ and $\mathbb{Y}$ consists of two functors ${\mathsf{L}: \mathbb{X} \to \mathbb{Y}}$, called the left adjoint, and $\mathsf{R}: \mathbb{Y} \to \mathbb{X}$, called the right adjoint, and two natural transformations $\eta_A: A \to \mathsf{R}\mathsf{L}(A)$, called the unit, and $\varepsilon_B: \mathsf{L}\mathsf{R}(B) \to B$, called the counit, such that ${\varepsilon_{\mathsf{L}(A)} \circ \mathsf{L}(\eta_A) = 1_{\mathsf{L}(A)}}$ and $\mathsf{R}(\varepsilon_A) \circ \eta_{\mathsf{R}(A)} = 1_{\mathsf{R}(A)}$. As a shorthand, we write adjunctions as $(\eta, \varepsilon): \mathsf{L} \dashv \mathsf{R}$. 

\begin{definition} \label{def:adjoint-tangent-structure} A tangent category $(\mathbb{X}, \mathbb{T})$ is said to have \textbf{adjoint tangent structure} if, for every $n \in \mathbb{N}$, the functor $\mathsf{T}_n$ admits a left adjoint $\mathsf{T}^\circ_n$ with unit $\eta(n)_A: A \to \mathsf{T}_n\mathsf{T}^\circ_n(A)$ and counit $\varepsilon(n)_A: \mathsf{T}^\circ_n\mathsf{T}_n(A) \to A$, or again, $(\eta(n), \varepsilon(n)): \mathsf{T}^\circ_n \dashv \mathsf{T}_n$. By convention, we simply denote $\mathsf{T}_1 = \mathsf{T}$, $\mathsf{T}_1^\circ = \mathsf{T}^\circ$, $\eta = \eta(1)$, and $\varepsilon = \varepsilon(1)$, so $(\eta, \varepsilon): \mathsf{T}^\circ \dashv \mathsf{T}$. 
\end{definition}

Using the adjoint tangent structure, we now give a full description of the resulting tangent category on the opposite category. Giving a tangent structure on the (opposite) category $\mathbb{X}^{op}$ corresponds to giving a ``dual tangent structure'' on $\mathbb{X}$, that is, the types of all the natural transformations are reversed.

\begin{thm} \cite[Proposition 5.17]{cockett2014differential}
\label{theorem:dual-tangent-structure}
Let $(\mathbb{X}, \mathbb{T})$ be a tangent category with adjoint tangent structure. Consider: 
\begin{enumerate}[{\em (i)}]
\item The adjoint projection $p^\circ_A: A \to \mathsf{T}^\circ(A)$, defined as:
\[p^\circ_A := p_{\mathsf{T}^\circ(A)} \circ \eta_A\] 
where the $n$-fold pushout of $p^\circ_A$ is $\mathsf{T}^\circ_n(A)$ with injections $q^\circ_j: \mathsf{T}^\circ(A) \to \mathsf{T}^\circ_n(A)$ defined as:
\[q^\circ_j = \varepsilon_{\mathsf{T}^\circ_n(A)} \circ \mathsf{T}^\circ(q_j) \circ \mathsf{T}^\circ(\eta_A)\]
\item The adjoint sum $s^\circ_A: \mathsf{T}^\circ(A) \to \mathsf{T}^\circ_2(A)$, defined as:
\[s^\circ_A := \varepsilon_{\mathsf{T}^\circ_2(A)} \circ \mathsf{T}^\circ(s_{\mathsf{T}^\circ_2(A)}) \circ \mathsf{T}\left(\eta(2)_A\right)\]
\item The adjoint zero map $z^\circ_A: \mathsf{T}^\circ(A) \to A$, defined as: 
\[z^\circ_A := \varepsilon_A \circ \mathsf{T}^\circ(z_A)\]
\item The adjoint vertical lift $l^\circ_A:{\mathsf{T}^\circ}^2(A) \to \mathsf{T}^\circ(A)$, defined as:
\[l^\circ_A := \varepsilon_{\mathsf{T}^\circ(A)} \circ \mathsf{T}^\circ(\varepsilon_{\mathsf{T}\mathsf{T}^\circ(A)}) \circ {\mathsf{T}^\circ}^2(l_{\mathsf{T}^\circ(A)}) \circ {\mathsf{T}^\circ}^2(\eta_A)\]
\item The adjoint canonical flip $c^\circ_A: {\mathsf{T}^\circ}^2(A) \to {\mathsf{T}^\circ}^2(A)$, defined as:
\[c^\circ_A := \varepsilon_{{\mathsf{T}^\circ}^2(A)} \circ \mathsf{T}^\circ(\varepsilon_{\mathsf{T}{\mathsf{T}^\circ}^2(A)}) \circ {\mathsf{T}^\circ}^2(c_{{\mathsf{T}^\circ}^2(A)}) \circ {\mathsf{T}^\circ}^2\mathsf{T}(\eta_{\mathsf{T}^\circ(A)}) \circ  {\mathsf{T}^\circ}^2(\eta_A)\]
\end{enumerate}
Then, $\mathbb{T}^\circ= (\mathsf{T}^\circ,  p^\circ, s^\circ, z^\circ, l^\circ, c^\circ)$ is a tangent structure on $\mathbb{X}^{op}$, and so, $(\mathbb{X}^{op}, \mathbb{T}^\circ)$ is a tangent category. Similarly, if $(\mathbb{X}, \mathbb{T})$ is a Rosick\'y tangent category with adjoint tangent structure, consider: 
\begin{enumerate}[{\em (i)}]
\setcounter{enumi}{5}
\item The adjoint negative map $n^\circ_A:  \mathsf{T}^\circ(A) \to {\mathsf{T}}^\circ(A)$ defined as:
\[n^\circ_A := \varepsilon_{\mathsf{T}^\circ(A)} \circ \mathsf{T}^\circ(n_{\mathsf{T}^\circ(A)}) \circ \eta_{\mathsf{T}^\circ(A)}\]
\end{enumerate}
Then, $\mathbb{T}^\circ= (\mathsf{T}^\circ,  p^\circ, s^\circ, z^\circ, l^\circ, c^\circ, n^\circ)$ is a Rosick\'y tangent structure on $\mathbb{X}^{op}$, and so, $(\mathbb{X}^{op}, \mathbb{T}^\circ)$ is a Rosick\'y tangent category. Furthermore, if $(\mathbb{X}, \mathbb{T})$ is a Cartesian (Rosick\'y) category with adjoint tangent structure and $\mathbb{X}$ also has finite coproducts, then $(\mathbb{X}^{op}, \mathbb{T}^\circ)$ is a Cartesian (Rosick\'y) tangent category. 
\end{thm}

Note that, if a tangent category $(\mathbb{X}, \mathbb{T})$ has adjoint tangent structure, then $(\mathbb{X}^{op}, \mathbb{T}^\circ)$ also has adjoint tangent structure. Applying Theorem \ref{theorem:dual-tangent-structure} on $(\mathbb{X}^{op}, \mathbb{T}^\circ)$ gives back the original tangent structure $(\mathbb{X}, \mathbb{T})$. 

To show that a tangent category has adjoint tangent structure, proving that $\mathsf{T}_n$ admits a left adjoint $\mathsf{T}^\circ_n$ for each $n$ can sometimes be a strenuous task. However, when $\mathsf T$ admits a left adjoint, and when the $n$-fold pushouts of the adjoint projection ${p^\circ_A: A \to \mathsf{T}^\circ(A)}$ exist, then this pushout provides a left adjoint for $\mathsf{T}_n$. This is particularly useful if the starting tangent category is cocomplete. We thank Martin Frankland for stating and proving the following lemma, for which we propose here our own version of the proof: 

\begin{lemma}[Frankland]
\label{lemma:dual-tangent-structure} Let $\mathbb{X}$ be a category, $\mathsf{T}: \mathbb{X} \to  \mathbb{X}$ a functor, and $p_A: \mathsf{T}(A) \to A$ a natural transformation such that for each $n\in \mathbb{N}$, the $n$-fold pullback of $p_A$ exists, denoted by $\mathsf{T}_n(A)$ with projections $q_j: \mathsf{T}_n(A) \to \mathsf{T}(A)$. Suppose that $\mathsf{T}$ has a left adjoint $\mathsf{T}^\circ$ with unit $\eta_A: A \to \mathsf{T}\mathsf{T}^\circ(A)$ and counit $\varepsilon_A: \mathsf{T}^\circ \mathsf{T}(A) \to A$, or again, $(\eta, \varepsilon): \mathsf{T}^\circ \dashv \mathsf{T}$. Furthermore, define the natural transformation $p^\circ_A: A \to \mathsf{T}^\circ(A)$ as $p^\circ_A := p_{\mathsf{T}^\circ(A)} \circ \eta_A$, and suppose that the $n$-fold pushout of $p^\circ_A$ exists, denoted as $\mathsf{T}^\circ_n(A)$ with injections $q^\circ_j: \mathsf{T}^\circ(A) \to \mathsf{T}^\circ_n(A)$. Then, $\mathsf{T}^\circ_n$ is a left adjoint for $\mathsf{T}_n$, where the unit $\eta(n)_A: A \to \mathsf{T}_n\mathsf{T}^\circ_n(A)$ and counit $\varepsilon(n)_A: \mathsf{T}^\circ_n\mathsf{T}_n(A) \to A$ are defined using the universal property of the pullback and pushout, that is, as the unique maps satisfying: 
\begin{align*}
q_j \circ \eta(n)_A = \mathsf{T}(q^\circ_j) \circ \eta_A, && \varepsilon(n)_A \circ q^\circ_j = \varepsilon_A \circ \mathsf{T}^\circ(q_j), && \forall\ 1 \leq j \leq n,
\end{align*}
so $(\eta(n), \varepsilon(n)): \mathsf{T}^\circ_n \dashv \mathsf{T}_n$. 
\end{lemma}
\begin{proof} First, note that, for all $1 \leq i,j \leq n$, we have $p_A \circ \mathsf{T}(q^\circ_j) \circ \eta_A = p_A \circ \mathsf{T}(q^\circ_i) \circ \eta_A$ and $\varepsilon_A \circ \mathsf{T}^\circ(q_j) \circ p^\circ_A = \varepsilon_A \circ \mathsf{T}^\circ(q_i) \circ p^\circ_A$ (which we leave as an exercise for the reader). Therefore, it follows that $\eta(n)_A$ and $\varepsilon(n)_A$ are indeed well-defined. To prove that $\mathsf{T}^\circ_n$ is a left adjoint of $\mathsf{T}_n$, we need to show that the two triangle identities hold. To do so, we will take advantage of the (co)universal property of the pullback and pushout by instead proving that the desired identities hold when precomposed by the pushout injections or postcomposed by the pullback projections. First, note that for all $n\in \mathbb{N}$ and $1 \leq j \leq n$, $q_j: \mathsf{T}_n(A) \to \mathsf{T}(A)$ and $q^\circ_j: \mathsf{T}^\circ(A) \to \mathsf{T}^\circ_n(A)$ are natural transformations. Therefore, we compute: 
\begin{gather*} 
\varepsilon(n)_{\mathsf{T}^\circ_n(A)} \circ \mathsf{T}^\circ_n\left(\eta(n)_A \right) \circ q^\circ_j =  \varepsilon(n)_{\mathsf{T}^\circ_n(A)} \circ q^\circ_j \circ  \mathsf{T}^\circ\left(\eta(n)_A \right) = \varepsilon_{\mathsf{T}^\circ_n(A)} \circ \mathsf{T}^\circ(q_j) \circ  \mathsf{T}^\circ\left(\eta(n)_A \right) \\
= \varepsilon_{\mathsf{T}^\circ_n(A)} \circ \mathsf{T}^\circ\left( q_j \circ \eta(n)_A  \right) =
\varepsilon_{\mathsf{T}^\circ_n(A)} \circ \mathsf{T}^\circ\left( \mathsf{T}(q^\circ_j) \circ \eta_A   \right)  = \varepsilon_{\mathsf{T}^\circ_n(A)} \circ \mathsf{T}^\circ \mathsf{T}(q^\circ_j) \circ \mathsf{T}^\circ(\eta_A) \\
= q^\circ_j \circ \varepsilon_{\mathsf{T}^\circ(A)} \circ \mathsf{T}^\circ(\eta_A) = q^\circ_j \circ 1_{\mathsf{T}^\circ(A)} = q^\circ_j 
\end{gather*}

\begin{gather*}
q_j \circ \mathsf{T}_n(\varepsilon(n)_A) \circ \eta(n)_{\mathsf{T}_n(A)} = \mathsf{T}(\varepsilon(n)_A) \circ q_j \circ \eta(n)_{\mathsf{T}_n(A)} =\mathsf{T}(\varepsilon(n)_A) \circ \mathsf{T}(q^\circ_j) \circ \eta_{\mathsf{T}_n(A)} \\
=\mathsf{T}\left(\varepsilon(n)_A \circ q^\circ_j\right)\circ \eta_{\mathsf{T}_n(A)} =
\mathsf{T}\left(\varepsilon_A \circ \mathsf{T}^\circ(q_j)\right)\circ \eta_{\mathsf{T}_n(A)} =\mathsf{T}(\varepsilon_A) \circ \mathsf{T}\mathsf{T}^\circ(q_j) \circ \eta_{\mathsf{T}_n(A)} \\
= \mathsf{T}(\varepsilon_A) \circ  \eta_{\mathsf{T}(A)} \circ q_j =1_{\mathsf{T}(A)} \circ q_j =q_j 
\end{gather*}
So, for all $1 \leq j \leq n$, $\varepsilon(n)_{\mathsf{T}^\circ_n(A)} \circ \mathsf{T}^\circ_n\left(\eta(n)_A \right) \circ q^\circ_j =q^\circ_j$ and $q_j \circ \mathsf{T}_n(\varepsilon(n)_A) \circ \eta(n)_{\mathsf{T}_n(A)} = q_j$. Therefore, by the couniversal property of the pushout and the universal property of the pullback respectively, it follows that:
\begin{align*}
\varepsilon(n)_{\mathsf{T}^\circ_n(A)} \circ \mathsf{T}^\circ_n\left(\eta(n)_A \right) = 1_{\mathsf{T}^\circ_n(A)} &&  \mathsf{T}_n(\varepsilon(n)_A) \circ \eta(n)_{\mathsf{T}_n(A)} = 1_{\mathsf{T}_n(A)}
\end{align*}
So we conclude that $(\eta(n), \varepsilon(n)): \mathsf{T}^\circ_n \dashv \mathsf{T}_n$.
\end{proof}

\begin{corollary}\label{corollary:dual-tangent-structure} 
Let $(\mathbb{X}, \mathbb{T})$ be a (Rosick\'y) tangent category. Suppose that the tangent bundle functor $\mathsf{T}$ has a left adjoint $\mathsf{T}^\circ$, and that, for all $n$, the $n$-fold pushout of the map $p^\circ_A: A \to \mathsf{T}^\circ(A)$ from Lemma \ref{lemma:dual-tangent-structure} exists. Then, $(\mathbb{X}, \mathbb{T})$ has adjoint tangent structure, and therefore, $(\mathbb{X}^{op}, \mathbb{T}^\circ)$ inherits the structure of a (Rosick\'y) tangent category defined in Theorem \ref{theorem:dual-tangent-structure}.
\end{corollary}

Per the above corollary, if $(\mathbb{X}, \mathbb{T})$ is a Cartesian (Rosick\'y) tangent category whose tangent bundle functor has a left adjoint, and if $\mathbb{X}$ is (finitely) cocomplete, then $(\mathbb{X}, \mathbb{T})$ has adjoint tangent structure, and therefore $(\mathbb{X}^{op}, \mathbb{T}^\circ)$ is a Cartesian (Rosick\'y) tangent category. In Section \ref{subsec:adjoint-tan-operad} we will show that the opposite category of algebras of an operad is a tangent category using this fact. In particular, we will discuss the adjoint tangent structure of algebras, commutative algebras, and Lie algebras. There are also important examples of tangent categories with adjoint tangent structures related to synthetic differential geometry \cite{GARNER2018668}, differential linear logic \cite{cockett_et_al:LIPIcs:2020:11660}, and algebraic geometry \cite{cruttwellLemay:AlgebraicGeometry}. 

\subsection{Vector Fields and their Lie Bracket}

Vector fields are a fundamental concept in differential geometry which, heuristically, correspond to assigning smoothly to each point of a smooth manifold a tangent vector in the tangent space over that point. The notion of a vector field can easily be generalized to tangent categories, and is simply defined as a section of the projection.  

\begin{definition} \label{def:vecfield} \cite[Definition 3.1]{cockett2014differential} In a tangent category $(\mathbb{X}, \mathbb{T})$, a \textbf{vector field} on an object $A$ of $\mathbb{X}$ is a map ${v: A \to \mathsf{T}(A)}$ which is a section of the projection $p_A$, that is, $p_A \circ v = 1_A$. The set of all vector fields on $A$ in $(\mathbb{X}, \mathbb{T})$ is denoted $\mathsf{V}_\mathbb{T}(A)$. 
\end{definition}

In any tangent category, the zero map $z_A: A \to \mathsf{T}(A)$ is a vector field, and the universal property of the lift induces a vector field $\mathcal{L}_A: \mathsf{T}(A) \to \mathsf{T}^2(A)$ \cite[Section 3.1]{cockett2014differential}, which generalizes the Liouville vector field, the canonical vector field on the tangent bundle of a smooth manifold. One can also define a category of vector fields \cite[Definition 2.8]{cockett2021differential}, which turns out to also be a tangent category \cite[Proposition 2.10]{cockett2021differential}. Vector fields can also be used to generalize dynamical systems and solve differential equations in a tangent category \cite{cockett2021differential}. In a Rosick\'y tangent category, the set of vector fields of any object is in fact a Lie algebra:

\begin{proposition}~\cite[Theorem 4.2]{cockett2015jacobi} \label{prop:lie} In a Rosick\'y tangent category $(\mathbb{X}, \mathbb{T})$, for any object $A$, $\mathsf{V}_\mathbb{T}(A)$ is a Lie algebra where in particular the Lie bracket of vector fields is defined as in \cite[Definition 3.14]{cockett2014differential}. 
\end{proposition}

In Section \ref{sec:vf-operad}, we will show that vector fields in the tangent category of algebras over an operad correspond precisely to derivations in the operadic sense. This can be seen as a generalization of the famous result that vector fields of a smooth manifold are in bijective correspondence with derivations of the associated $\mathcal{C}^\infty$-ring of said manifold. 

We turn our attention to vector fields in the setting of an adjoint tangent structure. If $(\mathbb{X}, \mathbb{T})$ is a tangent category with adjoint tangent structure, a vector field in $(\mathbb{X}^{op}, \mathbb{T}^\circ)$ corresponds to a map $v: \mathsf{T}^\circ(A) \to A$ which is a retract of the adjoint projection in $\mathbb{X}$, that is, $v \circ p^\circ_A = 1_A$. It turns out that vector fields over $A$ in $(\mathbb{X}^{op}, \mathbb{T}^\circ)$ correspond precisely to vector fields over $A$ in $(\mathbb{X}, \mathbb{T})$. This comes as no surprise since $\mathsf{T}^\circ$ is a left adjoint of $\mathsf{T}$, and therefore, there is a natural bijective correspondence between maps of type $A \to \mathsf{T}(A)$ and of type $\mathsf{T}^\circ(A) \to A$. Furthermore, this equivalence also preserves the Lie algebra structure. 

\begin{lemma}\label{lem:adjoint-vf} Let $(\mathbb{X}, \mathbb{T})$ be a tangent category with adjoint tangent structure and let $(\mathbb{X}^{op}, \mathbb{T}^\circ)$ be the induced tangent category as defined in Theorem \ref{theorem:dual-tangent-structure}. For any object $A$ of $\mathbb{X}$,
\begin{enumerate}[{\em (i)}]
\item If $v \in \mathsf{V}_{\mathbb{T}}(A)$, define $v^\sharp:  \mathsf{T}^\circ(A) \to A$ by $v^\sharp := \varepsilon_A \circ \mathsf{T}^\circ(v)$. Then $v^\sharp \in \mathsf{V}_{\mathbb{T}^\circ}(A)$. 
\item If $w \in \mathsf{V}_{\mathbb{T}^\circ}(A)$, define $w^\flat: A \to \mathsf{T}(A)$ by $w^\flat := \mathsf{T}(w) \circ \eta_A$. Then $w^\flat \in \mathsf{V}_{\mathbb{T}}(A)$.
\end{enumerate}
Furthermore, these constructions are inverses of each other, that is, ${v^\sharp}^\flat = v$ and ${w^\flat}^\sharp = w$, and therefore, we have $\mathsf{V}_{\mathbb{T}}(A) \cong \mathsf{V}_{\mathbb{T}^\circ}(A)$. If $(\mathbb{X}, \mathbb{T})$ is a Rosick\'y tangent category, then the isomorphism $\mathsf{V}_{\mathbb{T}
}(A) \cong \mathsf{V}_{\mathbb{T}^\circ}(A)$ is also a Lie algebra isomorphism. 
\end{lemma}

\subsection{Cartesian Differential Categories and Differential Objects}\label{sec:CDC-diffobj}

In this section, we review differential objects and Cartesian differential categories. While tangent categories axiomatize the apparatus necessary for differential calculus over smooth manifolds, Cartesian differential categories instead axiomatize differential calculus over Euclidean spaces. In particular, a Cartesian differential category is defined in terms of a differential combinator, which is a generalization of the total derivative operator. Every Cartesian differential category is a Cartesian tangent category, where the tangent bundle functor is constructed using the differential combinator. On the other hand, to extract a Cartesian differential category from a Cartesian tangent category, one must look at a special class of objects: the differential objects. Essentially, differential objects generalize the Euclidean spaces in a tangent category, and the subcategory of differential objects is a Cartesian differential category, where the differential combinator is built using the tangent bundle functor. In fact, this results in an adjunction between the category of Cartesian differential categories and the category of Cartesian tangent categories  \cite[Theorem 4.12]{cockett2014differential}. 

Let us begin with Cartesian differential categories, which were introduced by Blute, Cockett, and Seely in \cite{blute2009cartesian}. The underlying category of a Cartesian differential category is a \textbf{Cartesian left additive category}, which in particular is a category with finite products, such that every homset is a commutative monoid and for which pre-composition only preserves additive structures \cite[Definition 1.2.1]{blute2009cartesian}. Cartesian differential categories are Cartesian left additive categories equipped with a differential combinator, whose axioms include analogues of the chain rule, linearity of the derivative, symmetry of the partial derivatives, etc. We do not provide axioms here and invite interested readers to learn more about Cartesian differential categories in \cite{blute2009cartesian,ikonicoff2021cartesian}.  

\begin{definition}\label{def:CDC}\cite[Definition 2.1.1]{blute2009cartesian} A \text{Cartesian differential category} is a Cartesian left additive category $\mathbb{X}$ equipped with a \textbf{differential combinator} $\mathsf{D}$, which is a family of operators $\mathsf{D}: \mathbb{X}(A,B) \to \mathbb{X}(A \times A,B)$, such that seven axioms \textbf{[CD.1]} to \textbf{[CD.7]} from \cite[Definition 2.3]{ikonicoff2021cartesian} hold. For a map $f: A \to B$, $\mathsf{D}[f]: A \times A \to B$ is called the \textbf{derivative} of $f$.
\end{definition}

Every Cartesian differential category is a Cartesian tangent category, where in particular, the tangent bundle functor is defined on objects as $\mathsf{T}(A) = A \times A$, and on maps as $\mathsf{T}(f) = \langle f \circ \pi_1, \mathsf{D}[f] \rangle$ \cite[Proposition 4.7]{cockett2014differential}. See \cite[Section 2]{ikonicoff2021cartesian} for examples of Cartesian differential categories. The canonical example of a Cartesian differential category is the Lawvere theory of real smooth functions, which provides a direct link to classical multivariable calculus. In Section \ref{sec:operads} we will explain how the opposite category of the Kleisli category of an operad, and a certain Lawvere theory of polynomials of an operad, are both Cartesian differential categories.

Let us now turn our attention to differential objects. Differential objects were first introduced in \cite[Definition 4.8]{cockett2014differential}, however, the definition was later updated in \cite[Definition 3.1]{cockettCruttwellDiffBundles} to include an important compatibility with the vertical lift. 

\begin{definition}\label{def:diffobj}\cite[Definition 3.1]{cockettCruttwellDiffBundles} In a Cartesian tangent category $(\mathbb{X}, \mathbb{T})$, a \textbf{differential object} is a quadruple $(A, \hat{p}, \sigma, \zeta)$ consisting of: 
\begin{enumerate}[{\em (i)}]
\item An object $A$ of $\mathbb{X}$;
\item A map $\hat{p}: \mathsf{T}(A) \to A$, called the \textbf{differential projection};
\item A map $\sigma: A \times A \to A$, called the \textbf{sum};
\item A map $\zeta: \ast \to A$, called the \textbf{zero};
\end{enumerate}
and such that the equalities in \cite[Definition 3.1]{cockettCruttwellDiffBundles} hold.  
Let $\mathsf{DIFF}[(\mathbb{X}, \mathbb{T})]$ be the category of differential objects of $(\mathbb{X}, \mathbb{T})$ and all maps of $\mathbb{X}$ between the underlying objects. 
\end{definition}

In Lemma \ref{lem:diffobjBen}, we will provide an alternative, but equivalent, characterization of differential objects in a Cartesian Rosick\'y tangent category. 

A differential object $A$ should be interpreted as an Euclidean space. One of the axioms of a differential object says that $(A, \sigma, \zeta)$ is a commutative monoid, generalizing the fact a Euclidean space is also a vector space. Another axiom says that $\langle p_A, \hat{p}_A \rangle: \mathsf{T}(A) \to A \times A$ is an isomorphism, so $\mathsf{T}(A) \cong A \times A$. This is an analogue of the fact that the tangent bundle of an Euclidean space is isomorphic to the product of the Euclidean space with itself. The differential projection then arises from the association of Euclidean space with each tangent space, and this association is compatible with the vertical lift embedding. For any Cartesian tangent category $(\mathbb{X}, \mathbb{T})$, $\mathsf{DIFF}[(\mathbb{X}, \mathbb{T})]$ is a Cartesian differential category where for a map $f: A \to B$, its differential is defined as $\mathsf{D}[f] = \hat{p} \circ \mathsf{T}(f) \circ \langle p_A, \hat{p}_A \rangle^{-1}$ \cite[Theorem 4.11]{cockett2014differential}. Conversely, in a Cartesian differential category, every object has a canonical and unique differential object structure \cite[Proposition 4.7]{cockett2014differential}.

Interestingly, differential objects do not usually behave well with respect to the adjoint tangent structure. Indeed, even if a Cartesian tangent category $(\mathbb{X}, \mathbb{T})$ has adjoint tangent structure, a differential object in $(\mathbb{X}, \mathbb{T})$ does not necessarily result in a differential object in $(\mathbb{X}^{op}, \mathbb{T}^\circ)$, and vice-versa. In fact, $(\mathbb{X}^{op}, \mathbb{T}^\circ)$ could have many differential objects while $(\mathbb{X}, \mathbb{T})$ may have no non-trivial ones. This is precisely the case for algebras over an operad. Indeed, in Section \ref{subsec:diffobj-operad}, we will see how the differential objects in the opposite category of algebras over an operad always correspond to modules (in the operadic sense) over the arity-zero part (the units) of the operad. In particular, for the opposite category of (commutative) algebras, the differential objects correspond precisely to modules over the base commutative ring. On the other hand, we will explain why differential objects in the category of algebras of an operad are in a certain sense trivial.

\section{CoCartesian Differential Monads}\label{sec:ccdm}

The main objective of this section is to prove that the category of algebras of a coCartesian differential monad is a tangent category, obtained by lifting the biproduct tangent structure from the base category. This is a crucial result for the story of this paper: in Section \ref{subsection:cCD-monad-for-operads}, we will show that the monad associated to any operad is always a coCartesian differential monad. As such, from this general result, we are able to obtain a tangent structure for the category of algebras of an operad without having to check all the axioms for a tangent category. In this section, we also discuss adjoint tangent structures, vector fields, and differential objects for coCartesian differential monads. By dualizing the results of this section, we also obtain the answer to the question asked in the conclusion of \cite{ikonicoff2021cartesian} regarding the coEilenberg--Moore category of a Cartesian differential comonad: we show that this category is a tangent category.

\subsection{Tangent Monads for Biproducts}

In this section, we discuss the canonical tangent structure induced by biproducts and tangent monads, which are precisely the kind of monads that lift said tangent structure to their categories of algebras. Recall that a \textbf{monad} on a category $\mathbb{X}$ is a triple $(\mathsf{S}, \mu, \eta)$ consisting of a functor ${\mathsf{S}: \mathbb{X} \to \mathbb{X}}$, a natural transformation $\mu_A: \mathsf{S}\mathsf{S}(A) \to \mathsf{S}(A)$, called the \textbf{monad multiplication}, and a natural transformation ${\eta_A: A \to \mathsf{S}(A)}$, called the \textbf{monad unit}, such that the following equalities hold:
\begin{align*}
\mu_A \circ \mathsf{S}(\eta_A) = 1_{\mathsf{S}(A)} = \mu_A \circ \eta_{\mathsf{S}(A)} && \mu_A \circ \mathsf{S}(\mu_A) = \mu_A \circ \mu_{\mathsf{S}(A)}   
\end{align*}
For a monad $(\mathsf{S}, \mu, \eta)$, an \textbf{$\mathsf{S}$-algebra} is a pair $(A,\alpha)$ consisting of an object $A$ and a map $\alpha: \mathsf{S}(A) \to A$ of $\mathbb{X}$, called the \textbf{$\mathsf S$-algebra structure map}, such that the following equalities hold: 
\begin{align*}
\alpha \circ \eta_A = 1_A && \alpha \circ \mu_A = \alpha \circ \mathsf{S}(\alpha)     
\end{align*}
An \textbf{$\mathsf{S}$-algebra morphism} $f: (A, \alpha)\to (B, \beta)$ is a map $f: A\to B$ in $\mathbb{X}$ such that the following equality holds: 
\begin{align*}
f \circ \alpha = \beta \circ \mathsf{S}(f)     
\end{align*}
We denote $\mathsf{ALG}_\mathsf{S}$ the category whose objects are $\mathsf{S}$-algebras and whose maps are $\mathsf{S}$-algebra morphisms. $\mathsf{ALG}_\mathsf{S}$ is also often called the \textbf{Eilenberg--Moore category} of the monad $(\mathsf{S}, \mu, \eta)$. Lastly, recall that the \textbf{free $\mathsf{S}$-algebra} over an object $A$ is the $\mathsf{S}$-algebra $(\mathsf{S}(A), \mu_A)$. 

A \textbf{tangent monad} \cite[Definition 19]{cockett_et_al:LIPIcs:2020:11660} on a tangent category $(\mathbb{X}, \mathbb{T})$ is a monad $(\mathsf{S}, \mu, \eta)$ on $\mathbb{X}$ which also comes equipped with a \textbf{distributive law} \cite[Lemma 1]{johnstone1975adjoint}, which is a natural transformation of type ${\lambda_A: \mathsf{S}\mathsf{T}(A) \to \mathsf{T}\mathsf{S}(A)}$ which is compatible with the tangent structure, in the sense that $(\mathsf{S}, \lambda)$ is a \textbf{tangent morphism} \cite[Definition 2.7]{cockett2014differential}, and is also compatible with the monad structure, in the sense that $\mu$ and $\eta$ are \textbf{tangent transformations} \cite[Definition 4.18]{cockettCruttwellDiffBundles}. In other words, a tangent monad is a monad in the 2-category of tangent categories, tangent morphisms, and tangent transformations \cite[Section 5.9]{blute2018affine}. Similarly, a Rosick\'y tangent monad is a tangent monad on a Rosick\'y tangent category such that the distributive law is also compatible with the negative map in the obvious way. 

By \cite[Proposition 20]{cockett_et_al:LIPIcs:2020:11660}, the category of algebras of a (Rosick\'y) tangent monad is a (Rosick\'y) tangent category, where the tangent bundle on an $\mathsf{S}$-algebra is:
\[ \mathsf{T}(A, \alpha) = (\mathsf{T}(A), \mathsf{T}(\alpha) \circ \lambda_A) \]
Furthermore, the forgetful functor from $\mathsf{S}$-algebras down to the base (Rosick\'y) tangent category preserves the (Rosick\'y) tangent structure strictly. In other words, the forgetful functor is a \emph{strict} tangent morphism (i.e. a tangent morphism whose distributive law is the identity \cite[Definition 5.2]{blute2018affine}). Therefore, we say that tangent monads ``lift'' the (Rosick\'y) tangent structure of the base (Rosick\'y) tangent category to the category of algebras. The finite products in the category of $\mathsf{S}$-algebras are also ``lifted'' from the base category. Hence, for a (Rosick\'y) tangent monad $\mathsf{S}$ on a Cartesian (Rosick\'y) tangent category, the category of $\mathsf{S}$-algebras will also be a Cartesian (Rosick\'y) tangent category, such that the forgetful functor preserves the Cartesian (Rosick\'y) tangent structure strictly. 

In this paper, we are interested in the specific case of tangent monads on categories with finite biproducts, and will therefore give the definition of a tangent monad in this setting. To do so, we must first review the tangent structure induced by biproducts. By a \textbf{semi-additive category}, we mean a category with finite biproducts. Alternatively, recall that a semi-additive category can be described as a category with finite products which is enriched over commutative monoids. In this setting, each homset is a commutative monoid, so we can sum parallel maps together $f+g$ and we have zero maps $0$, and composition preserves this additive structure. Keeping this in mind, we will use product notation $\times$ for biproducts rather than direct sum notation $\oplus$, as the tangent structure is more intuitive from the product perspective. By an \textbf{additive category}, we mean a semi-additive category that is also enriched over Abelian groups, that is, each homset is furthermore an Abelian group, and so, each map $f$ admits a negative $-f$. 

Let us now describe in full detail the canonical tangent structure for (semi-)additive categories. It turns out that this tangent structure in fact arises from a Cartesian differential structure. Indeed, every semi-additive category is canonically a Cartesian differential category where the differential combinator is defined as $\mathsf{D}[f] = f \circ \pi_2$ \cite[Example 2.5]{ikonicoff2021cartesian}. Thus the following tangent structure is obtained by applying \cite[Proposition 4.7]{cockett2014differential} (which we reviewed briefly in Section \ref{sec:CDC-diffobj} above) on a semi-additive category with its canonical Cartesian differential structure. 

\begin{lemma}\cite[Section 5]{cockett_et_al:LIPIcs:2020:11660}\label{lemma:biproduct} Let $\mathbb{X}$ be a semi-additive category. Then define:  
\begin{enumerate}[{\em (i)}]
\item The tangent bundle functor $\mathsf{B}: \mathbb{X} \to \mathbb{X}$ to be the diagonal functor, that is, the functor defined on objects as $\mathsf{B}(A) = A \times A$ and on maps as $\mathsf{B}(f) = f \times f$;
\item The projection $p^\times_A: A \times A \to A$ as the first projection of the product:
\[p^\times_A = \pi_1\] 
and where the $n$-fold pullback of $p^\times_A$ is $\mathsf{B}_n(A) := \prod\limits^{n+1}_{i=1} A$ and where the $j$-th pullback projection $q^\times_j: \prod\limits^{n+1}_{i=1} A \to A \times A$ projects out the first and $j+1$-th term:
\[q^\times_j := \langle \pi_1, \pi_{j+1} \rangle\] 
\item The sum $s^\times_A: A \times A \times A \to A \times A$ as the sum of the last two components:
\[s^\times_A := \langle \pi_1, \pi_2 + \pi_3 \rangle\] 
\item The zero map $z^\times_A: A \to A \times A$ as the injection into the first component:
\[z^\times_A := \langle 1_A, 0 \rangle\]
\item The vertical lift $l^\times_A: A \times A \to A \times A \times A \times A$ as the injection of the first component in the first component and the second component in the fourth component:
\[l^\times_A = \langle \pi_1, 0, 0, \pi_2 \rangle\]
\item The canonical flip $c^\times_A: A \times A \times A \times A \to A \times A \times A \times A$ as the transposition of the second and third components:
\[c^\times_A := \langle \pi_1, \pi_3, \pi_2, \pi_4 \rangle\] 
\end{enumerate}
Then, $\mathbb{B}= (\mathsf{B},  p^\times, s^\times, z^\times, l^\times, c^\times)$ is a tangent structure on $\mathbb{X}$, and so, $(\mathbb{X}, \mathbb{B})$ is a Cartesian tangent category. Similarly, if $\mathbb{X}$ is an additive category, then define: 
\begin{enumerate}[{\em (i)}]
\setcounter{enumi}{6}
\item The negative map $n^\times_A: A \times A \to A\times A$ as taking the negative of the second component:
\[n^\times_A := \langle \pi_1, -\pi_2 \rangle\] 
\end{enumerate}
Then, $\mathbb{B}= (\mathsf{B},  p^\times, s^\times, z^\times, l^\times, c^\times, n^\times)$ is a Rosick\'y tangent structure on $\mathbb{X}$, and so, $(\mathbb{X}, \mathbb{B})$ is a Cartesian Rosick\'y tangent category. 
\end{lemma}

The following definition of a tangent monad is indeed the same definition as in \cite[Definition 19]{cockett_et_al:LIPIcs:2020:11660}, but in the specific case of the canonical tangent structure on a (semi-)additive category. 
\begin{definition} \cite[Definition 19]{cockett_et_al:LIPIcs:2020:11660} \label{def:tangent-monad} Let $\mathbb{X}$ be a semi-additive category (resp. additive category), and let $(\mathsf{S}, \mu, \eta)$ be a monad on $\mathbb{X}$. A \textbf{$\mathbb{B}$-distributive law} over $(\mathsf{S}, \mu, \eta)$ is a natural transformation $\lambda_A: \mathsf{S}(A \times A) \to \mathsf{S}(A) \times \mathsf{S}(A)$ such that: 
\begin{enumerate}[{\em (i)}]
\item $\lambda$ is a \textbf{distributive law} of the functor $\mathsf{B}$ over the monad $(\mathsf{S}, \mu, \eta)$, that is, the following equalities hold: 
\begin{align*}
\lambda_A \circ \mu_{A \times A} = (\mu_A \times \mu_A) \circ \lambda_{\mathsf{S}(A)} \circ \mathsf{S}(\lambda_A) && \lambda_A \circ \eta_{A \times A} = \eta_A \times \eta_A
\end{align*}
\item $\lambda$ is compatible with the biproduct (Rosick\'y) tangent structure $\mathbb{B}$ in the sense that the following equalities hold: 
\begin{equation*}\begin{gathered}
p^\times_A \circ \lambda_A = \mathsf{S}(p^\times_A) \quad \quad \quad \lambda_A \circ \mathsf{S}\left( z^\times_A \right) = z^\times_{\mathsf{S}(A)} \\
\lambda_A \circ \mathsf{S}\left( s^\times_A \right) = s^\times_{\mathsf{S}(A)} \circ \left \langle \mathsf{S}(\pi_1), \pi_2 \circ \lambda_A \circ \mathsf{S}(q^\times_1),\pi_2 \circ \lambda_A \circ \mathsf{S}(q^\times_2) \right\rangle \\ 
l^\times_{\mathsf{S}(A)} \circ \lambda_A = (\lambda_A \times \lambda_A) \circ \lambda_{A \times A} \circ \mathsf{S}\left( l^\times_A \right)  \\
c^\times_{\mathsf{S}(A)} \circ (\lambda_A \times \lambda_A) \circ \lambda_{A \times A} = (\lambda_A \times \lambda_A) \circ \lambda_{A \times A} \circ \mathsf{S}\left( c^\times_A \right) \\
\left( \text{ and if additive, then also } n^\times_{\mathsf{S}(A)} \circ \lambda_A = \lambda_A \circ \mathsf{S}(n^\times_A) ~ \right)
\end{gathered}\end{equation*}
\end{enumerate}
A \textbf{(Rosick\'y) tangent monad} on $(\mathbb{X}, \mathbb{B})$ is a quadruple $(\mathsf{S}, \mu, \eta, \lambda)$ consisting of a monad $(\mathsf{S}, \mu, \eta)$ on $\mathbb{X}$ and a $\mathbb{B}$-distributive law $\lambda$ over $(\mathsf{S}, \mu, \eta)$.   
\end{definition}

As discussed above, the category of algebras of a (Rosick\'y) tangent monad on a semi-additive (resp. additive) category is a Cartesian (Rosick\'y) tangent category where the tangent structure maps are defined in the same way as in $\mathbb{B}$. However, we stress that, since the biproduct structure does not necessarily ``lift'' to the category of algebras, the resulting tangent structure is not given by Lemma \ref{lemma:biproduct}, even if the underlying tangent structure maps are the same. We will review this in full detail in Theorem \ref{thm:S-tan} below, after taking a closer look at the $\mathbb{B}$-distributive law. Indeed, the type of the $\mathbb{B}$-distributive law $\lambda_A: \mathsf{S}(A \times A) \to \mathsf{S}(A) \times \mathsf{S}(A)$ implies that it is the pairing of two maps of type $\mathsf{S}(A \times A) \to \mathsf{S}(A)$. The axiom $p^\times_A \circ \lambda_A = \mathsf{S}(p^\times_A)$ can be re-written as $\pi_1 \circ \lambda_A = \mathsf{S}(\pi_1)$. Therefore, $\lambda_A$ must be of the form $\lambda_A = \langle \mathsf{S}(\pi_1), \lambda^\prime_A \rangle$ for some map $\lambda^\prime_A = \mathsf{S}(A \times A) \to \mathsf{S}(A)$. In the next section, we will prove that this $\lambda^\prime_A$ is completely determined by a \textbf{differential combinator transformation}, and therefore, that tangent monads on semi-additive categories can be precisely described as coCartesian differential monads.

\subsection{CoCartesian Differential Monads}\label{subsec:cdmonad}

In this section, we show that, on (semi-)additive categories, coCartesian differential monads are precisely tangent monads and therefore, that the category of algebras of a coCartesian differential monad is a Cartesian tangent category. CoCartesian differential monads are the dual of Cartesian differential comonads, as introduced by the first and third named authors in \cite{ikonicoff2021cartesian}. They are precisely the kind of monads on a semi-additive category such that the opposite category of the Kleisli category is a Cartesian differential category \cite[Theorem 3.5]{ikonicoff2021cartesian}. Therefore, a coCartesian differential monad is a monad on a semi-additive category equipped with an additional natural transformation called a differential combinator transformation, which captures the differentiation of maps in the Kleisli category. The axioms of a differential combinator transformation are analogues of the axioms of a differential combinator for Cartesian differential categories. 

\begin{definition} \cite[Example 3.14]{ikonicoff2021cartesian}\label{defi:cCDM} Let $\mathbb{X}$ be a semi-additive category and $(\mathsf{S}, \mu, \eta)$ a monad on $\mathbb{X}$. A \textbf{differential combinator transformation} on $(\mathsf{S}, \mu, \eta)$ is a natural transformation $\partial_A: \mathsf{S}(A) \to \mathsf{S}(A \times A)$ such that the following equalities hold: 
\begin{description}
\item [{\bf [DC.1]}] $\mathsf{S}(\pi_1) \circ \partial_A = 0$
\item [{\bf [DC.2]}] $\mathsf{S}(\langle \pi_1, \pi_2, \pi_2 \rangle) \circ \partial_A = \mathsf{S}(\langle \pi_1, \pi_2, 0 \rangle) \circ \partial_A + \mathsf{S}(\langle \pi_1, 0, \pi_2\rangle) \circ \partial_A$
\item [{\bf [DC.3]}] $\partial_A \circ \eta_A = \eta_{A \times A} \circ \langle 0, 1_A \rangle$
\item [{\bf [DC.4]}] $\partial_A \circ \mu_A = \mu_{A \times A} \circ \mathsf{S}\left( \mathsf{S}(\langle 1_A, 0 \rangle) \circ \pi_1 + \partial_A \circ \pi_2 \right) \circ \partial_{\mathsf{S}(A)}$
\item [{\bf [DC.5]}] $\mathsf{S}(\langle \pi_1, \pi_4 \rangle) \circ \partial_{A \times A} \circ \partial_A = \partial_A$
\item [{\bf [DC.6]}] $\mathsf{S}\left( \left \langle \pi_1, \pi_3, \pi_2, \pi_4 \right \rangle \right)  \circ \partial_{A \times A} \circ \partial_A = \partial_{A \times A} \circ \partial_A$
\end{description}
A \textbf{coCartesian differential monad} on a semi-additive category $\mathbb{X}$ is a quadruple $(\mathsf{S}, \mu, \eta, \partial)$ consisting of a monad $(\mathsf{S}, \mu, \eta)$ on $\mathbb{X}$ and a differential combinator transformation $\partial$ on $(\mathsf{S}, \mu, \eta)$. 
\end{definition}

It turns out that, for a coCartesian differential monad on an additive category, the differential combinator transformation is also compatible with the negative map, and thus, no additional axiom is required.  

\begin{lemma} Let $\mathbb{X}$ be an additive category and let $(\mathsf{S}, \mu, \eta, \partial)$ be a coCartesian differential monad on $\mathbb{X}$. The following equality holds: 

\begin{description}
 \item[\textbf{[DC.N]}] $\mathsf{S}(\langle \pi_1, -\pi_2 \rangle) \circ \partial_A = -\partial_A$  
\end{description}
\end{lemma}
\begin{proof} To prove that $\mathsf{S}(\langle \pi_1, -\pi_2 \rangle) \circ \partial_A = -\partial_A$, it suffices to show that $\partial_A + \mathsf{S}(\langle \pi_1, -\pi_2 \rangle) \circ \partial_A = 0$. First observe that, by {\bf [DC.1]} and {\bf [DC.2]}, it follows that: 
\begin{align*}
\mathsf{S}(\langle \pi_1,0 \rangle) \circ \partial_A = 0
\end{align*}
Also, note that we have the following equalities:
\begin{align*}
\begin{split}
\langle \pi_1, \pi_2 - \pi_3 \rangle \circ \langle \pi_1, \pi_2, 0 \rangle = 1_{A \times A}\\
\langle \pi_1, \pi_2 - \pi_3 \rangle \circ \langle \pi_1, 0, \pi_2 \rangle = \langle \pi_1, - \pi_2 \rangle \\
\langle \pi_1, \pi_2 - \pi_3 \rangle \circ \langle \pi_1, \pi_2, \pi_2 \rangle = \langle \pi_1, 0 \rangle 
\end{split}
\end{align*}
Then, using these identities, and both {\bf [DC.1]} and {\bf [DC.2]}, we compute: 
\begin{gather*}
\partial_A +  \mathsf{S}(\langle \pi_1, -\pi_2 \rangle) \circ \partial_A =  \mathsf{S}\left( \langle \pi_1, \pi_2 - \pi_3 \rangle \right) \circ \left( \mathsf{S}(\langle \pi_1, \pi_2, 0 \rangle) \circ \partial_A + \mathsf{S}(\langle \pi_1, 0, \pi_2\rangle) \circ \partial_A \right) \\
=  \mathsf{S}\left( \langle \pi_1, \pi_2 - \pi_3 \rangle \right) \circ \mathsf{S}(\langle \pi_1, \pi_2, \pi_2 \rangle) \circ \partial_A =  \mathsf{S}(\langle \pi_1,0 \rangle) \circ \partial_A = 0 
\end{gather*}
So we conclude that $\mathsf{S}(\langle \pi_1, -\pi_2 \rangle) \circ \partial_A = -\partial_A$. 
\end{proof}

As mentioned above, the opposite category of the Kleisli category of a coCartesian differential monad is a Cartesian differential category. Recall that for a monad $(\mathsf{S}, \mu, \eta)$ on a category $\mathbb{X}$, its \textbf{Kleisli category} is the category $\mathsf{Kl}_{\mathsf{S}}$ whose objects are the same as $\mathbb{X}$ and where a map from $A$ to $B$ in $\mathsf{Kl}_{\mathsf{S}}$ is a map of type $A \to \mathsf{S}(B)$ in $\mathbb{X}$. In a Cartesian differential category, for a map $f: A \to B$, its derivative will be of type ${\mathsf{D}[f]: A \times A \to B}$. For a coCartesian differential monad $(\mathsf{S}, \mu, \eta, \partial)$, since $\mathsf{Kl}^{op}_{\mathsf{S}}$ is a Cartesian differential category, then for a map $f: A \to \mathsf{S}(B)$ in $\mathbb{X}$, its derivative will be of type $\mathsf{D}[f]: A \to \mathsf{S}(B \times B)$ in $\mathbb{X}$, which is given by post-composing with the differential combinator transformation. We invite interested readers to see \cite{ikonicoff2021cartesian} for full details on the subject. 

\begin{proposition}\cite[Theorem 3.5]{ikonicoff2021cartesian} Let $\mathbb{X}$ be a semi-additive category and let $(\mathsf{S}, \mu, \eta, \partial)$ be a coCartesian differential monad on $\mathbb{X}$. Then, $\mathsf{Kl}^{op}_{\mathsf{S}}$ is a Cartesian differential category, where the differential combinator $\mathsf{D}$, viewed as operators $\mathsf{D}: \mathbb{X}(A, \mathsf{S}(B)) \to \mathbb{X}(A, \mathsf{S}(B \times B))$, is defined as the following composition in $\mathbb{X}$: $\mathsf{D}[f] := \partial_B \circ f$.  
\end{proposition}

In Section \ref{subsec:CDC-operad}, we will also review the notion of a $\mathsf{D}$-linear counit for coCartesian differential monads \cite[Definition 3.8]{ikonicoff2021cartesian}. We elected not to review it in this section since the $\mathsf{D}$-linear counit does not appear to play a role for the tangent structure story (but is important for the Cartesian differential structure).  

We will now prove that every coCartesian differential monad is a tangent monad, where the $\mathbb{B}$-distributive law is constructed using the differential combinator transformation. 

\begin{proposition} \label{prop:cdmonad-tanmonad} Let $(\mathsf{S}, \mu, \eta, \partial)$ be a coCartesian differential monad on a semi-additive category (resp. additive category) $\mathbb{X}$. Define the natural transformation $\lambda_A: \mathsf{S}(A \times A) \to \mathsf{S}(A) \times \mathsf{S}(A)$ as follows:
\begin{align*}
\lambda_A := \left\langle \mathsf{S}(\pi_1), \mathsf{S}\left( \pi_1 + \pi_4 \right) \circ \partial_{A \times A} \right \rangle  
\end{align*}
Then, $(\mathsf{S}, \mu, \eta, \lambda)$ is a (Rosick\'y) tangent monad on $(\mathbb{X}, \mathbb{B})$.
\end{proposition}
\begin{proof} Recall the following useful identities regarding the product's pairing operator: 
\begin{align*}
(f \times g) \circ \langle h, k \rangle \circ j = \left \langle f \circ h \circ j, g \circ k \circ j \right \rangle && \langle f \circ \pi_1, g \circ \pi_2 \rangle = f \times g 
\end{align*}
We will show that $\lambda_A$ satisfies the equalities from Definition \ref{def:tangent-monad}. 

For simplicity and readability, we will omit the subscripts for the natural transformations. 
\begin{enumerate}[{\em (i)}]
\item $(\mu \times \mu) \circ \lambda \circ \mathsf{S}(\lambda) = \lambda \circ \mu$ 

First, observe that, by definition of $\lambda$: 
\begin{align*}
(\pi_1 + \pi_4) \circ (\lambda \times \lambda) = \mathsf{S}(\pi_1) \circ \pi_1 + \mathsf{S}\left( \pi_1 + \pi_4 \right) \circ \partial \circ \pi_2
\end{align*}
Also note that for four copies of $A$, we have that: 
\begin{align*}
\left( \pi_1 + \pi_4 \right) \circ \langle 1_{A \times A}, 0 \rangle = \pi_1 
\end{align*}
Then, using \textbf{[DC.4]}, we compute that: 
\begin{gather*}
(\mu \times \mu) \circ \lambda \circ \mathsf{S}(\lambda) =  \left\langle \mu \circ \mathsf{S}(\pi_1) \circ \mathsf{S}(\lambda), \mu \circ \mathsf{S}\left( \pi_1 + \pi_4 \right) \circ \partial \circ \mathsf{S}(\lambda) \right \rangle \\
=  \left \langle \mu \circ \mathsf{S}\mathsf{S}(\pi_1), \mu \circ \mathsf{S}\left( \mathsf{S}(\pi_1) \circ \pi_1 + \mathsf{S}\left( \pi_1 + \pi_4 \right) \circ \partial \circ \pi_2  \right) \circ \partial  \right \rangle  \\
= \left \langle\mathsf{S}(\pi_1) \circ \mu, \mu \circ \mathsf{S}\left( \mathsf{S}\left( \pi_1 + \pi_4 \right) \circ \mathsf{S}(\langle 1, 0 \rangle) \circ \pi_1 + \mathsf{S}\left( \pi_1 + \pi_4 \right) \circ \partial \circ \pi_2  \right) \circ \partial  \right \rangle \\
= \left \langle\mathsf{S}(\pi_1) \circ \mu, \mathsf{S}\left( \pi_1 + \pi_4 \right) \circ \mu \circ \mathsf{S}\left(  \mathsf{S}(\langle 1, 0 \rangle) \circ \pi_1 + \mathsf{S}\left( \pi_1 + \pi_4 \right) \circ \partial \circ \pi_2  \right) \circ \partial  \right \rangle \\
= \left \langle\mathsf{S}(\pi_1) \circ \mu, \mathsf{S}(\pi_1 + \pi_4) \circ \partial \circ \mu  \right \rangle = \left\langle \mathsf{S}(\pi_1), \mathsf{S}\left( \pi_1 + \pi_4 \right) \circ \partial \right \rangle \circ \mu = \lambda \circ \mu
\end{gather*}
\item $\lambda \circ \eta= \eta \times \eta$

Note that, in the case of four copies of $A$: 
\begin{align*}
(\pi_1 + \pi_4) \circ \langle 0, 1_{A \times A} \rangle = \pi_2 
\end{align*}
Then, using \textbf{[DC.3]}, we compute: 
\begin{gather*}
\lambda \circ \eta = \left \langle\mathsf{S}\left(\pi_1\right) \circ \eta , \mathsf{S}\left( \pi_1 + \pi_4 \right) \circ \partial \circ \eta  \right \rangle = \left \langle\mathsf{S}\left(\pi_1\right) \circ \eta , \mathsf{S}\left( \pi_1 + \pi_4 \right) \circ \eta \circ \langle 0,1 \rangle \right \rangle\\
= \left \langle \eta \circ \pi_1, \eta \circ (\pi_1 + \pi_4) \circ \langle 0, 1 \rangle  \right \rangle = \left \langle \eta \circ \pi_1, \eta \circ \pi_2  \right \rangle = \eta \times \eta 
\end{gather*}
\item $p^\times \circ \lambda = \mathsf{S}(p^\times)$ is immediate by the definition of $\lambda$ and $p^\times$. 
\item $\lambda \circ \mathsf{S}(z^\times) = z^\times$ 

First, observe that we have:
\begin{align*}
\pi_1 \circ z^\times = 1 &&   \left(\pi_1 + \pi_4 \right) \circ (z^\times \times z^\times) = \pi_1 
\end{align*}
Then, using \textbf{[DC.1]}, we compute: 
\begin{gather*}
\lambda \circ \mathsf{S}(z^\times)= \left \langle\mathsf{S}\left(\pi_1\right) \circ \mathsf{S}(z^\times), \mathsf{S}\left(\pi_1 + \pi_4 \right) \circ \partial  \circ \mathsf{S}(z^\times)   \right \rangle = \left \langle 1, \mathsf{S}(\pi_1) \circ \partial \right  \rangle \\
= \left \langle 1 , \mathsf{S}\left(\pi_1  \right) \circ \partial \right \rangle = \left \langle 1 , 0 \right \rangle  = z^\times
\end{gather*}

\item $\lambda \circ \mathsf{S}\left( s^\times \right) = s^\times \circ \left \langle \mathsf{S}(\pi_1), \pi_2 \circ \lambda \circ \mathsf{S}(q^\times_1),\pi_2 \circ \lambda \circ \mathsf{S}(q^\times_2) \right\rangle$ 

First, observe that we have: 
\begin{align*}
s^\times \circ \langle f, g, h \rangle = \langle f, g+h \rangle
\end{align*}
We leave it as an exercise for the reader to check that the following equality also holds for $\pi_2 \circ \lambda_A \circ  \mathsf{S}\left( q^\times_1 \right): \mathsf{S}(A \times A \times A) \to \mathsf{S}(A)$ and $\pi_2 \circ \lambda_A \circ  \mathsf{S}\left(q^\times_2 \right): \mathsf{S}(A \times A \times A) \to \mathsf{S}(A)$: 
\begin{align*}
\pi_2 \circ \lambda_A \circ  \mathsf{S}\left( q^\times_1 \right)  = \mathsf{S}\left( \langle \pi_1, \pi_5 \rangle \right)  \circ \partial &&   \pi_2 \circ \lambda_A \circ  \mathsf{S}\left( q^\times_2 \right)  = \mathsf{S}\left(\langle \pi_1, \pi_6 \rangle \right)  \circ \partial 
\end{align*}
and also that: 
\begin{align*}
\pi_1 \circ s^\times = \pi_1 && (\pi_1 + \pi_4) \circ (s^\times \times s^\times) = \langle \pi_1, \pi_5 + \pi_6 \rangle
\end{align*}
and lastly that for six copies of $A$ and nine copies of $A$: 
\begin{equation*}\begin{gathered}
\langle \pi_1, \pi_5  \rangle =  \left \langle \pi_1, \pi_8 + \pi_9 \right \rangle \circ \left \langle \pi_1, \pi_2, 0 \right \rangle \quad \quad \quad \langle \pi_1, \pi_6  \rangle =  \left \langle \pi_1, \pi_8 + \pi_9 \right \rangle \circ \left \langle \pi_1, 0, \pi_2 \right \rangle \\
\langle \pi_1, \pi_5 + \pi_6 \rangle =  \left \langle \pi_1, \pi_8 + \pi_9 \right \rangle \circ \left \langle \pi_1, \pi_2, \pi_2 \right \rangle 
\end{gathered}\end{equation*}
Then, using \textbf{[DC.2]}, we compute: 
\begin{align*}
\lambda \circ \mathsf{S}\left( s^\times \right) &= \left\langle  \mathsf{S}(\pi_1) \circ \mathsf{S}\left( s^\times \right),  \mathsf{S}\left( \pi_1 + \pi_4 \right) \circ \partial \circ \mathsf{S}\left( s^\times \right) \right \rangle \\
&=   \left\langle  \mathsf{S}(\pi_1), \mathsf{S}\left( \left \langle \pi_1, \pi_5 + \pi_6 \right \rangle \right) \circ \partial \right \rangle =   \left\langle  \mathsf{S}(\pi_1), \mathsf{S}\left( \left \langle \pi_1, \pi_8 + \pi_9 \right \rangle \right) \circ \mathsf{S}\left( \left \langle \pi_1, \pi_2, \pi_2 \right \rangle \right) \circ \partial \right \rangle  \\
&=\left\langle  \mathsf{S}(\pi_1), \mathsf{S}\left( \left \langle \pi_1, \pi_8 + \pi_9 \right \rangle \right) \circ \left( \mathsf{S}(\langle \pi_1, \pi_2, 0 \rangle) \circ \partial + \mathsf{S}(\langle \pi_1, 0, \pi_2\rangle) \circ \partial \right) \right \rangle \\
&= \left\langle  \mathsf{S}(\pi_1), \mathsf{S}\left( \left \langle \pi_1, \pi_5 \right \rangle \right) \circ \partial +  \mathsf{S}\left( \left \langle \pi_1, \pi_6 \right \rangle \right) \circ \partial \right \rangle  \\ 
&= s^\times \circ \left\langle  \mathsf{S}(\pi_1), \mathsf{S}\left( \left \langle \pi_1, \pi_5 \right \rangle \right) \circ \partial, \mathsf{S}\left( \left \langle \pi_1, \pi_6 \right \rangle \right) \circ \partial \right \rangle \\
&= s^\times \circ \left \langle \mathsf{S}(\pi_1), \pi_2 \circ \lambda_A \circ \mathsf{S}(q^\times_1),\pi_2 \circ \lambda_A \circ \mathsf{S}(q^\times_2) \right \rangle 
\end{align*}
\end{enumerate}

For the next two identities, it will be useful to expand out $(\lambda \times \lambda) \circ \lambda: \mathsf{S}(A \times A \times A \times A) \to \mathsf{S}(A) \times \mathsf{S}(A) \times \mathsf{S}(A) \times \mathsf{S}(A)$. We leave it as an exercise for the reader to check that: 
\begin{align*}
(\lambda \times \lambda) \circ \lambda = \left\langle \mathsf{S}(\pi_1), \mathsf{S}(\pi_1 + \pi_6) \circ \partial, \mathsf{S}(\pi_1 + \pi_7) \circ \partial, \mathsf{S}(\pi_1 + \pi_7 + \pi_{10} + \pi_{16}) \circ \partial \circ \partial \right \rangle
\end{align*}

\begin{enumerate}[{\em (i)}]
\setcounter{enumi}{5}
\item $l^\times \circ \lambda = (\lambda \times \lambda) \circ \lambda \circ \mathsf{S}\left(l^\times \right)$ 

First, observe that we have: 
\begin{align*}
l^\times \circ \langle f, g \rangle = \langle f, 0, 0, g \rangle
\end{align*}
and also, using additive enrichment and the definition of $\langle -, - \rangle$, we have: 
\begin{equation*}\begin{gathered}
\pi_1 \circ l^\times = \pi_1 \quad \quad (\pi_1 + \pi_6) \circ \left( l^\times \times l^\times \right) = \pi_1 = (\pi_1 + \pi_7) \circ \left( l^\times \times l^\times \right)  \\
(\pi_1 + \pi_7 + \pi_{10} + \pi_{16}) \circ \left(l^\times \times l^\times \times l^\times \times l^\times \right) = (\pi_1 + \pi_4) \circ \langle \pi_1, \pi_4 \rangle   
\end{gathered}\end{equation*}
Then, using \textbf{[DC.1]} and \textbf{[DC.5]}, we compute:\\
\resizebox{\linewidth}{!}{
\begin{minipage}{\linewidth}
\begin{align*}
(\lambda \times \lambda) \circ \lambda \circ \mathsf{S}\left( l^\times \right) &= \big\langle \mathsf{S}(\pi_1) \circ \mathsf{S}\left( l^\times \right), \mathsf{S}(\pi_1 + \pi_6) \circ \partial \circ \mathsf{S}\left( l^\times \right),\mathsf{S}(\pi_1 + \pi_7) \circ \partial \circ \mathsf{S}\left( l^\times \right),\\
&\qquad \mathsf{S}(\pi_1 + \pi_7 + \pi_{10} + \pi_{16}) \circ \partial \circ \partial \circ \mathsf{S}\left( l^\times \right) \big\rangle \\ 
&=  \left\langle \mathsf{S}(\pi_1), \mathsf{S}(\pi_1) \circ \partial, \mathsf{S}(\pi_1) \circ \partial, \mathsf{S}(\pi_1 + \pi_4) \circ \mathsf{S}(\langle \pi_1, \pi_4 \rangle) \circ \partial \circ \partial \right \rangle \\
& = \left\langle \mathsf{S}(\pi_1),0, 0, \mathsf{S}(\pi_1 + \pi_4) \circ \partial \right \rangle= l^\times \circ \left\langle \mathsf{S}(\pi_1), \mathsf{S}(\pi_1 + \pi_4) \circ \partial \right \rangle = l^\times \circ \lambda
\end{align*}
\end{minipage}
}\\
\item $\left \langle \pi_1, \pi_3, \pi_2, \pi_4 \right \rangle \circ (\lambda \times \lambda) \circ \lambda = (\lambda \times \lambda) \circ \lambda \circ \mathsf{S}\left( \left \langle \pi_1, \pi_3, \pi_2, \pi_4 \right \rangle \right)$ 

First, observe that we have: 
\begin{align*}
c^\times \circ \left \langle f,g,h,k \right \rangle =  \left \langle f,h,g,k \right \rangle
\end{align*}
and also, using additive enrichment and the definition of $\langle -, - \rangle$, we have: 
\begin{equation*}\begin{gathered}
\pi_1 \circ c^\times  = \pi_1 \quad \quad \quad (\pi_1 + \pi_6) \circ  c^\times  = \pi_1 +\pi_7 \quad \quad \quad (\pi_1 + \pi_7) \circ \left( c^\times \times c^\times \right) = \pi_1 +\pi_6 \\
(\pi_1 + \pi_7 + \pi_{10} + \pi_{16}) \circ \left( c^\times \times c^\times \times c^\times \times c^\times  \right) = (\pi_1 + \pi_7 + \pi_{10} + \pi_{16}) \circ  \left \langle \pi_1, \pi_3, \pi_2, \pi_4 \right \rangle
\end{gathered}\end{equation*}
Then, using \textbf{[DC.6]}, we compute: 
\begin{align*}
\hskip-1cm(\lambda \times \lambda) \circ \lambda \circ \mathsf{S}\left( c^\times \right) &=  \left\langle \mathsf{S}(\pi_1) \circ \mathsf{S}\left( c^\times \right), \mathsf{S}(\pi_1 + \pi_6) \circ \partial \circ \mathsf{S}\left( c^\times \right), \mathsf{S}(\pi_1 + \pi_7) \circ \partial \circ \mathsf{S}\left( c^\times \right),\right.\\
&\qquad\left.\mathsf{S}(\pi_1 + \pi_7 + \pi_{10} + \pi_{16}) \circ \partial \circ \partial \circ \mathsf{S}\left( c^\times \right) \right \rangle \\ 
&=~\big\langle \mathsf{S}(\pi_1), \mathsf{S}(\pi_1 + \pi_7) \circ \partial, \mathsf{S}(\pi_1 + \pi_6) \circ \partial,\\&\qquad \mathsf{S}(\pi_1 + \pi_7 + \pi_{10} + \pi_{16}) \circ \mathsf{S}(\left \langle \pi_1, \pi_3, \pi_2, \pi_4 \right \rangle) \circ \partial \circ \partial \big\rangle  \\
&=~ \left\langle \mathsf{S}(\pi_1), \mathsf{S}(\pi_1 + \pi_7) \circ \partial, \mathsf{S}(\pi_1 + \pi_6) \circ \partial, \mathsf{S}(\pi_1 + \pi_7 + \pi_{10} + \pi_{16})  \circ \partial \circ \partial \right \rangle  \\ 
&=~ c^\times \circ \left\langle \mathsf{S}(\pi_1), \mathsf{S}(\pi_1 + \pi_6) \circ \partial, \mathsf{S}(\pi_1 + \pi_7) \circ \partial, \mathsf{S}(\pi_1 + \pi_7 + \pi_{10} + \pi_{16})  \circ \partial \circ \partial \right \rangle \\
&=~c^\times \circ (\lambda \times \lambda) \circ \lambda 
\end{align*}

\end{enumerate}
We conclude that $(\mathsf{S}, \mu, \eta, \lambda)$ is a tangent monad. If $\mathbb{X}$ is also an additive category, it remains to show that $\lambda$ also satisfies the last identity from Definition \ref{def:tangent-monad}. 
\begin{enumerate}[{\em (i)}]
\setcounter{enumi}{7}
\item $n^\times \circ \lambda = \lambda \circ \mathsf{S}(n^\times)$

First, observe that: 
\begin{align*}
n^\times \circ (\langle f, g \rangle) = \langle f, -g \rangle    
\end{align*}
and also that: 
\begin{align*}
\pi_1 \circ n^\times = \pi_1 && (\pi_1 + \pi_4) \circ (n^\times \times n^\times) = (\pi_1 + \pi_4) \circ n^\times
\end{align*}
Then, using \textbf{[DC.N]}, we compute:
\begin{align*}
n^\times \circ \lambda &= n^\times \circ \left \langle\mathsf{S}\left(\pi_1\right) , \mathsf{S}\left(\pi_1 + \pi_4 \right) \circ \partial   \right \rangle = \left \langle\mathsf{S}\left(\pi_1\right) , \mathsf{S}\left(\pi_1 + \pi_4 \right) \circ (-\partial)   \right \rangle \\
&= \left \langle\mathsf{S}\left(\pi_1\right) , \mathsf{S}\left(\pi_1 + \pi_4 \right) \circ \mathsf{S}(n^\times) \circ \partial   \right \rangle = \left \langle\mathsf{S}\left(\pi_1\right) \circ \mathsf{S}(n^\times) , \mathsf{S}\left(\pi_1 + \pi_4 \right)  \circ \partial  \circ \mathsf{S}(n^\times) \right \rangle \\
&= \lambda \circ \mathsf{S}(n^\times) 
\end{align*}
\end{enumerate}
We conclude that $(\mathsf{S}, \mu, \eta, \lambda)$ is a Rosick\'y tangent monad.\end{proof}

The converse of Proposition \ref{prop:cdmonad-tanmonad} is also true, that is, every tangent monad on a semi-additive category induces a coCartesian differential monad. Furthermore, these constructions are inverses of each other, and therefore, for a semi-additive category, the data of a tangent monad is equivalent to that of a coCartesian differential monad. 

\begin{lemma}\label{lem:S-tan} Let $\mathbb{X}$ be a semi-additive category, and let $(\mathsf{S}, \mu, \eta, \lambda)$ be a tangent monad on $(\mathbb{X}, \mathbb{B})$. Define the natural transformation $\partial_A: \mathsf{S}(A) \to \mathsf{S}(A \times A)$ as follows: 
\begin{align*}
\partial_A := \pi_2 \circ \lambda_{A \times A} \circ \mathsf{S}(\langle 1_A, 0, 0, 1_A \rangle)  
\end{align*}
Then, $(\mathsf{S}, \mu, \eta, \partial)$ is a coCartesian differential monad on $\mathbb{X}$.
\end{lemma}
\begin{proof}Since the proof is essentially again by brute force calculations, and not necessarily more enlightening for this paper, we omit them here and instead simply give a sketch. To prove \textbf{[DC.1]}, we use the axiom between $\lambda$ and $z^\times$. To prove \textbf{[DC.2]}, we use the axiom between $\lambda$ and $s^\times$. To prove \textbf{[DC.3]}, we use the axiom between $\lambda$ and $\eta$. To prove \textbf{[DC.4]}, we use the axiom between $\lambda$ and $\mu$. To prove \textbf{[DC.5]}, we use the axiom between $\lambda$ and $l^\times$. And lastly, to prove \textbf{[DC.6]}, we use the axiom between $\lambda$ and $c^\times$.
\end{proof}

\begin{corollary} For a monad $(\mathsf{S}, \mu, \eta)$ on a semi-additive category $\mathbb{X}$, there is a bijective correspondence between $\mathbb{B}$-distributive laws and differential combinator transformations. Therefore, for a semi-additive category $\mathbb{X}$, there is a bijective correspondence between tangent monads on $(\mathbb{X}, \mathbb{B})$ and coCartesian differential monads on $\mathbb{X}$. 
\end{corollary}
\begin{proof} We must show that the constructions of Proposition \ref{prop:cdmonad-tanmonad} and Lemma \ref{lem:S-tan} are mutually inverse. Starting with a $\mathbb{B}$-distributive law $\lambda$, observe first that, for two copies of $A$ we have that: 
\begin{align*}
\left( (\pi_1 + \pi_4) \times (\pi_1 + \pi_4) \right) \circ  \langle 1_{A \times A}, 0, 0, 1_{A \times A} \rangle = 1_{A \times A}
\end{align*}
Then, we compute: 
\begin{align*}
\left\langle \mathsf{S}(\pi_1), \mathsf{S}\left( \pi_1 + \pi_4 \right) \circ \partial_{A \times A} \right \rangle &= \left\langle \mathsf{S}(\pi_1), \mathsf{S}\left( \pi_1 + \pi_4 \right) \circ \pi_2 \circ \lambda_{A \times A \times A \times A} \circ \mathsf{S}(\langle 1_{A \times A}, 0, 0, 1_{A \times A} \rangle)   \right \rangle \\
&=~ \left\langle \pi_1 \circ \lambda_A,  \pi_2 \circ \lambda_A   \right \rangle = \left\langle \pi_1 ,  \pi_2   \right \rangle \circ \lambda_A = \lambda_A 
\end{align*}
Starting instead from a differential combinator transformation $\partial$, observe first that, for eight copies of $A$: 
\begin{align*}
\left( \pi_1 + \pi_4 \right) \circ  \left( \langle 1_{A}, 0, 0, 1_{A} \rangle \times \langle 1_{A}, 0, 0, 1_{A} \rangle \right) = 1_{A \times A}
\end{align*}
Then, we compute: 
\begin{gather*}
\pi_2 \circ \lambda_{A \times A} \circ \mathsf{S}(\langle 1_A, 0, 0, 1_A \rangle)  =  \mathsf{S}\left( \pi_1 + \pi_4 \right) \circ \partial_{A \times A} \circ  \mathsf{S}(\langle 1_A, 0, 0, 1_A \rangle) = \partial_A 
\end{gather*}   
We conclude that $\mathbb{B}$-distributive laws and differential combinator transformations are indeed in bijective correspondence, and therefore, so are coCartesian differential monads and tangent monads. 
\hfill \end{proof}

Since every coCartesian differential monad is a tangent monad, by applying \cite[Proposition 20]{cockett_et_al:LIPIcs:2020:11660}, we obtain a tangent structure on the category of algebras of a coCartesian differential monad. We expand out this construction in detail. Recall that, for a monad $(\mathsf{S}, \mu, \eta)$ on a category $\mathbb{X}$ with finite products, $\mathsf{Alg}_{\mathsf{S}}$ also has finite products:
\[ (A, \alpha) \times (A^\prime, \alpha^\prime) := (A \times A^\prime, \langle \alpha \circ \mathsf{S}(\pi_1), \alpha^\prime \circ \mathsf{S}(\pi_2) \rangle) \]
The terminal object is $(\ast, t_{\mathsf{S}(\ast)})$, and the projections and pairings are the same as in $\mathbb{X}$. 

\begin{thm}\label{thm:S-tan} Let $(\mathsf{S}, \mu, \eta, \partial)$ be a coCartesian differential monad on a semi-additive category $\mathbb{X}$. Define: 
\begin{enumerate}[{\em (i)}]
\item The tangent bundle functor as the functor $\mathsf{T}: \mathsf{ALG}_{\mathsf{S}} \to \mathsf{ALG}_{\mathsf{S}}$ defined on objects as:
\begin{align*}
\mathsf{T}(A,\alpha) = \left(A \times A, (\alpha \times \alpha) \circ \lambda \right) = \left(A \times A, \left\langle \alpha \circ \mathsf{S}(\pi_1), \alpha \circ \mathsf{S}\left( \pi_1 + \pi_4 \right) \circ \partial_{A \times A} \right \rangle \right)
\end{align*}
and on maps as $\mathsf{T}(f) = f \times f$;
\item The projection as the natural transformation $p_{(A,\alpha)}:\mathsf{T}(A,\alpha) \to (A, \alpha)$ defined as:
\[p_{(A,\alpha)} := \pi_1\] 
and where the $n$-fold pullback of $p_{(A,\alpha)}$ is:
\begin{multline*}
\mathsf{T}_n(A,\alpha)=\Bigg( \prod\limits^{n+1}_{i=1} A,  \big\langle \alpha \circ \mathsf{S}(\pi_1), \alpha \circ \mathsf{S}\left( \pi_1 + \pi_4 \right) \circ \partial_{A \times A} \circ \mathsf{S}(\langle \pi_1, \pi_2 \rangle), \\\hdots, \alpha \circ \mathsf{S}\left( \pi_1 + \pi_4 \right) \circ \partial_{A \times A}  \circ \mathsf{S}(\langle \pi_1, \pi_{n+1} \rangle) \big\rangle  \Bigg)
\end{multline*}
with pullback projections ${q_j: \mathsf{T}_n(A,\alpha) \to \mathsf{T}(A,\alpha)}$ defined as: \[q_j = \langle \pi_1, \pi_{j+1} \rangle\]
\item The sum as the natural transformation $s_A: \mathsf{T}_2(A,\alpha) \to \mathsf{T}(A,\alpha)$ defined as:
\[s_{(A,\alpha)} := \langle \pi_1, \pi_2 + \pi_3 \rangle\]
\item The zero map as the natural transformation $z_{(A,\alpha)}: (A, \alpha) \to \mathsf{T}(A,\alpha)$ defined as:
\[z_{(A,\alpha)} := \langle 1_A, 0 \rangle\]
\item The vertical lift as the natural transformation $l_{(A,\alpha)}:\mathsf{T}^2(A,\alpha) \to \mathsf{T}(A,\alpha)$ defined as:
\[l_{(A,\alpha)} := \langle \pi_1, 0,0, \pi_2 \rangle\]
\item The canonical flip as the natural transformation $c_{(A,\alpha)}: \mathsf{T}^2(A,\alpha) \to \mathsf{T}^2(A,\alpha)$ defined as:
\[c_{(A,\alpha)} := \!\langle \pi_1, \pi_3, \pi_2, \pi_4 \rangle\]
\end{enumerate}
Then, $\mathbb{T} = (\mathsf{T},  p, s, z, l, c)$ is a tangent structure on $\mathsf{ALG}_\mathsf{S}$, and so $(\mathsf{ALG}_\mathsf{S}, \mathbb{T})$ is a Cartesian tangent category. If $\mathbb{X}$ is also an additive category, then define: 
\begin{enumerate}[{\em (i)}]
\setcounter{enumi}{6}
\item The negative map as the natural transformation $n_{(A,\alpha)}:  \mathsf{T}(A,\alpha) \to \mathsf{T}(A,\alpha)$ defined as: 
\[n_{(A,\alpha)} := \langle \pi_1, -\pi_2 \rangle\]
\end{enumerate}
Then, $\mathbb{T}= (\mathsf{T},  p, s, z, l, c, n)$ is a Rosick\'y tangent structure on $\mathsf{ALG}_\mathsf{S}$, and so, $(\mathsf{ALG}_\mathsf{S}, \mathbb{T})$ is a Cartesian Rosick\'y tangent category. 
\end{thm}

We again stress that even if for an $\mathsf{S}$-algebra $(A, \alpha)$, its tangent bundle $\mathsf{T}(A,\alpha)$ and the product $(A, \alpha) \times (A, \alpha)$ have the same underlying object $A \times A$, they are in general not equal (or even isomorphic) as $\mathsf{S}$-algebras. They are only equal when $(A, \alpha)$ is a differential object, as we will discuss below in Section \ref{subsec:diffobj-ccdc}. 

\subsection{Adjoint Tangent Structure for coCartesian Differential Monads}
\label{subsection:adjoint-tangent-structure-cCD-monads}

In this section, we discuss adjoint tangent structure for the category of algebras of a coCartesian differential monad. The first thing to observe is that, for any semi-additive category $\mathbb{X}$, $(\mathbb{X}, \mathbb{B})$ has adjoint tangent structure, where, since $\mathbb{X}^{op}$ is also a (semi-)additive category, the tangent bundle functor is the same as the one in $\mathbb X$. Concretely, to explain why $(\mathbb{X}, \mathbb{B})$ has adjoint tangent structure, by Lemma \ref{lemma:dual-tangent-structure} it suffices to explain why the tangent bundle functor has a left adjoint. 

\begin{lemma}\cite[Section 6]{cockett_et_al:LIPIcs:2020:11660}\label{lemma:biproduct-adjoint} Let $\mathbb{X}$ be a semi-additive category. The tangent bundle functor $\mathsf{B}$ is its own left adjoint, where the unit $\eta^\times_A: A \to A \times A \times A \times A$ of the adjunction injects $A$ into the first and last component:
\[\eta^\times_A := \langle 1_A, 0, 0, 1_A \rangle\] 
and where the counit $\varepsilon_A: A \times A \times A \times A \to A$ sums the first and last components:
\[\varepsilon^\times_A := \pi_1 + \pi_4.\] 
So $(\eta^\times, \varepsilon^\times): \mathsf{B} \dashv \mathsf{B}$. Therefore, $(\mathbb{X}, \mathbb{B})$ has adjoint tangent structure, where: 
\begin{enumerate}[{\em (i)}]
\item The adjoint tangent bundle functor is the tangent bundle functor $\mathsf{B}: \mathbb{X}^{op} \to \mathbb{X}^{op}$;
\item The adjoint projection $p^{\times^\circ}_A: A \times A \to A$ is precisely the zero of the tangent structure:
\[p^{\times^\circ}_A = \langle 1_A, 0 \rangle\] 
and $\mathsf{B}_n(A)$ is the $n$-fold pushout of $p^{\times^\circ}_A$ where the $j$-th pushout injection $q^{\times^\circ}_j: A \times A \to \mathsf{B}_n(A)$ is given by injecting the first component into the first component and the second component into the $j$-th component: 
\[q^{\times^\circ}_j := \langle \pi_1, 0, \hdots, 0, \pi_2, 0, \hdots, 0 \rangle\]
\item The adjoint sum $s^{\times^\circ}_A: A \times A \to A \times A \times A$ is given by copying the second component:
\[s^{\times^\circ}_A := \langle \pi_1, \pi_2, \pi_2 \rangle\]
\item The adjoint vertical lift $l^{\times^\circ}_A: A \times A \times A \times A \to A \times A$ projects the first component onto the first component and the fourth component onto the second component:
\[l^{\times^\circ}_A := \langle \pi_1, \pi_4 \rangle\] 
\item The adjoint canonical flip $c^{\times^\circ}_A: A \times A \times A \times A \to A \times A \times A \times A$ is the same as the canonical flip:
\[c^{\times^\circ}_A := \langle \pi_1, \pi_3, \pi_2, \pi_4 \rangle\]
\end{enumerate}
Then $\mathbb{B}^\circ= (\mathsf{B},  p^{\times^\circ}, s^{\times^\circ}, z^{\times^\circ}, l^{\times^\circ}, c^{\times^\circ})$ is a tangent structure on $\mathbb{X}^{op}$, and so $(\mathbb{X}^{op}, \mathbb{B}^\circ)$ is a Cartesian tangent category. Similarly, if $\mathbb{X}$ is an additive category, then: 
\begin{enumerate}[{\em (i)}]
\setcounter{enumi}{6}
\item The adjoint negative $n^{\times^\circ}: A\times A \to A \times A$ is the same as the negative of the tangent structure:
\[n^{\times^\circ}_A = \langle \pi_1, - \pi_2 \rangle\]
\end{enumerate}
Then $\mathbb{B}^\circ= (\mathsf{B},  p^{\times^\circ}, s^{\times^\circ}, z^{\times^\circ}, l^{\times^\circ}, c^{\times^\circ}, n^{\times^\circ})$ is a Rosick\'y tangent structure on $\mathbb{X}^{op}$, and so $(\mathbb{X}^{op}, \mathbb{B}^\circ)$ is a Cartesian Rosick\'y tangent category.
\end{lemma}

If $(\mathsf{S}, \mu, \eta, \partial)$ is a coCartesian differential monad on a (semi-)additive category $\mathbb{X}$, then the tangent bundle functor $\mathsf{T}$ on $\mathsf{ALG}_\mathsf{S}$ is a lifting of $\mathsf{B}$ via the distributive law $\lambda$, in the sense of \cite[Lemma 1]{johnstone1975adjoint}. Even if $\mathsf{B}$ is its own left adjoint, in general, $\mathsf{T}$ will not be its own left adjoint, or even necessarily have a left adjoint. As discussed in \cite{johnstone1975adjoint}, a sufficient condition for $\mathsf{T}$ to have a left adjoint is the existence of reflexive coeqalizers in $\mathsf{ALG}_\mathsf{S}$. 

By \cite[Theorem 2]{johnstone1975adjoint}, if $\mathsf{ALG}_\mathsf{S}$ has reflexive coequalizers, then the tangent bundle functor $\mathsf{T}$ has a left adjoint $\mathsf{T}^\circ: \mathsf{ALG}_\mathsf{S} \to \mathsf{ALG}_\mathsf{S}$. So, we have a unit $\eta^\mathsf{S}_{(A, \alpha)}: (A, \alpha) \to \mathsf{T} \mathsf{T}^\circ(A,\alpha)$ and a counit $\varepsilon^\mathsf{S}_{(A, \alpha)}: \mathsf{T}^\circ \mathsf{T}(A,\alpha) \to (A, \alpha)$, and thus, $(\eta^\mathsf{S}, \varepsilon^\mathsf{S}): \mathsf{T}^\circ \dashv \mathsf{T}$. In particular, as explained in the proof of \cite[Theorem 2]{johnstone1975adjoint}, $\mathsf{T}^\circ(A, \alpha)$ is the reflexive coequalizer in $\mathsf{ALG}_\mathsf{S}$ of $\mathsf{S}(\alpha \times \alpha)$ and the composite:
\[ \mu_{A \times A} \circ \mathsf{S}\left(\varepsilon^\times_{\mathsf{S}(A \times A)}\right) \circ \mathsf{S}(\lambda_{A\times A} \times \lambda_{A\times A}) \circ \mathsf{S}\left( \mathsf{S}(\eta^\times_A) \times \mathsf{S}(\eta^\times_A) \right) \]
which are both $\mathsf{S}$-algebra morphisms of type $\left( \mathsf{S}\left(\mathsf{S}(A) \times \mathsf{S}(A) \right), \mu_{\mathsf{S}(A) \times \mathsf{S}(A)} \right)\to \left( \mathsf{S}(A \times A), \mu_{A \times A} \right)$ with common section $\mathsf{S}(\eta_A \times \eta_A): \left( \mathsf{S}(A \times A), \mu_{A \times A} \right) \to \left( \mathsf{S}\left(\mathsf{S}(A) \times \mathsf{S}(A) \right), \mu_{\mathsf{S}(A) \times \mathsf{S}(A)} \right)$. Note that, in general, coequalizers in $\mathsf{ALG}_\mathsf{S}$ can be very different from coequalizers in $\mathbb{X}$. Therefore, the underlying object of $\mathsf{T}^\circ(A, \alpha)$ will in general not be $A \times A$. In fact, in most cases, it is a much more complex object. However, again as explained in the proof of \cite[Theorem 2]{johnstone1975adjoint}, if we consider free algebras, for every object $A$ in $\mathbb{X}$, we do have a natural isomorphism $\tau_A: \mathsf{T}^\circ\left( \mathsf{S}(A), \mu_A \right) \to (\mathsf{S}(A \times A), \mu_{A \times A})$, defined as:
\[ \tau_A := \varepsilon^\mathsf{S}_{\left( \mathsf{S}\left(A \times A \right), \mu_{A \times A} \right)} \circ \mathsf{T}^\circ(\lambda_{A \times A}) \circ \mathsf{T}^\circ \mathsf{S}(\eta^\times_A) \]
where the distributive law is here seen as an $\mathsf{S}$-algebra morphism $\lambda_B: (\mathsf{S}(B \times B), \mu_{B \times B}) \to \mathsf{T}\left( \mathsf{S}(B), \mu_B \right)$, so that $\mathsf{T}^\circ(\lambda_B)$ is well-defined. So we have that $\mathsf{T}^\circ\left( \mathsf{S}(A), \mu_A \right) \cong (\mathsf{S}(A \times A), \mu_{A \times A})$. Similarly, for each $\mathsf{T}_n$, we also obtain left adjoints ${\mathsf{T}_n^\circ: \mathsf{ALG}_\mathsf{S} \to \mathsf{ALG}_\mathsf{S}}$. Therefore, $(\mathsf{ALG}_\mathsf{S}, \mathbb{T})$ has adjoint tangent structure, and so $(\mathsf{ALG}^{op}_\mathsf{S}, \mathbb{T}^\circ)$ is a (Rosick\'y) tangent category. One could also express the (Rosick\'y) tangent structure $\mathbb{T}^\circ$ using the fact that $\mathsf{T}^\circ(A, \alpha)$ is a reflexive coequalizer.

There is a relation between $(\mathsf{ALG}^{op}_\mathsf{S}, \mathbb{T}^\circ)$ and $(\mathbb{X}^{op}, \mathbb{B}^\circ)$ via the natural isomorphism $\tau$, that is, the following equalities hold  in $\mathbb{X}$: 
\begin{equation*}\begin{gathered}
\tau_A \circ p^\circ_{(\mathsf{S}(A), \mu_A)} = \mathsf{S}(p^{\times^\circ}_A) \quad \quad \quad \mathsf{S}(z^{\times^\circ}_A) \circ \tau_A = z^\circ_{(\mathsf{S}(A), \mu_A)}  \\ 
\mathsf{S}\left( s^{\times^\circ}_A \right) \circ \tau_A =  \left [\mathsf{S}(q^{\times^\circ}_1) \circ \tau_A, \mathsf{S}(q^{\times^\circ}_2) \circ \tau_A \right ]_{p^\circ_{(\mathsf{S}(A), \mu_A)}} \circ s^\circ_{(\mathsf{S}(A), \mu_A)} \\ 
\tau_A \circ l^\circ_{(\mathsf{S}(A), \mu_A)} =  \mathsf{S}\left( l^{\times^\circ}_A \right) \circ \tau_{A \times A} \circ \mathsf{T}^\circ(\tau_A) \\
\tau_{A \times A} \circ \mathsf{T}^\circ(\tau_A) \circ c^\circ_{(\mathsf{S}(A), \mu_A)} = \mathsf{S}\left( c^{\times^\circ}_A \right) \circ \tau_{A \times A} \circ \mathsf{T}^\circ(\tau_A)  \\
\left( \text{ and if additive, then also } \tau_A \circ n^\circ_{(\mathsf{S}(A), \mu_A)} =  \mathsf{S}(n^{\times^\circ}_A) \circ \tau_A ~ \right)
\end{gathered}\end{equation*}
These above equalities follow more or less immediately from the fact that $\tau$ is built out of the counit of $\mathsf{T}^\circ \dashv \mathsf{T}$, the unit of $\mathsf{B} \dashv \mathsf{B}$, and $\lambda$. Indeed, $\lambda$ is a strict compatibility between $\mathbb{B}$ and $\mathbb{T}$, and, since $\mathbb{B}^\circ$ and $\mathbb{T}^\circ$ are constructed using the units and counits of their respective adjunctions, it follows that $\tau$ will also provide a compatibility between $\mathbb{B}^\circ$ and $\mathbb{T}^\circ$. Furthermore, these above identities essentially say that $\tau$ is part of a strong (Rosick\'y) tangent morphism in the sense of \cite[Definition 2.7]{cockett2014differential}. Indeed, consider the free $\mathsf{S}$-algebra functor $\mathsf{F}^\mathsf{S}: \mathbb{X} \to \mathsf{ALG}_\mathsf{S}$, which maps objects to $\mathsf{F}^\mathsf{S}(A) = (\mathsf{S}(A), \mu_A)$ and maps to $\mathsf{F}^\mathsf{S}(f) = \mathsf{S}(f)$. Then, $\tau$ is a natural isomorphism of type $\tau_A: \mathsf{T}^\circ \mathsf{F}^\mathsf{S}(A) \to \mathsf{F}^\mathsf{S}(A \times A)$, and so we obtain that $(\mathsf{F}^{\mathsf{S}}, \tau): (\mathbb{X}^{op}, \mathbb{B}^\circ) \to (\mathsf{ALG}^{op}_\mathsf{S}, \mathbb{T}^\circ)$ is a strong (Rosick\'y) tangent morphism. 

Lastly, we would also like $(\mathsf{ALG}^{op}_\mathsf{S}, \mathbb{T}^\circ)$ to be a Cartesian (Rosick\'y) tangent category. To do so, we also require that $\mathsf{ALG}_\mathsf{S}$ has finite coproducts, which we denote by $+$. First, observe that $(\mathsf{S}(\ast), \mu_\ast)$ is the initial object in $\mathsf{ALG}_\mathsf{S}$. However, it is important to note that, even if $\times$ is a coproduct in $\mathbb{X}$, the underlying object of the coproduct of $\mathsf{S}$-algebras $(A, \alpha) + (A, \alpha^\prime)$ is in general not $A \times A^\prime$. We do however have a natural isomorphism:
\[\theta_{A,B}: (\mathsf{S}(A), \mu_A) + (\mathsf{S}(B), \mu_{B}) \to (\mathsf{S}(A \times B), \mu_{A \times B})\] 
So, $\mathsf{F}^\mathsf{S}$ preserves coproducts up to isomorphism. In this case, $(\mathsf{ALG}^{op}_\mathsf{S}, \mathbb{T}^\circ)$ will be a Cartesian (Rosick\'y) tangent category, and $(\mathsf{F}^{\mathsf{S}}, \tau): (\mathbb{X}^{op}, \mathbb{B}^\circ) \to (\mathsf{ALG}^{op}_\mathsf{S}, \mathbb{T}^\circ)$ will be a strong Cartesian (Rosick\'y) tangent morphism in the sense of \cite[Definition 2.7]{cockett2014differential}, while the forgetful functor will be a \emph{strict} Cartesian (Rosick\'y) tangent morphism. We summarize these results as follows: 

\begin{thm}
\label{theorem:adjoint-tangent-structure-for-cCD-monads}
Let $(\mathsf{S}, \mu, \eta, \partial)$ be a coCartesian differential monad on a semi-additive (resp. additive) category $\mathbb{X}$. If $\mathsf{ALG}_\mathsf{S}$ has all reflexive coequalizers, then $(\mathsf{ALG}_\mathsf{S}, \mathbb{T})$ has adjoint tangent structure. So $(\mathsf{ALG}^{op}_\mathsf{S}, \mathbb{T}^\circ)$ is a (Rosick\'y) tangent category, and there is a natural isomorphism $\tau_A: \mathsf{T}^\circ\left( \mathsf{S}(A), \mu_A \right) \to (\mathsf{S}(A \times A), \mu_{A \times A})$ such that, with the free algebra functor $\mathsf{F}^{\mathsf{S}}$ forms a strong (Rosick\'y) tangent morphism $(\mathsf{F}^{\mathsf{S}}, \tau): (\mathbb{X}^{op}, \mathbb{B}^\circ) \to (\mathsf{ALG}^{op}_\mathsf{S}, \mathbb{T}^\circ)$. If $\mathsf{ALG}_\mathsf{S}$ also has finite coproducts, then $(\mathsf{ALG}^{op}_\mathsf{S}, \mathbb{T}^\circ)$ is a Cartesian (Rosick\'y) tangent category, and also $(\mathsf{F}^{\mathsf{S}}, \tau): (\mathbb{X}^{op}, \mathbb{B}^\circ) \to (\mathsf{ALG}^{op}_\mathsf{S}, \mathbb{T}^\circ)$ is a strong Cartesian (Rosick\'y) tangent morphism.  
\end{thm}

Let us now re-state this result in terms of Cartesian differential comonads \cite[Definition 3.1]{ikonicoff2021cartesian}, the dual of coCartesian differential monads. We do this so that it is clearly recorded for future work, and also definitively answers the open question postulated in the conclusion of \cite{ikonicoff2021cartesian}. 

\begin{proposition} For a Cartesian differential comonad on a semi-additive (resp. additive) category, the opposite of the coEilenberg--Moore category is a Cartesian (Rosick\'y) tangent category. Furthermore, if the coEilenberg--Moore category has coreflexive equalizers (and finite products), then the coEilenberg--Moore category is a (Cartesian) (Rosick\'y) tangent category with adjoint tangent structure
\end{proposition}

\subsection{Vector Fields for coCartesian Differential Monads}\label{sec:SDer}

Let us now discuss vector fields in the category of algebras of a coCartesian differential monad. First, note that, since the forgetful functor preserves the tangent structure strictly, vector fields in the category of algebras will also be vector fields in the base category. Now, if $\mathbb{X}$ is a semi-additive category, then vector fields in $(\mathbb{X}, \mathbb{B})$ correspond to endomorphisms $A \to A$. Indeed, a vector field on an object $A$ in $(\mathbb{X}, \mathbb{B})$ is a map $v: A \to A \times A$ such that $\pi_1 \circ v = 1_A$. Therefore, $v$ is of the form $v = \langle 1_A, f_v \rangle$ for a unique map $f_v: A \to A$. Conversely, for any endomorphism $f: A \to A$, $\langle 1_A, f \rangle: A \to A \times A$ is a vector field. So, $\mathsf{V}_{\mathbb{B}}(A)$ is isomorphic to the set of endomorphisms of $A$. If $\mathbb{X}$ is an additive category, then the tangent category Lie bracket is given by the commutator: $[v,w] = \langle 1_A, (f_v\circ f_w) - (f_w \circ f_v) \rangle$. Therefore, for a coCartesian differential monad $(\mathsf{S}, \mu, \eta, \partial)$, vector fields in $(\mathsf{ALG}_\mathsf{S}, \mathbb{T})$ will correspond to certain endomorphisms in $\mathbb{X}$ which induce $\mathsf{S}$-algebra morphisms. We call these special endomorphisms $\mathsf{S}$-derivations. This terminology comes from the fact that, as we will see in Section \ref{sec:vf-operad}, vector fields in the tangent category of algebras of an operad correspond precisely to derivations in the operadic sense.

\begin{definition}
\label{definition:monadic-derivations}
Let $(\mathsf{S}, \mu, \eta, \partial)$ be a coCartesian differential monad on a semi-additive category $\mathbb{X}$. An \textbf{$\mathsf{S}$-derivation} on an $\mathsf{S}$-algebra $(A, \alpha)$ is a map $D: A \to A$ such that the following equality holds: 
\begin{align*}
D \circ \alpha = \alpha \circ \mathsf{S}(\pi_1 + D \circ \pi_2) \circ \partial_A    
\end{align*}
Let $\mathsf{DER}_{\mathsf{S}}(A, \alpha)$ be the set of $\mathsf{S}$-derivations on $(A, \alpha)$.    
\end{definition}

\begin{lemma}
\label{lemma:abstract-derivations-vector-fields}
Let $(\mathsf{S}, \mu, \eta, \partial)$ be a coCartesian differential monad on a semi-additive category $\mathbb{X}$. Then, for an $S$-algebra $(A, \alpha)$, there is a bijective correspondence between vector fields on $(A, \alpha)$ in $(\mathsf{ALG}_\mathsf{S}, \mathbb{T})$ and $\mathsf{S}$-derivations on $(A, \alpha)$. Explicitly, 
\begin{enumerate}[{\em (i)}]
\item If $v \in \mathsf{V}_{\mathbb{T}}(A,\alpha)$, define $D_{v}: A \to A$ as $D_{v} := \pi_2 \circ v$. Then, $D_{v} \in \mathsf{DER}_{\mathsf{S}}(A, \alpha)$;
\item If $D \in \mathsf{DER}_{\mathsf{S}}(A, \alpha)$, define $v_D: A \to A \times A$ as $v_D = \langle 1_A, D \rangle$. Then $v_D \in \mathsf{V}_{\mathbb{T}}(A,\alpha)$;
\end{enumerate}
and these constructions are inverses of each other. So, $\mathsf{V}_{\mathbb{T}}(A,\alpha) \cong \mathsf{DER}_{\mathsf{S}}(A, \alpha)$. If $\mathbb{X}$ is also an additive category, then the induced Lie bracket on $\mathsf{V}_{\mathbb{T}}(A,\alpha)$, as defined in Proposition \ref{prop:lie}, is $[v,w] = \langle 1_A, (D_v\circ D_w) - (D_w \circ D_v) \rangle$. 
\end{lemma}
\begin{proof} Starting from vector fields, let $v \in \mathsf{V}_{\mathbb{T}}(A,\alpha)$. In $\mathbb{X}$, this means that the vector field is of type ${v: A \to A \times A}$ and $\pi_1 \circ v = 1_A$. We also have that $v: (A,\alpha) \to \mathsf{T}(A,\alpha)$ is an $\mathsf{S}$-algebra morphism, which means that:
\begin{align*}
v \circ \alpha = \left\langle \alpha \circ \mathsf{S}(\pi_1), \alpha \circ \mathsf{S}\left( \pi_1 + \pi_4 \right) \circ \partial_{A \times A} \right \rangle \circ \mathsf{S}(v)
\end{align*}
Now, observe that:
\begin{align*}
(\pi_1 + \pi_4) \circ (v \times v) = \pi_1 + D_v \circ \pi_2    
\end{align*}
Therefore, we compute that: 
\begin{align*}
\alpha \circ \mathsf{S}(\pi_1 + D_v \circ \pi_2) \circ \partial_A  &= \alpha \circ \mathsf{S}\left(\pi_1 + \pi_4 \right) \circ \partial_{A \times A} \circ \mathsf{S}\left(v\right) \\
&= \pi_2 \circ \left\langle \alpha \circ \mathsf{S}(\pi_1), \alpha \circ \mathsf{S}\left( [\pi_1, \pi_2] \right) \circ \partial_{A \times A} \right \rangle \circ \mathsf{S}(v) \\&= \pi_2 \circ v \circ \alpha =  D_v \circ \alpha 
\end{align*}
So $D_v$ is an $\mathsf{S}$-derivation. Conversely, let $D \in \mathsf{DER}_{\mathsf{S}}(A, \alpha)$. We must first show that $v_D$ is an $\mathsf{S}$-algebra morphism from $(A,\alpha)$ to $\mathsf{T}(A,\alpha)$. First, note that as before, by definition we have: 
\begin{align*}
(\pi_1 + \pi_4) \circ (v_D \times v_D) = \pi_1 + D \circ \pi_2    
\end{align*}
So, we compute: 
\begin{align*}
\hskip-.5cm\left\langle \alpha \circ \mathsf{S}(\pi_1), \alpha \circ \mathsf{S}\left( \pi_1 + \pi_4 \right) \circ \partial_{A \times A} \right \rangle \circ \mathsf{S}(v_D)
&=\left\langle \alpha \circ \mathsf{S}(\pi_1) \circ \mathsf{S}\left(v_D \right), \alpha \circ \mathsf{S}\left( \pi_1 + \pi_4 \right) \circ \partial_{A \times A} \circ \mathsf{S}\left(v_D \right)  \right \rangle \\ 
&= \left\langle \alpha , \alpha \circ \mathsf{S}\left( \pi_1 + D \circ \pi_2  \right) \circ \partial_A   \right \rangle = \left\langle \alpha , D \circ \alpha   \right \rangle\\& = \langle 1_A, D \rangle \circ \alpha = v_D \circ \alpha
\end{align*}
So, $v_D: (A,\alpha) \to \mathsf{T}(A,\alpha)$ is an $\mathsf{S}$-algebra morphism. By definition, $\pi_1 \circ v_D = 1_A$, so $p_{(A,\alpha)} \circ v_D = 1_{(A,\alpha)}$. Thus, we conclude that $v_D$ is a vector field on $(A, \alpha)$ in $(\mathsf{ALG}_\mathsf{S}, \mathbb{T})$. Furthermore, these constructions are clearly inverses of each other, that is, $v_{D_v} = v$ and $D_{v_D} = D$. So, we conclude that $\mathsf{V}_{\mathbb{T}}(A,\alpha) \cong \mathsf{DER}_{\mathsf{S}}(A, \alpha)$. 

Lastly, since for tangent monads, the forgetful functor preserves the tangent structure strictly, and the Lie bracket is completely defined using the tangent structure, it follows that the forgetful functor also preserves the Lie bracket. This implies that the Lie bracket in the category of algebras must be the same as the Lie bracket in the base category, or in other words, the tangent monad ``lifts'' the Lie bracket. Therefore, if $\mathbb{X}$ is an additive category, then for $v, w \in \mathsf{V}_{\mathbb{T}}(A,\alpha)$, we have $[v,w] = \langle 1_A, (D_v\circ D_w) - (D_w \circ D_v) \rangle$ as desired. 
\end{proof}

By Lemma \ref{lem:adjoint-vf}, if $(\mathsf{ALG}_\mathsf{S}, \mathbb{T})$ has adjoint tangent structure, then the vector fields in $(\mathsf{ALG}^{op}_\mathsf{S}, \mathbb{T}^\circ)$ also correspond to $\mathsf{S}$-derivations. 

We note that $\mathsf{S}$-derivations also generalize the notion of derivations for codifferential categories, as defined by the third named author in \cite{lemay2019differential}. Indeed, every codifferential category comes equipped with a canonical coCartesian differential monad $\mathsf{S}$ \cite[Example 3.13]{ikonicoff2021cartesian}, and $\mathsf{S}$-derivations in the sense above are precisely the $\mathsf{S}$-derivations defined in \cite[Definition 5.1]{lemay2019differential}, where the latter generalize the notion of differential algebras in codifferential categories. 

\subsection{Differential Objects for coCartesian Differential Monads}\label{subsec:diffobj-ccdc}

In this section, we discuss differential objects in both the category of algebras of a coCartesian differential monad and its opposite category. In particular, we will explain how the free algebras of the monad are always differential objects in the opposite category. This was somewhat to be expected since the subcategory of free algebras is equivalent to the Kleisli category of the monad, whose opposite category is known to be a Cartesian differential category. 

First, observe that, if $\mathbb{X}$ is a semi-additive category, every object $A$ in $\mathbb{X}$ has a canonical and unique differential object structure in $(\mathbb{X}, \mathbb{B})$. Indeed, recall that in a semi-additive category, every object has a unique commutative monoid structure where the sum $\sigma: A \times A \to A$ sums the two components together, $\sigma := \pi_1 + \pi_2$, and the zero ${\zeta: \ast \to A}$ is simply the zero map $\zeta = 0$. It is straightforward to check that for every object $A$, $(A, \pi_1, \pi_1 + \pi_2, 0)$ is a differential object in $(\mathbb{X}, \mathbb{B})$. Now suppose that $(A, \hat{p}, \pi_1 + \pi_2, 0)$ is a differential object in $(\mathbb{X}, \mathbb{B})$, where note that the differential projection is of type $\hat{p}: A \times A \to A$. It then easily follows from the fact that since $\langle p, \hat{p} \rangle$ is an isomorphism, and compatibility of $\hat{p}$ with the vertical lift and zero, that the differential projection must be the second projection, so $\hat{p} = \pi_2$. Therefore we have that $\mathsf{DIFF}[(\mathbb{X}, \mathbb{B})] = \mathbb{X}$. 

Now, let $(\mathsf{S}, \mu, \eta, \partial)$ be a coCartesian differential monad on a semi-additive category $\mathbb{X}$. As explained in \cite[Section 4.3]{cockett2014differential}, strong Cartesian tangent morphisms send differential objects to differential objects. Therefore, since the forgetful functor preserves the Cartesian tangent structure strictly, it preserves differential objects. So, if $\left( (A, \alpha), \hat{p}, \sigma, \zeta \right)$ is a differential object in $(\mathsf{ALG}_\mathsf{S}, \mathbb{T})$ then $(A, \hat{p}, \sigma, \zeta)$ must also be a differential object in $(\mathbb{X}, \mathbb{B})$. As explained in the previous paragraph, this means that the differential object structure of $(A, \alpha)$ must be of the form $\hat{p} = \pi_2$, $\sigma := \pi_1 + \pi_2$, and $\zeta = 0$. So an $\mathsf{S}$-algbera $(A, \alpha)$ can have at most one differential object structure if and only if ${\pi_2: \mathsf{T}(A,\alpha) \to (A, \alpha)}$, $\pi_1 + \pi_2: (A, \alpha) \times (A, \alpha) \to (A, \alpha)$, and $0: (\ast, t_{\mathbb T\ast}) \to (A, \alpha)$ are $\mathsf{S}$-algebra morphisms. We may equivalently express this as follows: 

\begin{lemma}\label{lem:diffobj-Salg} Let $(\mathsf{S}, \mu, \eta, \partial)$ be a coCartesian differential monad on a semi-additive category $\mathbb{X}$. Then, an $\mathsf{S}$-algebra $(A, \alpha)$ has a (necessarily unique) differential object structure if and only if the following equalities hold: 
\begin{align}\label{diffobjSalgeq}
  \alpha = \alpha \circ \mathsf{S}(\pi_2) \circ \partial_A && \alpha \circ \mathsf{S}(\pi_1 + \pi_2) = \alpha \circ \mathsf{S}(\pi_1) + \alpha \circ \mathsf{S}(\pi_2) && \alpha \circ \mathsf{S}(0) = 0
\end{align}
\end{lemma}
\begin{proof} For the $\Rightarrow$ direction, suppose that $\left( (A, \alpha), \pi_2, \pi_1 + \pi_2, 0 \right)$ is a differential object. That $\pi_2: \mathsf{T}(A,\alpha) \to (A, \alpha)$ is a $\mathsf{S}$-algebra morphism implies that $\alpha \circ \mathsf{S}(\pi_2) = \alpha \circ \mathsf{S}(\pi_1 + \pi_4) \circ \partial_A$. Pre-composing both sides by $\mathsf{S}(\langle 0,1_A \rangle)$ we get $\alpha = \alpha \circ \mathsf{S}(\pi_2) \circ \partial_A$. On the other hand, that $\pi_1 + \pi_2: (A, \alpha) \times (A, \alpha) \to (A, \alpha)$ and $0: (\ast, t_{\mathbb{T}\ast}) \to (A, \alpha)$ are $\mathsf{S}$-algebra morphisms immediately imply the two other equalities of (\ref{diffobjSalgeq}). 

Conversely, for the $\Leftarrow$ direction, assume that the equations of (\ref{diffobjSalgeq}) hold. Per the above discussion, we need to show that ${\pi_2: \mathsf{T}(A,\alpha) \to (A, \alpha)}$, $\pi_1 + \pi_2: (A, \alpha) \times (A, \alpha) \to (A, \alpha)$, and $0: (\ast, t_{\mathbb{T}\ast}) \to (A, \alpha)$ are all $\mathsf{S}$-algebra morphisms. However, the second and third equality of (\ref{diffobjSalgeq}) immediately imply that $\pi_1 + \pi_2$ and $0$ are $\mathsf{S}$-algebra morphisms. To show that $\pi_1$ is also an $\mathsf{S}$-algebra morphism, first note that the second equality of (\ref{diffobjSalgeq}) implies that: 
\begin{align*}
     \alpha \circ \mathsf{S}(\pi_1 + \pi_4) = \alpha \circ \mathsf{S}(\pi_1) + \alpha \circ \mathsf{S}(\pi_4) 
\end{align*}
And note that since $\pi_4 = \pi_2 \circ (\pi_2 \times \pi_2)$, using the naturality of $\partial$, we get that: 
\begin{align*}
      \mathsf{S}(\pi_4) \circ \partial_{A \times A} = \mathsf{S}(\pi_2) \circ \partial_A \circ \mathsf{S}(\pi_2) 
\end{align*}
Then using these identities, first equality of (\ref{diffobjSalgeq}), and \textbf{[DC.1]}, we compute: 
\begin{align*}
\pi_2 \circ (\alpha \times \alpha) \circ \lambda_A &= \alpha \circ \pi_2 \circ \lambda_A = \alpha \circ \mathsf{S}(\pi_1 + \pi_4) \circ \partial_{A \times A} = \alpha \circ \mathsf{S}(\pi_1) \circ \partial_{A \times A} + \alpha \circ \mathsf{S}(\pi_4) \circ \partial_{A \times A} \\
&= 0 + \alpha \circ \mathsf{S}(\pi_2) \circ \partial_A \circ \mathsf{S}(\pi_2) = \alpha \circ \mathsf{S}(\pi_2)
\end{align*}
  So $\pi_2$ is an $\mathsf{S}$-algebra morphism. So we conclude that $\left( (A, \alpha), \pi_2, \pi_1 + \pi_2, 0 \right)$ is a differential object.
\end{proof}

Furthermore, the fact that $\pi_1: \mathsf{T}(A,\alpha) \to (A, \alpha)$ is an $\mathsf{S}$-algebra morphism actually implies that we have an equality $\mathsf{T}(A,\alpha) = (A, \alpha) \times (A, \alpha)$ on the nose. This is quite a strong requirement. So, one should not expect many differential objects in $(\mathsf{ALG}_\mathsf{S}, \mathbb{T})$. However, the terminal object $(\ast, 1_\ast)$ is always a differential object. 

Let us turn our attention to differential objects in the opposite category. Again, observe first that every object $A$ in a semi-additive category $\mathbb{X}$ also has a unique differential object structure in $(\mathbb{X}^{op}, \mathbb{B}^\circ)$. Viewed in $\mathbb{X}$, the differential projection $\hat{p}^\circ: A \to A \times A$ is the injection into the second component, $\hat{p}^\circ = \langle 0, 1_A \rangle$, the sum $\sigma^\circ: A \to A\times A$ is the copy map, $\sigma^\circ = \langle 1_A, 1_A \rangle$, and the zero ${\zeta^\circ: A \to \ast}$ is the zero map in the other direction, $\zeta^\circ = 0$. So, $(A, \hat{p}^\circ, \sigma^\circ, \zeta^\circ)$ is a differential object in $(\mathbb{X}^{op}, \mathbb{B}^\circ)$. So, $\mathsf{DIFF}[(\mathbb{X}^{op}, \mathbb{B}^\circ)] = \mathbb{X}^{op}$. 

Suppose now that $\mathsf{ALG}_\mathsf{S}$ has all reflexive coequalizers and finite coproducts. Since $(\mathsf{F}^{\mathsf{S}}, \tau)$ is a strong Cartesian tangent morphism, $(\mathsf{F}^{\mathsf{S}}, \tau)$ will map $(A, \hat{p}^\circ, \sigma^\circ, \zeta^\circ)$ to a differential object in $(\mathsf{ALG}^{op}_\mathsf{S}, \mathbb{T}^\circ)$ whose underlying $\mathsf{S}$-algebra is the free $\mathsf{S}$-algebra over $A$. Therefore, we see that $\mathsf{KL}^{op}_\mathsf{S}$ is a sub-Cartesian differential category of $\mathsf{DIFF}[(\mathsf{ALG}^{op}_\mathsf{S}, \mathbb{T}^\circ)]$.
\begin{proposition}\label{prop:diffobg-stanopp} Let $(\mathsf{S}, \mu, \eta, \partial)$ be a coCartesian differential monad on a semi-additive category $\mathbb{X}$, and suppose that $\mathsf{ALG}_\mathsf{S}$ has all reflexive coequalizers and finite coproducts. Then, for every object $A$ in $\mathbb{X}$, the free $\mathsf{S}$-algebra over $A$ has a differential object structure in $(\mathsf{ALG}^{op}_\mathsf{S}, \mathbb{T}^\circ)$.\\ Explicitly, $\left( \left( \mathsf{S}(A), \mu_A \right),\tau^{-1}_A \circ \mathsf{S}(\langle 0,1_A \rangle), \theta^{-1}_{A,A} \circ \mathsf{S}(\langle 1_A, 1_A \rangle), \mathsf{S}(0) \right)$ is a differential object in $(\mathsf{ALG}^{op}_\mathsf{S}, \mathbb{T}^\circ)$, where the composition in the quadruple is the one of $\mathbb{X}$. 
\end{proposition}  

For an arbitrary coCartesian differential monad, there could be other differential objects that are not free $\mathsf{S}$-algebras. In future work, it would be interesting to give a precise characterization of the coCartesian differential monads $\mathsf{S}$ whose differential objects coincide with free $\mathsf{S}$-algebras. 

We conclude by restating this in terms of Cartesian differential comonads. 

\begin{proposition}\label{proposition:free-algebras-are-differential-objects}For a Cartesian differential comonad on a semi-additive category such that the coEilenberg--Moore category has reflexive coequalizers and finite coproducts, every cofree coalgebra is a differential object in the coEilenberg--Moore category. 
\end{proposition}

\section{The Tangent Categories of Algebras of an Operad}
\label{sec:operads}

In this section, we build the main constructions of this paper: a tangent structure on the category of algebras over an operad, and on its opposite category. To do so, we will first show that the associated monad of an operad is a coCartesian differential monad. By using the results of Section \ref{sec:ccdm}, we then obtain a tangent structure on the category of algebras over an operad. We will explain how the tangent bundle is given by the semi-direct product, generalizing the tangent bundle given by dual numbers for commutative algebras. We will then show that we also have adjoint tangent structure, which gives us tangent structure for the opposite category of algebras over an operad. This time, the tangent bundle is given by the free algebra over the module of K\"ahler differentials, generalizing the tangent bundle of affine schemes. We will also discuss vector fields and differential objects in these tangent categories, and explain how they correspond respectively to derivations and certain modules. Lastly, since every operad gives a coCartesian differential monad, we also take a closer look at the induced Cartesian differential category. We will explain how this can intuitively be thought of as a Lawvere theory of polynomials over the operad. Throughout this section, we will also review the necessary concepts from the theory of operads as needed. 

In this paper, we work with operads in the category of modules over a commutative, unital ring. For a detailed introduction, we refer to \cite{kriz95,LV}. We note that the notion of an operad can be defined in the more general framework of a symmetric monoidal category with all small limits and colimits, such that the colimits preserve the monoidal structure \cite[Section 1.1.1]{Fresse-ModulesOperads}. We suspect that some of our constructions can be readily extended to the case where the base category also has finite biproducts. We leave this exploration for future work. 

\subsection{The coCartesian Differential Monad of an Operad}
\label{subsection:cCD-monad-for-operads}

From any operad, there is a canonical way to construct a monad. The algebras over the operad are, by definition, the algebras over this monad. The objective of this section is to prove that for any operad, said monad is in fact always a coCartesian differential monad. By the results of Section \ref{sec:ccdm}, it then follows that the category of algebras over an operad is a tangent category and that the opposite category of the Kleisli category of an operad is a Cartesian differential category. 

To give a coCartesian differential monad, we must first fix our (semi-)additive category. For the remainder of this section, we fix $R$ to be a commutative ring and we denote by $\mathsf{MOD}_R$ the category of $R$-modules and $R$-linear maps between them. It is well known that $\mathsf{MOD}_R$ is an additive category. We show that every operad induces a canonical coCartesian differential monad on $\mathsf{MOD}_R$. 

Throughout this paper, by an operad we mean a symmetric algebraic operad, which means an operad in $R$-modules that allows for permutations of arguments. This latter part is captured by actions of the symmetric group. For each $n \in \mathbb{N}$, we denote the symmetric group on $n$ letters by $\Sigma(n)$. More concretely, an \textbf{operad} is a sequence $\P=(\P(n))_{n\in\mathbb N}$ of $R$-modules such that:
\begin{enumerate}[{\em (i)}]
\item There is a distinguished element $1_\P\in\P(1)$;
\item For every $n$, there is a right action of $\Sigma(n)$ on $\P(n)$, which we denote by $\mu \cdot \sigma$ for all $\mu \in \P(n)$ and $\sigma \in \Sigma(n)$;
\item For every $n$ and $m$, there is a family of $R$-linear maps $\circ_i:\P(m)\otimes\P(n)\to\P(m+n-1)$ for all $i$ with $1\le i\le m$, called the \textbf{partial compositions}, which we denote by $\mu \circ_i \nu := \circ_i(\mu \otimes \nu)$. 
\end{enumerate}
The partial compositions are required to satisfy natural equivariance and associativity conditions, and $1_\P$ is required to play the role of a unit with respect to partial compositions, see \cite[Section 5.3.4]{LV} for full details. Using the partial compositions, we can also define the $R$-linear maps $\circ: \P(k)\otimes \P(n_1)\otimes\cdots\otimes \P(n_k)\to\P(n_1+\dots+n_k)$ for all $k$ and $n_i$, called the \textbf{complete composition}, which is defined on pure tensors by:
\[\circ(\mu\otimes \nu_1\otimes\dots\otimes\nu_k)=(\cdots((\mu\circ_k\nu_k)\circ_{k-1}\nu_{k-1})\cdots)\circ_1\nu_1,\] 
and then extended by $R$-linearity. As a shorthand, we denote:
\[\mu(\nu_1,\dots,\nu_k) := \circ(\mu\otimes \nu_1\otimes\dots\otimes\nu_k)\] 
After Theorem \ref{theorem:opdiffmonad} below, we give some well-known examples of operads. 

We now describe the monad associated to an operad $\P$ \cite[Section 5.1.2]{LV}. First, for any $R$-module $V$, define the $R$-module $\mathsf{S}(\P,V)$ as the coproduct of all $\P(n) \otimes V^{\otimes n}$ quotiented by $\Sigma(n)$, where the action is diagonal and $\Sigma(n)$ acts on $V^{\otimes n}$ by permuting the factors of the tensor power:
\begin{align*}
\mathsf{S}(\P,V)=\bigoplus_{n\in\mathbb N} \left( \P(n) \otimes V^{\otimes n} \right)_{\Sigma(n)}.
\end{align*}
As a shorthand, we denote the equivalence class of a pure tensor as follows:
\begin{align*}
(\mu;v_1,\dots,v_{n}) :=\left[ \mu\otimes v_1\otimes\cdots\otimes v_{n} \right] \in \left( \P(n) \otimes V^{\otimes {n}} \right)_{\Sigma(n)},  &&\text{where } \mu \in \P(n), v_i \in V. 
\end{align*}
Observe that, for any $R$-linear morphism with domain $\mathsf{S}(\P, V)$, it suffices to define said morphism on elements of the form $(\mu;v_1,\dots,v_{n})$, making sure the definition respects the action of the symmetric group, and then extend by $R$-linearity. With this in mind, define the functor $\mathsf{S}(\P,-): \mathsf{MOD}_R \to \mathsf{MOD}_R$ which maps an $R$-module $V$ to $\mathsf{S}(\P,V)$, and sends an $R$-linear morphism ${f: V \to W}$ to the $R$-linear morphism ${\mathsf{S}(\P, f): \mathsf{S}(\P, V) \to \mathsf{S}(\P, W)}$ which is defined as follows: 
\begin{align*}
\mathsf{S}(\P, f)(\mu;v_1,\dots,v_{n}) = (\mu;f(v_1),\dots,f(v_{n}))
\end{align*}
The monad unit $\eta_V: V \to   \mathsf{S}(\P,V)$ and the monad multiplication $\gamma_V: S\left(\P, \mathsf{S}(\P,V) \right) \to \mathsf{S}(\P,V)$ are respectively defined as follows: 
\[
\eta_V(v) = (1_\P; v) \in \P(1) \otimes V,
\]
\[
\gamma_V\left( \mu; \left(\nu_1; v_{1,1}, \hdots, v_{1,n_1} \right), \hdots, \left( \nu_{k}; v_{k,1}, \hdots, v_{k,n_k}  \right) \right) = \left( \mu \left( \nu_1, \hdots, \nu_{k}  \right); v_{1,1}, \hdots, v_{k,n_k} \right).
\]
Then, $(\mathsf{S}(\P,-),  \gamma, \eta)$ is a monad on $\mathsf{MOD}_R$ \cite[Section 3]{kriz95}. We will now prove that this monad is in fact also a coCartesian differential monad.

\begin{thm}\label{theorem:opdiffmonad} Let $\P$ be an operad. Let $\partial_V: \mathsf{S}(\P, V) \to \mathsf{S}(\P, V \times V)$ be the $R$-linear map such that: 
\begin{equation*}\begin{gathered} 
\partial_V\left( \mu; v_1, \hdots, v_{n} \right) = \sum_{i=1}^{n}\left(\mu; (v_1, 0),\hdots,(0, v_i),\hdots,(v_{n},0) \right).
\end{gathered}\end{equation*}
Here, in the sum, we use the first injection $V\to V\times V$ to all the inputs in $V$ except for the $i$-th input, for which we use the second injection. Then, $\partial$ is a differential combinator transformation for $(\mathsf{S}(\P,-),  \gamma, \eta)$, and thus $(\mathsf{S}(\P,-),  \gamma, \eta, \partial)$ is a coCartesian differential monad. 
\end{thm}
\begin{proof}
It is clear that $\partial$ is a natural transformation. Therefore, we must prove that \textbf{[DC.1]} to \textbf{[DC.6]} in Definition \ref{defi:cCDM} hold. First observe that for any $R$-linear morphism $f: V \times V \to W$, we have: 
\begin{align*}
\mathsf{S}(\P, f)\left( \partial_V\left( \mu; v_1, \hdots, v_{n} \right) \right) = \sum_{i=1}^{n}\left(\mu; f(v_1, 0),\hdots,f(0, v_i),\hdots,f(v_{n},0) \right) .
\end{align*}
This will help simplify our calculations. 

\begin{enumerate}[{\bf [DC.1]}]
\item Here, we use the fact that $(\mu; v_1, \hdots, 0, \hdots, v_{n}) =0$: 
\begin{gather*}
\mathsf{S}(\P, \pi_1)\left( \partial_V\left( \mu; v_1, \hdots, v_{n} \right) \right) =  \sum_{i=1}^{n}\left(\mu; v_1,\hdots,0,\hdots,v_{n} \right)=0 
\end{gather*}
So $\mathsf{S}(\P, \pi_1) \circ \partial_V = 0$
\item Here we use that $(\mu; v_1, \hdots, v_i + v_i^\prime, \hdots, v_{n}) = (\mu; v_1, \hdots, v_i, \hdots, v_{n}) + (\mu; v_1, \hdots, v_i^\prime, \hdots, v_{n})$:\\
\begin{multline*}
\mathsf{S}(\P,\langle \pi_1, \pi_2, \pi_2 \rangle)\left( \partial_V\left( \mu; v_1, \hdots, v_{n} \right) \right) = \sum_{i=1}^{n}\left(\mu; (v_1, 0,0) ,\hdots,(0,v_i, v_i),\hdots,(v_{n},0,0) \right) \\
= \sum_{i=1}^{n}\left(\mu; (v_1, 0,0) ,\hdots,(0,v_i, 0),\hdots,(v_{n},0,0) \right) \\\qquad+ \sum_{i=1}^{n}\left(\mu; (v_1, 0,0) ,\hdots,(0,0, v_i),\hdots,(v_{n},0,0) \right) \\
= \mathsf{S}(\P,\langle \pi_1, \pi_2, 0\rangle)\left( \partial_V\left( \mu; v_1, \hdots, v_n \right) \right) + \mathsf{S}(\P,\langle \pi_1, 0, \pi_2 \rangle)\left( \partial_V\left( \mu; v_1, \hdots, v_n \right) \right),
\end{multline*}
so $\mathsf{S}(\P,\langle \pi_1, \pi_2, \pi_2 \rangle) \circ \partial_V = \mathsf{S}(\P,\langle \pi_1, \pi_2, 0 \rangle) \circ \partial_V + \mathsf{S}(\P,\langle \pi_1, 0, \pi_2\rangle) \circ \partial_V$
\item This is just the case for $n=1$ which says that $\partial_V(\mu;v) = (\mu; (0,v))$. So $\partial_V \circ \eta_V = \eta_{V \times V} \circ \langle 0, 1_V \rangle$. 
\item We can start with an element of the form $\left( \mu; \left(\nu_1; v_{1,1}, \hdots, v_{1,n_1} \right), \hdots, \left( \nu_{k}; v_{k,1}, \hdots, v_{k,n_{k}}  \right) \right)$. On the one hand, we have that: 
\begin{multline*}
\partial_V\left( \gamma_V\left( \mu; \left(\nu_1; v_{1,1}, \hdots, v_{1,n_1} \right), \hdots, \left( \nu_{k}; v_{k,1}, \hdots, v_{k,n_{k}}  \right) \right) \right) 
\\= \partial_V\left( \left( \mu \left( \nu_1, \hdots, \nu_{k}  \right); v_{1,1}, \hdots, v_{k,n_{k}} \right) \right) \\
= \sum\limits^{k}_{i=1} \sum^{n_i}_{j_i=1} \left(\mu \left( \nu_1, \hdots, \nu_{k}  \right); (v_{1,1}, 0),\hdots,(0, v_{i,j_i}),\hdots,(v_{k,n_{k}},0) \right).
\end{multline*}
On the other hand, we have that:
\begin{multline*}
 \partial_{\mathsf{S}(\P,V)}\left( \mu; \left(\nu_1; v_{1,1}, \hdots, v_{1,n_1} \right), \hdots, \left( \nu_{k}; v_{k,1}, \hdots, v_{k,n_{k}}  \right) \right) \\
\hskip-.4cm= \sum\limits^{k}_{i=1} \Bigg(  \mu; \left( \left(\nu_1; v_{1,1}, \hdots, v_{1,n_1} \right), 0 \right), \hdots, \left(0, \left(\nu_i; v_{i,1}, \hdots, v_{i,n_i} \right) \right), \hdots \left( \left( \nu_{k}; v_{k,1}, \hdots, v_{k,n_{k}}  \right), 0 \right) \Bigg).
\end{multline*}

\noindent Then applying $\mathsf{S}\left(\P, \mathsf{S}(\P,\langle 1_V, 0 \rangle) \circ \pi_1 + \partial_V \circ \pi_2 \right)$, using the multilinearity in the module arguments, we get:
\begin{multline*}
\sum\limits^{k}_{i=1} \sum^{n_i}_{j_i=1} \bigg( \mu; \left(\nu_1; (v_{1,1},0), \hdots, (v_{1,n_1},0) \right),\\
\hdots, \left(\nu_i; (v_{i,1},0), \hdots, (0,v_{i,j_i}), \hdots, (v_{i,n_i},0)  \right),\hdots, \left( \nu_{k}; (v_{k,1},0), \hdots, (v_{k,n_{k}},0)  \right) \bigg).
\end{multline*}
Then, finally applying $\gamma_{V\times V}$, we get: 
\begin{align*}
\sum\limits^{k}_{i=1} \sum^{n_i}_{j_i=1} \left(\mu \left( \nu_1, \hdots, \nu_{k}  \right); (v_{1,1}, 0),\hdots,(0, v_{i,j_i}),\hdots,(v_{k,n_{k}},0) \right).
\end{align*}
So, $\partial_V \circ \gamma_V = \gamma_{V \times V} \circ S\left(\P, \mathsf{S}(\P,\langle 1_V, 0 \rangle) \circ \pi_1 + \partial_V \circ \pi_2 \right) \circ \partial_{\mathsf{S}(\P,V)}$. 
\end{enumerate}
For the remaining two relations, let us first expand out $\partial_{V \times V} \circ \partial_V$:
\begin{multline*}
\partial_{V\times V}\left( \partial_V\left( \mu; v_1, \hdots, v_{n} \right) \right)\\= \sum_{i=1}^{n} \sum_{1\leq j <i } \left(\mu; (v_1, 0, 0,0),\hdots,(0,0, v_j,0) \hdots, (0, v_i, 0,0 ),\hdots,(v_{n},0,0,0) \right) \\
+ \sum_{i=1}^{n} \left(\mu; (v_1, 0, 0,0),\hdots,(0,0, 0, v_i) \hdots, ,(v_{n},0,0,0) \right) \\
+ \sum_{i=1}^{n} \sum_{i < j \leq n } \left(\mu; (v_1, 0, 0,0), \hdots, (0, v_i, 0,0 ), \hdots,(0,0, v_j,0), \hdots,(v_{n},0,0,0) \right) .
\end{multline*}

\begin{enumerate}[{\bf [DC.1]}]
\setcounter{enumi}{4}
\item Here we again use that $(\mu; v_1, \hdots, 0, \hdots, v_n) =0$: 
\begin{multline*}
\mathsf{S}(\langle \pi_1, \pi_4 \rangle)\left(  \partial_{V\times V}\left( \partial_V\left( \mu; v_1, \hdots, v_{n} \right) \right) \right)\\ =\sum_{i=1}^{n} \sum_{1 \leq j < i} \left(\mu; (v_1, 0),\hdots,(0,0) \hdots, (0,0 ),\hdots,(v_{n},0) \right) \\
+ \sum_{i=1}^{n} \left(\mu; (v_1, 0),\hdots,(0, v_i) \hdots, ,(v_{n},0) \right)  \\+ \sum_{i=1}^{n} \sum_{i < j \leq n } \left(\mu; (v_1,0), \hdots, (0,0 ), \hdots,(0,0), \hdots,(v_{n},0,0,0) \right) \\
= 0 + \partial_V\left( \mu; v_1, \hdots, v_{n} \right) + 0 = \partial_V\left( \mu; v_1, \hdots, v_{n} \right),
\end{multline*}

so $\mathsf{S}(\langle \pi_1, \pi_4 \rangle) \circ \partial_{V \times V} \circ \partial_V = \partial_V$.
\item This amounts to a simple reindexing by swapping $i$ and $j$: 
\begin{align*}
& S\left(\P, \left \langle \pi_1, \pi_3, \pi_2, \pi_4 \right \rangle \right) \left( \partial_{V\times V}\left( \partial_V\left( \mu; v_1, \hdots, v_{n} \right) \right) \right) =\\
&= \sum_{i=1}^{n} \sum_{1 \leq j < i} \left(\mu; (v_1, 0, 0,0),\hdots,(0,v_j, 0,0) \hdots, (0, 0, v_i,0 ),\hdots,(v_{n},0,0,0) \right) \\
&+ \sum_{i=1}^{n} \left(\mu; (v_1, 0, 0,0),\hdots,(0,0, 0, v_i) \hdots, ,(v_{n},0,0,0) \right) \\
&+ \sum_{i=1}^{n} \sum_{i < j \leq n } \left(\mu; (v_1, 0, 0,0), \hdots, (0, v_i, 0,0 ), \hdots,(0,0, v_j,0), \hdots,(v_{n},0,0,0) \right) \\
&= \sum_{j=1}^{n} \sum_{1 \leq i < j} \left(\mu; (v_1, 0, 0,0),\hdots,(0,0, v_i,0) \hdots, (0, v_j, 0,0 ),\hdots,(v_{n},0,0,0) \right) \\
&+ \sum_{j=1}^{n} \left(\mu; (v_1, 0, 0,0),\hdots,(0,0, 0, v_j) \hdots, ,(v_{n},0,0,0) \right)\\
&+ \sum_{j=1}^{n} \sum_{j < i \leq n } \left(\mu; (v_1, 0, 0,0), \hdots, (0, v_j, 0,0 ), \hdots,(0,0, v_i,0), \hdots,(v_{n},0,0,0) \right) \\
&=\partial_{V\times V}\left( \partial_V\left( \mu; v_1, \hdots, v_{n} \right) \right),
\end{align*} 
so $S\left(\P, \left \langle \pi_1, \pi_3, \pi_2, \pi_4 \right \rangle \right) \partial_{V \times V} \circ \partial_V= \partial_{V \times V} \circ \partial_V$.
\end{enumerate}
We conclude that $\partial$ is a differential combinator transformation and that $(\mathsf{S}(\P,-),  \gamma, \eta, \partial)$ is a coCartesian differential monad. 
\end{proof}

Here are now some well-known examples of operads, their associated monad and resulting differential combinator transformation: 

\begin{example}
\label{example:operadfromalg}
Any (unital and associative) $R$-algebra $A$ induces an operad $A^\bullet$ where $A^\bullet(1)=A$ and $A^\bullet(n)=\mathsf{0}$ for $n \neq 1$. The associated monad is given by the free (left) $A$-module monad\footnote{To avoid confusion, the algebras of the operad $A^\bullet$ will simply be (left) $A$-modules and not $A$-algebras, as we will see in Example \ref{example:Abullet-modules}.}, that is, $\mathsf{S}(A^\bullet, V) = A \otimes V$. See \cite[Example 0, 5.2.10]{LV} for full details. The differential combinator transformation $\partial_V: A \otimes V \to A \otimes (V \times V)$ is simply given by injecting $V$ into the second component, so $\partial_V(a \otimes v) = a \otimes (0,v)$. 
\end{example}  

\begin{example}\label{example:operadCom}
The operad $\Com$ is defined by $\Com(n)=R$ for all $n$, with trivial action of the symmetric group. Unit and compositions are defined by identities of $R$. The associated monad is the symmetric algebra monad:
\[\mathsf{S}(\Com,V)=\mathsf{Sym}(V) = \bigoplus_{n \in \mathbb{N}} \left(V^{\otimes^n}\right)_{\Sigma(n)}\] 
See \cite[Example 2, 5.2.10]{LV} for full details. The differential combinator transformation $\partial_V: \mathsf{Sym}(V) \to \mathsf{Sym}(V \times V)$ is defined on pure symmetrized tensors as follows: 
\[\partial_V([v_1 \otimes \hdots \otimes v_{n}]) = \sum^{n}_{i=1} [(v_1, 0) \otimes \hdots \otimes (0,v_i) \otimes \hdots \otimes (v_{n},0)].\]
Now, if $V$ is a free $R$-module with basis $X$, then $\mathsf{Sym}(V)$ is isomorphic as an $R$-algebra to the polynomial $R$-algebra over $X$, $\mathsf{Sym}(V) \cong R[X]$. Also, $\mathsf{Sym}(V \times V)$ is isomorphic to the polynomial $R$-algebra over the disjoint union of $X$ with itself. So, writing $dX= \lbrace dx \vert~ \forall x \in X \rbrace$ to distinguish between the first copy and the second copy of $X$, we have that $\mathsf{Sym}(V \times V) \cong R[X, dX]$. In terms of polynomials, $\partial_V: R[X] \to R[X, dX]$ maps a polynomial to the sum of its partial derivatives:
\[\partial_V(p(\vec x)) = \sum^n_{i=1} \frac{\partial p(\vec x)}{\partial x_i} dx_i.\] 
Therefore, $\partial$ recaptures polynomial differentiation. 
\end{example}

\begin{example}\label{example:operadAss}
There is an operad $\Ass$ defined by $\Ass(n)=R[\Sigma(n)]$, the regular representation of the group $\Sigma(n)$. See \cite[Example 1, 5.2.10]{LV} for full details. The associated monad is the tensor algebra monad:
\[\mathsf{S}(\Ass,V)=\mathsf{Ten}(V) = \bigoplus_{n \in \mathbb{N}} V^{\otimes n}.\] 
The differential combinator transformation $\partial_V: \mathsf{Ten}(V) \to \mathsf{Ten}(V \times V)$ is defined on pure tensors as follows: 
\[\partial_V(v_1 \otimes \hdots \otimes v_{n}) = \sum^{n}_{i=1} (v_1, 0) \otimes \hdots \otimes (0,v_i) \otimes \hdots \otimes (v_{n}, 0).\] 
Now, if $V$ is a free $R$-module with basis $X$, then $\mathsf{Ten}(V)$ is isomorphic to the $R$-algebra of non-commutative polynomials over $X$. As such $\partial$ corresponds to differentiating non-commutative polynomials. 
\end{example}

\begin{example}\label{example:operadLie} There is an operad $\Lie$ whose associated monad is given by the free Lie algebra monad, $\mathsf{S}(\Lie, V) = \mathsf{Lie}(V)$. See \cite[Example 3, 5.2.10]{LV} for full details. In particular, $\mathsf{Lie}(V)$ is spanned by elements of the form: $\left[ v_1, [v_2,\hdots [v_{n-1}, v_{n}] \hdots] \right]$, so the Lie bracket of the Lie bracket of etc. of elements $v_i \in V$. The differential combinator transformation $\partial_V: \mathsf{Lie}(V) \to \mathsf{Lie}(V \times V)$ is defined on pure Lie brackets as follows:
\[\partial_V\left( \left[ v_1,[v_2, \hdots [v_{n-1}, v_{n}] \hdots] \right] \right) = \sum^{n}_{i=1} \left[ (v_1,0),\left[(v_2,0) ,\hdots \left[(0,v_i), \hdots [(v_{n-1},0), (v_{n},0)] \hdots\right] \right] \right].\]
Therefore, $\partial$ corresponds to differentiating Lie bracket polynomials.
\end{example}

\subsection{The Cartesian Differential Categories of an Operad}\label{subsec:CDC-operad}

The first consequence of Theorem \ref{theorem:opdiffmonad} is that the opposite category of the Kleisli category of a monad associated to an operad is a Cartesian differential category. As a shorthand, for an operad $\P$ we denote $\mathsf{KL}_\P := \mathsf{KL}_{\mathsf{S}(\P,-)}$ for the Kleisli category of $(\mathsf{S}(\P,-), \gamma, \eta)$. So we may state that: 

\begin{proposition} Let $\P$ be an operad. Then $\mathsf{KL}_\P ^{op}$, is a Cartesian differential category. 
\end{proposition} 

Let us unpack this Cartesian differential category.  Recall that the objects of $\mathsf{KL}^{op}_\P$ are $R$-modules, while a map $f: V \to W$ in $\mathsf{KL}^{op}_\P$ is an $R$-linear morphism of type $f: W \to \mathsf{S}(\P,V)$. The derivative ${\mathsf{D}[f]: V \times V \to W}$ in $\mathsf{KL}^{op}_\P$ is the $R$-linear morphism of type ${\mathsf{D}[f] : W \to \mathsf{S}(\P, V \times V)}$ defined as:
\[\mathsf{D}[f] = \partial_V \circ f.\] 
Let us give a bit of intuition about this. Let $X$ be a basis for a free $R$-module $V$. Then $(\mu;x_1, \hdots, x_{n})$ can be interpreted as a sort of monomial of degree $n$, which we call $\P$-monomials. With this in mind, an arbitrary element $P \in \mathsf{S}(\P, V)$ is therefore a finite sum of $\P$-monomials, and therefore we may interpret $P$ as a $\P$-polynomial. Now $(\mu; (x_1,0), \hdots, (0,x_i), \hdots, (x_{n},0))$ should be understood as the partial derivative of $(\mu;x_1, \hdots, x_{n})$ in the variable $x_i$. Therefore, the differential combinator transformation $\partial_V:\mathsf{S}(\P,V)\to \mathsf{S}(\P,V\times V)$ maps $\P$-polynomials to the sum of their partial derivatives, which we suggestively write as:
\[\partial_V(P)=\sum_{x\in X}\frac{\partial P}{\partial x}dx,\]
where the sum is well-defined since $P$ only depends on a finite number of elements of $X$. In other words, $\partial_V$ maps $\P$-polynomials to their total derivative. Now let us extend this intuition to the Kleisli category. If $W$ is another free $R$-module with basis $Y$, then a map $f: V \to W$ in $\mathsf{KL}^{op}_\P$ is precisely associated to a tuple of $\P$-polynomials in variables $X$:
\[ f \equiv \langle P_{y} \rangle_{y \in Y}.\] 
So its derivative $\mathsf{D}[f]: V \times V \to W$ in $\mathsf{Kl}_{\P}^{op}$ is associated to the tuple of the total derivative of each $\P$-polynomial:
\[\mathsf{D}[f] \equiv \left \langle \sum_{x\in X}\frac{\partial P_{y}}{\partial x}dx \right \rangle_{y \in Y}.\] 
Therefore, $\mathsf{KL}^{op}_\P$ can naively be understood as a generalized Lawvere theory of $\P$-polynomial. We can obtain a legitimate Lawvere theory of $\P$-polynomials by defining $\P\text{-}\mathsf{POLY}$ to be the category whose objects are the natural numbers $n \in \mathbb{N}$ and where a map $n \to m$ is an $m$-tuple $\langle P_1, \hdots, P_m \rangle$ where $P_i \in \mathsf{S}(\P, R^n)$. Then $\P\text{-}\mathsf{POLY}$ is equivalent to the full subcategory of $\mathsf{KL}^{op}_\P$ of finite dimensional $R$-modules, where in particular $\P\text{-}\mathsf{POLY}(n,m) \cong \mathsf{MOD}_R\left(R^m, \mathsf{S}(\P, R^n) \right)$. Also, note that $\P\text{-}\mathsf{POLY}(n,1) = \mathsf{S}(\P,R^n)$. Therefore we have that: 

\begin{proposition} Let $\P$ be an operad. Then $\P\text{-}\mathsf{POLY}$ is a Cartesian differential category where in particular for a map $P: n \to m$, $P= \langle P_1, \hdots, P_m \rangle$, its derivative $\mathsf{D}[n]: n \times n \to m$ is defined as follows: 
\[\mathsf{D}\left[\langle P_1, \hdots, P_m \rangle \right] = \langle \partial_{R^n}(P_1), \hdots, \partial_{R^n}(P_m) \rangle.\] 
\end{proposition}

Here is the resulting Cartesian differential category for our main examples of operads. 

\begin{example}
\label{example:CDCalg}
For an $R$-algebra $A$, $A^\bullet\text{-}\mathsf{POLY}$ is equivalent to the Cartesian differential category of $A$-linear maps. So let $A\text{-}\mathsf{LIN}$ be the category whose objects are $n \in \mathbb{N}$ and where a map $f: n \to m$ is an $A$-linear morphism $f: A^n \to A^m$. Then $A\text{-}\mathsf{LIN}$ is a Cartesian differential category where the differential combinator is defined as $\mathsf{D}[f](\vec x, \vec y) = f(\vec y)$ \cite[Example 2.6]{ikonicoff2021cartesian}, and furthermore we have that $A^\bullet\text{-}\mathsf{POLY} \simeq A\text{-}\mathsf{LIN}$ as Cartesian differential categories.
\end{example}  

\begin{example}\label{example:CDCCom}
For the operad $\Com$, $\Com\text{-}\mathsf{POLY}$ recaptures precisely polynomial differentiation since it is equivalent to the Lawvere theory of polynomials over $R$, $R\text{-}\mathsf{POLY}$, which is one of the main examples of Cartesian differential categories \cite[Example 2.6]{ikonicoff2021cartesian}. Concretely, $R\text{-}\mathsf{POLY}$ is the category whose objects are $n \in \mathbb{N}$ and where a map ${P: n \to m}$ is a $m$-tuple of polynomials in $n$ variables, that is, $P = \langle p_1(\vec x), \hdots, p_m(\vec x) \rangle$ with $p_i(\vec x) \in R[x_1, \hdots, x_n]$. $R\text{-}\mathsf{POLY}$ is a Cartesian differential category where the differential combinator is given by the standard differentiation of polynomials, that is, for a map ${P: n \to m}$, with $P = \langle p_1(\vec x), \hdots, p_m(\vec x) \rangle$, its derivative $\mathsf{D}[P]: n \times n \to m$ is defined as the tuple of the sum of the partial derivatives of the polynomials $p_i(\vec x)$:
\begin{align*}
\mathsf{D}[P](\vec x, \vec y) := \left( \sum \limits^n_{i=1} \frac{\partial p_1(\vec x)}{\partial x_i} y_i, \hdots, \sum \limits^n_{i=1} \frac{\partial p_n(\vec x)}{\partial x_i} y_i \right), && \sum \limits^n_{i=1} \frac{\partial p_j (\vec x)}{\partial x_i} y_i \in R[x_1, \hdots, x_n, y_1, \hdots, y_n].
\end{align*} 
So we have that $\Com\text{-}\mathsf{POLY} \simeq R\text{-}\mathsf{POLY}$ as Cartesian differential categories. 
\end{example}

\begin{example}\label{example:CDCAss}
For the operad $\Ass$, $\Ass\text{-}\mathsf{POLY}$ captures instead differentiating non-commutative polynomials since it is equivalent to the Lawvere theory of non-commutative polynomials. 
\end{example}

\begin{example}\label{example:CDCLie} For the operad $\Lie$, $\Lie\text{-}\mathsf{POLY}$ is the cartesian differential which is given by the Lawvere theory of Lie bracket polynomials. 
\end{example}

We conclude this section by discussing the notion of a $\mathsf{D}$-linear counit for a coCartesian differential monad.

\begin{definition}\cite[Definition 3.8]{ikonicoff2021cartesian} For a coCartesian differential monad $(\mathsf{S}, \mu, \eta, \partial)$ on a semi-additive category $\mathbb{X}$, a \textbf{$\mathsf{D}$-linear counit} is a natural transformation $\mathcal{E}_A: \mathsf{S}(A) \to A$ in $\mathbb X$ such that the following equalities hold: 
\begin{description}
\item[{\bf [DU.1]}] $\mathcal{E}_A \circ \eta_A = 1_A$,
\item[{\bf [DU.2]}] $\eta_A \circ \mathcal{E}_A = \mathsf{S}(\pi_2) \circ \partial_A$.
\end{description}
\end{definition}

In a Cartesian differential category with a differential combinator $\mathsf{D}$, there is an important class of maps called the \textbf{$\mathsf{D}$-linear maps} \cite[Definition 2.2.1]{blute2009cartesian} which are maps $f$ such that $\mathsf{D}[f] = f \circ \pi_2$. For a coCartesian differential monad, the $\mathsf{D}$-linear maps in the opposite category of its Kleisli category correspond precisely to the maps in the base category if and only if the coCartesian differential monad has a $\mathsf{D}$-linear counit \cite[Proposition 3.11]{ikonicoff2021cartesian}. We will now show that for an operad $\P$, its associated monad has a $\mathsf{D}$-linear counit $\mathcal{E}_V: \mathsf{S}(\P,V) \to V$ if and only if $\P(1)$ is of dimension one (as an $R$-module). Essentially, note that $\P(1) \otimes V \subset \mathsf{S}(\P, V)$, and therefore, if $\P(1) \cong R$, then there is a copy of $V$ inside $\mathsf{S}(\P, V)$. So the $\mathsf{D}$-linear counit amounts to projecting out the $V$ component of $\mathsf{S}(\P, V)$.

\begin{lemma} Let $\P$ be an operad. Then $(\mathsf{S}(\P,-),  \gamma, \eta, \partial)$ has a $\mathsf{D}$-linear counit $\mathcal{E}_V: \mathsf{S}(\P,V) \to V$ if and only if the $R$-linear morphism $e_\P: R \to \P(1)$ which picks out the distinguished element, $e_\P(1) = 1_\P$, is an isomorphism. Explicitly, if $e_\P$ is an isomorphism, define $\mathcal{E}_V: \mathsf{S}(\P,V) \to V$ as follows: 
\begin{align*}
 \mathcal{E}_V(\mu;v) = e^{-1}_\P(\mu) \cdot v \mbox{ if }\mu\in\P(1),&& \mathcal{E}_V(\mu;v_1, \hdots, v_{n}) = 0\mbox{ if }n\neq 1,
\end{align*}
and conversely if $\mathcal{E}_V: \mathsf{S}(\P,V) \to V$ is a $\mathsf{D}$-linear counit, then define $e^{-1}_\P: \P(1) \to R$ as $e^{-1}_\P(\mu) = \mathcal{E}_R(\mu;1)$. 
\end{lemma}
\begin{proof} Suppose that $e_\P$ is an isomorphism. We must check that $\mathcal{E}$ satisfies $\mathcal{E}_V \circ \eta_V = 1_V$ and $\eta_V \circ \mathcal{E}_V = \mathsf{S}(\P,\pi_2) \circ \partial_V$. The first identity is automatic since $e^{-1}_\P(1_\P) = 1$:
\[ \mathcal{E}_V(\eta_v(v))= \mathcal{E}_V(1_\P; v) = e^{-1}_\P(1_\P) \cdot v = 1\cdot v =v.\]
For the other identity, we must prove it in two cases.  For the case $\mu\in\P(1)$, note that $\mu = e^{-1}_\P(\mu) \cdot 1_\P$. So, using that $(r \cdot \mu; v) = (\mu; r \cdot v)$, we have that: 
\begin{align*}
\mathsf{S}(\P,\pi_2)\left( \partial_V (\mu;v) \right) &= (\mu; v) = (e^{-1}_\P(\mu) \cdot 1_\P; v) 
=( 1_\P; e^{-1}_\P(\mu) \cdot v) \\&= (1_\P; \mathcal{E}_V(\mu,v) ) = \eta_V( \mathcal{E}_V(\mu,v)).
\end{align*}
Lastly, when $n\neq 1$, using that $(\mu; v_1, \hdots, 0, \hdots, v_n) =0$, we compute that: 
\begin{align*}
\mathsf{S}(\P,\pi_2)\left( \partial_V (\mu;v_1, \hdots, v_{n}) \right) = \sum^{n}_{i=1} (\mu; 0, \hdots, v_i, \hdots, 0) = 0 = \eta_V( \mathcal{E}_V(\mu;v_1, \hdots, v_{n})),
\end{align*}    
so we have that $\mathcal{E}$ is a $\mathsf{D}$-linear counit. Conversely, suppose that $\mathcal{E}$ is a $\mathsf{D}$-linear counit, so in particular, $\mathcal{E}_R \circ \eta_R = 1$ and $\eta_R \circ \mathcal{E}_R = \mathsf{S}(\P,\pi_2) \circ \partial_R$. On the one hand, we have:
\begin{align*}
e^{-1}_\P( e_\P(1))= \mathcal{E}_R(1_\P;1) =  \mathcal{E}_R(\eta_R(1)) =1.
\end{align*}
On the other hand, first observe that by our notation, for $\mu, \nu \in \P(1)$, if $(\mu; 1) = (\nu; 1)$, then this means that $\mu=\nu$. We compute: 
\begin{align*}
&(e_\P( e^{-1}_\P (\mu)); 1) = (\mathcal{E}_R(\mu;1) \cdot 1_\P; 1) = (1_\P; \mathcal{E}_R(\mu;1) \\
&=\eta_R(\mathcal{E}_R(\mu;1)) = \mathsf{S}(\P,\pi_2)\left( \partial_R (\mu;1) \right) = (\mu;1) ,
\end{align*}
hence $e_\P( e^{-1}_\P(1) (\mu)) = \mu$, and so, $e_\P$ is an isomorphism.
\end{proof}

For the operads $\Com$, $\Ass$, and $\Lie$, their associated coCartesian differential monads all have a $\mathsf{D}$-linear counit which is given precisely by projecting out the copy of $V$ in the symmetric algebra, tensor algebra, or free Lie algebra respectively. On the other hand, for an arbitrary $R$-algebras $A$, the associated coCartesian differential monad of the operad $A^\bullet$ will not, in general, have a $\mathsf{D}$-linear counit unless $A \cong R$. 

\subsection{Tangent Structure for Algebras over an Operad}
\label{subsection:otc-concrete-description}

In this section, we describe the tangent structure on the category of algebras over an operad. We have already shown that the monad associated to an operad is a coCartesian differential monad, then, by applying Proposition \ref{prop:cdmonad-tanmonad} and Theorem \ref{thm:S-tan}, we can state the following: 

\begin{lemma}\label{cor:P-tan} Let $\P$ be an operad. For the coCartesian differential monad $(\mathsf{S}(\P,-),  \gamma, \eta, \partial)$, the induced $\mathbb{B}$-distributive law $\lambda_V: \mathsf{S}(\P, V \times V) \to \mathsf{S}(\P,V) \times \mathsf{S}(\P,V)$ as defined in Proposition \ref{prop:cdmonad-tanmonad} is given as follows:  
\begin{equation*}
 \lambda_V (\mu;(u_1,v_1),\dots,(u_{n},v_{n})) = \left( (\mu;u_1,\dots,u_{n}), \sum_{i=1}^{n}\left(\mu; u_1,\hdots, v_i,\hdots, u_{n} \right)  \right).
\end{equation*}
Furthermore, $(\mathsf{ALG}_{\mathsf{S}(\P,-)}, \mathbb{T})$ is a Cartesian Rosick\'y tangent category, where the Rosick\'y tangent structure is defined as in Theorem \ref{thm:S-tan}. 
\end{lemma}
\begin{proof} Recall that, by definition, $\lambda_V := \left\langle  \mathsf{S}(\P,\pi_1),  S\left(\P, \pi_1 + \pi_4 \right) \circ \partial_{V \times V} \right \rangle$. We leave it as an exercise for the reader to check that this gives precisely the formula above. 
\end{proof}

Let us give a more concrete description of the tangent structure by describing the tangent bundle in terms of semi-direct products. Following the standard terminology in operad literature, by an algebra over an operad we mean an algebra over the operad's associated monad \cite[Section 5.2.3]{LV}. More explicitly, for an operad $\P$, a $\P$-algebra is an $R$-module $A$ equipped with an $R$-linear morphism $\theta: \mathsf{S}(\P,A) \to A$ such that $\theta \circ \eta_A = 1_A$ and $\theta \circ \mathsf{S}(\P,\theta) = \theta \circ \gamma_A$. As a useful shorthand, we write:
\begin{align*}
\theta(\mu;a_1,\dots,a_{n}) := \mu(a_1,\dots,a_{n}),
\end{align*}
and so, we will only write $A$ for a $\P$-algebra when there is no confusion about its $\P$-algebra structure $\theta$. Therefore, the necessary $\P$-algebra identities can be expressed as: 
\begin{align*}
1_\P(a) = a, && \mu\left(\nu_1( a_{1,1}, \hdots, a_{1,n_1}), \hdots, \nu_k( v_{k,1}, \hdots, v_{k,n_k}) \right) = \mu \left( \nu_1, \hdots, \nu_k  \right)(a_{1,1}, \hdots, a_{k,n_k}) .
\end{align*}
Similarly, by a $\P$-algebra morphism we simply mean an $\mathsf{S}(\P,-)$-algebra morphism. So if $A$ and $B$ are $\P$-algebras, then a $\P$-algebra morphism is an $R$-linear morphism $f: A \to B$ which is compatible with $\P$-algebra structure, that is, the following equality holds: 
\begin{align*}
f(\mu(a_1,\dots,a_{n})) = \mu(f(a_1),\dots,f(a_{n})).
\end{align*}
Therefore, by the category of algebras over an operad we mean the Eilenberg--Moore category of its associated monad. So let $\mathsf{ALG}_\P$ be the category of $\P$-algebras and $\P$-algebra morphisms between them, or in other words, $\mathsf{ALG}_\P := \mathsf{ALG}_{\mathsf{S}(\P,-)}$. 

By Lemma \ref{cor:P-tan}, we already know that $\mathsf{ALG}_\P$ is a tangent category. However, we wish to give a more explicit description of the tangent bundle of a $\P$-algebra. We do this using the semi-direct product of a $\P$-algebra with itself. We note that, while the semi-direct product \cite[Section 12.3.2]{LV} can be more generally defined between a $\P$-algebra and a module (a notion that we review in the next section below), for the purpose of the tangent structure, we only need to understand it for a $\P$-algebra with itself. 

For a $\P$-algebra $A$, define the $\P$-algebra $A\ltimes A$ as the $R$-module $A \times A$ equipped with $\P$-algebra structure defined as follows: 
\begin{align*}
\mu((a_1, b_1),\hdots,(a_{n}, b_{n}))=\left(\mu(a_1,\hdots,a_{n}),\sum_{i=1}^{n}\mu(a_1,\hdots,b_i,\hdots,a_{n})\right).
\end{align*}
Note that, in $A\ltimes A$, the first component $A$ is viewed as a $\P$-algebra, while the second component $A$ is viewed as a module. More generally, we also define the $\P$-algebra $A\ltimes A^n$ to be the $R$-module $A \times \underbrace{A \times \hdots \times A}_{n \text{ times}}$ with $\P$-algebra structure defined as (where we denoted by $\vec{b}$ a tuple $\left(b_1,\hdots,b_n\right)$ of elements of $A$):
\begin{multline*}
\mu\left( (a_1, \vec{b}_1), \hdots, (a_{m}, \vec{b}_{m}) \right) = \\\left(\mu(a_1,\hdots,a_{m}),\sum_{i=1}^{m}\mu(a_1,\hdots,b_{i,1},\hdots,a_{m}), \hdots, \sum_{i=1}^{m}\mu(a_1,\hdots,b_{i,n},\hdots,a_{m})\right).
\end{multline*}
We will now prove that $A\ltimes A$ is precisely the tangent bundle over $A$.

\begin{lemma} Let $\P$ be an operad and $A$ a $\P$-algebra. Then $\mathsf{T}(A) = A\ltimes A$ and $\mathsf{T}_n(A) = A\ltimes A^n$, where $\mathsf{T}(A)$ and $\mathsf{T}_n(A)$ are defined as in Theorem \ref{thm:S-tan}.
\end{lemma}
\begin{proof} Let $\theta: \mathsf{S}(\P,A) \to A$ denote the $\P$-algebra structure map of $A$. Then, by Theorem \ref{thm:S-tan}, $\mathsf{T}(A)$ is the $R$-module $A \times A$ with $\P$-algebra structure map defined as $(\theta \times \theta) \circ \lambda_A$, where $\lambda$ was given in Lemma \ref{cor:P-tan}. So, since $\mathsf{T}(A)$ and $A\ltimes A$ have the same underlying $R$-module $A \times A$, we must show that they have the same $\P$-algebra structure. We compute: 
\begin{align*}
(\theta \times \theta)\left(\lambda_A\left( (a_1, b_1),\dots,(a_{n}, b_{n}) \right) \right) &=   (\theta \times \theta)\left( (\mu;a_1,\dots,a_{n}), \sum_{i=1}^{n}\left(\mu; a_1,\hdots, b_i,\hdots, a_{n} \right)  \right) \\
&= \left( \mu(a_1,\hdots, a_{n}), \sum_{i=1}^{n}\mu\left(a_1,\hdots, b_i,\hdots, a_{n} \right)  \right).
\end{align*}
We conclude that $\mathsf{T}(A) = A\ltimes A$. Similarly, we can also show that $\mathsf{T}_n(A) = A\ltimes A^n$. 
\end{proof}

Therefore, we may write the tangent structure of Theorem \ref{thm:S-tan} for $\mathsf{ALG}_\P$ in terms of semi-direct products. 

\begin{thm}
\label{theorem:footT-semi-direct-product}
Let $\P$ be an operad. Consider: 
\begin{enumerate}[{\em (i)}]
\item The tangent bundle functor $\mathsf{T}: \mathsf{ALG}_{\P} \to \mathsf{ALG}_{\P}$ defined on objects as $\mathsf{T}(A) = A\ltimes A$ and on maps as $\mathsf{T}(f) = (f \times f)$, that is:
\[\mathsf{T}(f)(a,b) = (f(a), f(b))\]
\item The projection $p_A: A\ltimes A \to A$ defined as:
\[p_A(a,b) = a\] 
with $n$-fold pullback $\mathsf{T}_n(A) = A\ltimes A^n$, and projections $q_j: A\ltimes A^n \to A\ltimes A$ defined as:
\[q_j(a,b_1, \hdots, b_n) = (a,b_j)\] 
\item The sum $s_A: A \ltimes A^2 \to A \ltimes A$ defined as:
\[s_A(a,b_1, b_2) = (a, b_1+b_2)\]
\item The zero map $z_A: A \to A\ltimes A$ defined as:
\[z_A(a) = (a,0)\]
\item The vertical lift $\ell_A: A\ltimes A \to (A\ltimes A) \ltimes (A\ltimes A)$ defined as:
\[\ell_A(a,b) = (a,0,0,b)\]
\item The canonical flip $c_A: (A\ltimes A) \ltimes (A\ltimes A) \to (A\ltimes A) \ltimes (A\ltimes A)$ defined as:
\[c_A(a,b,c,d) = (a,c,b,d)\]
\item The negative map $n_A: A\ltimes A \to A\ltimes A$ defined as:
\[n_A(a,b) = (a, -b)\] 
\end{enumerate}
Then, $\mathbb{T}= (\mathsf{T},  p, s, z, l, c, n)$ is a Rosick\'y tangent structure on $\mathsf{ALG}_\P$, and so, $(\mathsf{ALG}_\P, \mathbb{T})$ is a Cartesian Rosick\'y tangent category. 
\end{thm}

We now consider what the resulting tangent categories are for our main examples of operads. 

\begin{example}\label{example:Abullet-modules} Let $A$ be an $R$-algebra. Then for the operad $A^\bullet$, the $A^\bullet$-algebras are precisely $A$-modules. So, $\mathsf{ALG}_{A^\bullet} =\mathsf{MOD}_A$, the category of $A$-modules and $A$-linear maps between them. The resulting tangent structure on $\mathsf{MOD}_A$ is precisely the biproduct structure from Lemma \ref{lemma:biproduct}, so in particular, for an $A$-module $M$, its tangent bundle is simply $\mathsf{T}(M) = M \times M$. So $(\mathsf{MOD}_A, \mathbb{B})$ is a Cartesian Rosick\'y tangent category. 
\end{example}

\begin{example}
\label{example:tangent-structure-over-Alg-Com}
For the operad $\Com$, the $\Com$-algebras are precisely (associative and unital) commutative $R$-algebras. So $\mathsf{ALG}_\Com = \mathsf{CALG}_R$, the category of commutative $R$-algebras and $R$-algebra morphisms between them. Up to isomorphism, the resulting tangent structure is the one described in \cite[Section 2.2]{cruttwellLemay:AlgebraicGeometry}, where the tangent bundle is given by dual numbers. Indeed, for a commutative $R$-algebra $A$, let $A[\epsilon]$ be its $R$-algebra of dual numbers, which we use as a shorthand for $A[\epsilon] := A[x]/(x^2)$: 
\begin{align*} A[\epsilon] = \lbrace a + b \epsilon \vert~ \forall a,b \in A \rbrace, &&\text{with } \epsilon^2 = 0.
\end{align*}
It is easy to see that $A\ltimes A \cong A[\epsilon]$ via $(a,b) \mapsto a + b\epsilon$. So we may express the tangent structure using instead dual numbers, so $\mathsf{T}(A) = A[\epsilon]$. We may write $\mathsf{T}^2(A)$ and $\mathsf{T}_n(A)$ as multivariable dual numbers in the following way: 
\begin{align*} \mathsf{T}^2(A) = A[\epsilon][\epsilon^\prime] =  \lbrace a + b \epsilon + c\epsilon^\prime + d \epsilon\epsilon^\prime \vert~ \forall a,b,c,d \in A \rbrace, &&\text{with } \epsilon^2 = {\epsilon^\prime}^2 = 0, \\
\mathsf{T}_n(A) = A[\epsilon_1, \hdots, \epsilon_n] =  \lbrace a + b_1 \epsilon_1 + \hdots + b_n \epsilon_n \vert~ \forall a,b_j \in A \rbrace, && \text{with } \epsilon_i\epsilon_j = 0.
\end{align*}
The rest of the tangent structure is given as follows: 
\begin{gather*}
p_{A}(a+b \epsilon) = a, \quad \quad \quad s_A(a + b \epsilon_1 + c \epsilon_2) = a + (b+c) \epsilon, \quad \quad \quad z_A(a) = a, \\
l_A(a + b\epsilon) = a + b \epsilon\epsilon^\prime, \quad \quad \quad c_A( a + b \epsilon + c\epsilon^\prime + d \epsilon\epsilon^\prime) =  a + c \epsilon + b\epsilon^\prime + d \epsilon\epsilon^\prime, \\
n_A(a + b \epsilon) = a - b \epsilon,
\end{gather*}
so $(\mathsf{CALG}_R, \mathbb{T})$ is a Cartesian Rosick\'y tangent category. 
\end{example}

\begin{example}
\label{example:tangent-structure-over-Alg-Ass}
For the operad $\Ass$, the $\Ass$-algebras are precisely (associative and unital) $R$-algebras. So $\mathsf{ALG}_\Ass = \mathsf{ALG}_R$, the category of $R$-algebras and $R$-algebra morphisms between them. Again, up to isomorphism, the resulting tangent structure is precisely the same as for commutative algebras in Example \ref{example:tangent-structure-over-Alg-Com}, so in particular for an $R$-algebra $A$, $\mathsf{T}(A) = A[\epsilon] \cong A\ltimes A$. Therefore, $(\mathsf{ALG}_R, \mathbb{T})$ is a Cartesian Rosick\'y tangent category. 
\end{example}

\begin{example}
\label{example:tangent-structure-over-Alg-Lie}
For the operad $\Lie$, the $\Lie$-algebras are precisely Lie algebras over $R$. So $\mathsf{ALG}_\Ass = \mathsf{LIE}_R$, the category of Lie algebras over $R$ and Lie algebra morphisms between them. For a Lie algebra $\Lieg$, the Lie algebra $\Lieg\ltimes\Lieg$ is the $R$-module $\Lieg\times\Lieg$ with the Lie brackets defined as $\left[(x_1,y_1),(x_2,y_2) \right]:=([x_1,x_2],[x_1,y_2]+[y_1,x_2])$. So $(\mathsf{LIE}_R, \mathbb{T})$ is a Cartesian Rosick\'y tangent category, which we stress is a new example of a tangent category. 
\end{example}

We conclude this section by mentioning that the construction of Theorem \ref{theorem:footT-semi-direct-product} provides a (contravariant) functor between the category of operads and the category of Cartesian Rosick\'y tangent categories. Briefly, for operads $\P$ and $\P^\prime$, an operad morphism $f: \P \to \P^\prime$ is a sequence of equivariant $R$-linear morphisms $f(n): \P(n) \to \P^\prime(n)$ which preserve the partial compositions and the distinguished object. Then, let $\mathsf{OPERAD}_R$ be the category of operads and operad morphisms between them. On the other hand, for $(\mathbb{X}, \mathbb{T})$ and $(\mathbb{X}^\prime, \mathbb{T}^\prime)$, a strict Cartesian Rosick\'y tangent morphism $\mathsf{F}: (\mathbb{X}, \mathbb{T}) \to (\mathbb{X}^\prime, \mathbb{T}^\prime)$ is a functor $\mathsf{F}: \mathbb{X} \to \mathbb{X}^\prime$ which preserves the product strictly and also preserve the tangent structure, in the sense that $\mathsf{F} \circ \mathsf{T} = \mathsf{T}^\prime \circ \mathsf{F}$, $\mathsf{F}(p) = p^\prime$, etc. Let $\mathsf{CRTAN}_{=}$ be the category of Cartesian Rosick\'y tangent categories and strict Cartesian tangent morphisms between them. Every operad morphism ${f: \P \to \P^\prime}$ induces a functor $\mathsf{ALG}_f: \mathsf{ALG}_{\P^\prime} \to \mathsf{ALG}_{\P}$ by mapping a $\P^\prime$-algebra $A$ to $A$ with $\P$-algebra structure defined as $\mu(a_1, \hdots, a_n) := f(\mu)(a_1, \hdots, a_n)$. It is straightforward to see that $\mathsf{ALG}_f: (\mathsf{ALG}_{\P^\prime}, \mathbb{T}) \to (\mathsf{ALG}_\P, \mathbb{T})$ is a strict Cartesian Rosick\'y tangent morphism. Therefore we obtain a functor $\mathsf{ALG}: \mathsf{OPERAD}^{op}_R \to \mathsf{CRTAN}_{=}$ which sends an operad $\P$ to $(\mathsf{ALG}_\P, \mathbb{T})$ and operad morphisms $f: \P \to \P^\prime$ to $\mathsf{ALG}_f: (\mathsf{ALG}_{\P^\prime}, \mathbb{T}) \to (\mathsf{ALG}_\P, \mathbb{T})$. 

In future work, it will be interesting to see if the left adjoint of the pullback functor $\mathsf{ALG}_f$ induced by an operad morphism $f$ is also a tangent morphism.

\subsection{Adjoint Tangent Structure of Algebras over an Operad}\label{subsec:adjoint-tan-operad}

The objective of this section is to study a tangent structure for the opposite category of algebras over an operad, which is adjoint to the one we defined in Section \ref{subsection:otc-concrete-description}. 

Since the category of algebras over an operad is always cocomplete, it admits all reflexive coequalizers. So, by Theorem \ref{theorem:adjoint-tangent-structure-for-cCD-monads}, we deduce that the tangent structure defined in Section \ref{subsection:otc-concrete-description} has adjoint tangent structure. The adjoint tangent bundle is given by the free algebra over the module of K\"ahler differentials of an algebra. This is quite a mouthful, so let's break it down piece by piece. 

Let $\P$ be an operad and $A$ a $\P$-algebra. Then an \textbf{$A$-module} \cite[Section 12.3.1]{LV} is an $R$-module $M$ equipped with a family of $R$-linear morphisms $\psi_{n+1}:\P(n+1)\otimes A^{\otimes n}\otimes M\to M$, called \textbf{evaluation maps}, satisfying natural equivariance, associativity, and a unit map: $\eta_M:M\to \bigoplus_{n\in\N}\P(n+1)\otimes A^{\otimes n}\otimes M$ playing the role of a unit for the evaluation. As a shorthand, we write:
\begin{align*}
\mu(a_1,\dots,a_{n},x):=\psi_{n+1}(\mu\otimes a_1\otimes\dots\otimes a_{n}\otimes x).
\end{align*}
If $M$ and $M^\prime$ are $A$-modules, then an $A$-linear morphism is an $R$-linear morphism $f: M \to M^\prime$ which preserves the evaluation and unit maps in the sense that: 
\begin{align*}
\eta_{M\prime}(f(x))&=\left(\bigoplus_{n+1}\P(n+1)\otimes A^{\otimes n}\otimes f\right)\circ\eta_M(x),\\
f(\mu(a_1,\dots,a_{n},x))&= \mu(a_1,\dots,a_{n},f(x))  .
\end{align*}
Let $\mathsf{MOD}_A$ be the category of $A$-modules and $A$-linear morphisms between them. Among the $A$-modules, there is an important one called the module of K\"ahler differentials of $A$, which generalizes the classical notion of K\"ahler differentials. 

We must first describe derivations for algebras over an operad. For an $A$-module $M$, we adopt the following useful notation:
\begin{align*}
\mu(a_1,\dots,a_{i},x,a_{i+1},\dots,a_{n})=\left(\mu\cdot(i~ i+1\hdots n)\right)(a_1,\hdots,a_{n},x),
\end{align*}
where $(i~ i+1\dots n)$ is the $n+1-i$-cycle permutation. An \textbf{$A$-derivation} \cite[Section 12.3.7]{LV} evaluated in an $A$-module $M$ is an $R$-linear morphism $D:A\to M$ satisfying:
\begin{align*}
D(\mu(a_1,\dots,a_{n}))=\sum_{i=1}^{n}\mu(a_1,\hdots,D(a_i),\hdots,a_{n}).
\end{align*}
This equality is called the Leibniz rule. 

Now let $\Der(A,M)$ be the $R$-module of $A$-derivations evaluated in $M$. This way, we define a functor $\Der(A,-)$ which is representable \cite[Proposition 12.3.11]{LV}. The \textbf{module of K\"ahler differentials of $A$} \cite[Section 12.3.8]{LV} is an $A$-module which represents $\Der(A,-)$, that is, an $A$-module $\Omega_A$ such that, for all $A$-modules $M$, $\Der(A,M) \cong \Mod_A(\Omega_A,M)$. This means that there is an $A$-derivation $\mathsf{d}: A \to \Omega_A$ which is universal in the sense that for every $A$-derivation ${D: A \to M}$, there exists a unique $A$-linear morphism $\overline{D}:\Omega_A\to M$ such that $\overline{D} \circ \mathsf{d} = D$. We do not need a concrete description of $\Omega_A$ for this paper, see \cite[Lemma 12.3.12]{LV} for full details. However, it is interesting to point out that, for the $\P$-algebra $\mathsf{S}(\P,V)$, $\Omega_{\mathsf{S}(\P,V)}$ is isomorphic to the sub-$\mathsf{S}(\P,V)$-module of $\mathsf{S}(\P,V\times V)$ generated by elements of the form $\left(\mu;(v_1,0),\hdots, (0,v_i), \hdots, (v_{n},0)\right)$. Therefore, the differential combinator transformation $\partial_V: \mathsf{S}(\P,V) \to \mathsf{S}(\P,V \times V)$ factors through $\Omega_{\mathsf{S}(V)}$ by composing the derivation $\mathsf{d}:\mathsf{S}(\P, V)\to \Omega_{\mathsf{S}(V)}$ with the inclusion $\Omega_{\mathsf{S}(V)}\to \mathsf{S}(\P,V\times V)$. 

For an arbitrary $\P$-algebra $A$, $\Omega_A$ is not a $\P$-algebra. One might be tempted to consider $\mathsf{S}(\P, \Omega_A)$ as a candidate for the tangent bundle over $A$. However, this is not the adjoint functor we are looking for. We will instead take the free $A$-algebra over $\Omega_A$. An \textbf{$A$-algebra}, also called a $\P$-algebra under $A$ \cite{moerdijkEnvelopingOperads}, is a $\P$-algebra $B$ equipped with a $\P$-algebra morphism $u: A \to B$. Now, if $B$ and $B^\prime$ are $A$-algebras with $u: A \to B$ and $u^\prime: A \to B^\prime$ respectively, then an $A$-algebra morphism is a $\P$-algebra morphism $f: B \to B^\prime$ which also preserves the $A$-algebra structure, so $f \circ u = u^\prime$. Let $\mathsf{ALG}_A$ be the category of $A$-algebras and $A$-algebra morphisms between them. Every $A$-algebra $B$ is also an $A$-module where the evaluation is given by:
\begin{align*}
\mu(a_1,\dots,a_{n},b) := \mu(u(a_1),\dots,u(a_{n}),b),
\end{align*}
and similarly, every $A$-algebra morphism is also an $A$-module morphism. We obtain a functor $\mathsf{U}_A: \mathsf{ALG}_A \to \mathsf{MOD}_A$. The functor $\mathsf{U}_A$ has a left adjoint $\Free_A:\mathsf{MOD}_A \to \mathsf{ALG}_A$. In the next proposition, we provide a concrete description of $\Free_A$. This is an extension of a result due to Ginzburg, who proved the existence of $\Free_A$ for quadratic operads \cite[Lemma~5.2]{Ginzburg-NCSymplecticGeometry}. This means that, for every $A$-module $M$, there exists an $A$-algebra $\Free_A(M)$ with $u_M: A \to \Free_A(M)$ called the \textbf{free $A$-algebra over $M$}.

\begin{proposition}
\label{prop:free-A-algebra-functor}
Let $A$ be a $\P$-algebra and let $M$ be a module over $A$ (in the operadic sense). Consider the $\P$-algebra $\Free_AM$ obtained by quotienting the free algebra $\mathsf{S}(\P,A\times M)$ by the ideal generated by the following relations:
\begin{multline*}
(\mu;(a_1,0),\dots,(a_{k-1},0),(a_k,x),(a_{k+1},0)\dots,(a_n,0))\\
=(\mu(a_1,\dots,a_n),\mu(a_1,\dots,a_{k-1},x,a_{k+1},\dots,a_n)),
\end{multline*}
for every $\mu\in\P(n)$, $a_1,\dots,a_n\in A$, $x\in M$ and positive integer $n$. Then, $\Free_A:\mathsf{MOD}_A\to\mathsf{ALG}_A$ extends to a left adjoint to the functor $\mathsf{U}_A:\mathsf{ALG}_A\to\mathsf{MOD}_A$, where the $A$-algebra structure $u_M: A \to \Free_AM$ is defined as the injection $u_M(a) =(a,0)$.
\end{proposition}
\begin{proof} Note that we have an $A$-module morphism $\iota_M: M \to \Free_A M$ defined as the inclusion $\iota_M(x) = (0,x)$. Now given an $A$-algebra $u:A\to B$ and an $A$-algebra morphism $f:\Free_A M\to B$, we can define an $A$-module morphsim $f^\flat: M \to \Free_A M$ as the composite $f^\flat = f \circ \iota_M$. Conversely, given an $A$-module morphism $g:M\to\mathsf{U}_AB$, it is not difficult to check that the $\P$-algebra morphism $\mathsf{S}(\P,A\times M) \to B$ mapping $(a,x) \mapsto u(a)+g(x)$  lifts to the quotient. As such, it provides a well-defined $A$-algebra morphism $g^\sharp:\Free_AM\to B$. The final step is to note that $f\mapsto f^\flat$ and $g\mapsto g^\sharp$ are inverses of each other. Thus, we obtain a natural bijection $\mathsf{ALG}_A(\Free_AM,B)\cong\mathsf{MOD}_A(M,\mathsf{U}_AB)$, and thus an adjunction as desired. 
\end{proof}

With all this setup, we can finally define the adjoint tangent bundle of a $\P$-algebra to be $\mathsf{T}^\circ(A) := \Free_A(\Omega_A)$. Using the combined universal properties of both $\Omega_A$ and $\Free_A(-)$, we can conclude that:
\begin{align*}
\mathsf{ALG}_\P\left( \Free_A(\Omega_A) , A^\prime \right) \cong \mathsf{ALG}_\P(A, A^\prime \ltimes A^\prime).
\end{align*}
Therefore, $\mathsf{T}^\circ: \mathsf{ALG}_\P \to \mathsf{ALG}_\P$ is indeed a left adjoint to $\mathsf{T}: \mathsf{ALG}_\P \to \mathsf{ALG}_\P$. However, let us give a more concrete description of the adjoint tangent bundle and the adjunction. For a $\P$-algebra $A$, its adjoint tangent bundle $\mathsf{T}^\circ(A)$ is explicitly given by the $\P$-algebra obtained by quotienting $\mathsf{S}(\P,A\times A)$ by the following relations:
\begin{align*}
 \mu((a_1,0),\dots,(a_{n},0))&=(\mu(a_1,\dots,a_{n}),0), \\
(0,\mu(a_1,\dots,a_{n})) &= \sum_{i=1}^{n}\mu((a_1,0),\hdots, (0,a_i), \hdots,(a_n,0)).
\end{align*}
As useful shorthand, we write $a := (a,0) \in \mathsf{T}^\circ(A)$ and $\mathsf{d}(a) := (0,a) \in \mathsf{T}^\circ(A)$ for all $a \in A$. The above relations then state that $\mu(a_1,\dots,a_{n})$ in $\mathsf{T}^\circ(A)$ corresponds to $\mu(a_1,\dots,a_{n})$ in $A$, and that
\begin{align*}
\mathsf{d}\left( \mu(a_1,\dots,a_{n})\right) &= \sum_{i=1}^{n}\mu(a_1,\hdots, \mathsf{d}(a_i),\dots, a_n).
\end{align*}
Note that $\mathsf{d}$ can be thought of an $R$-linear function: $\mathsf{d}(r\cdot a+s\cdot b)=r\cdot \mathsf{d}(a)+s \cdot \mathsf{d}(b)$. As a $\P$-algebra, $\mathsf{T}^\circ(A)$ is generated by $a$ and $\mathsf{d}(a)$ for all $a \in A$. As such, to define $\P$-algebra morphisms with domain $\mathsf{T}^\circ(A)$, it suffices to define them on the generators $a$ and $\mathsf{d}(a)$, making sure that the definition is compatible with the above relations. For every $\P$-algebra morphism $f: P \to P^\prime$, define the $\P$-algebra morphism ${\mathsf{T}^\circ(f): \mathsf{T}^\circ(A) \to \mathsf{T}^\circ(A^\prime)}$ on generators as follows:
\begin{align*}
    \mathsf{T}^\circ(f)(a) = f(a), && \mathsf{T}^\circ(f)(\mathsf{d}(a)) = \mathsf{d}(f(a)).
\end{align*}
This gives us the desired functor ${\mathsf{T}^\circ: \mathsf{ALG}_\P \to \mathsf{ALG}_\P}$. Observe that we did not need the $A$-algebra structure of $\mathsf{T}^\circ(A)$ to build this functor. Nevertheless, readers familiar with modules of K\"ahler differentials may easily check that the presentation of $\mathsf{T}^\circ(A)$ given here recaptures precisely $\Free_A(\Omega_A)$ (especially using the $\mathsf{d}$ notation). That said, the $A$-algebra structure of $\mathsf{T}^\circ(A)$ will be precisely the adjoint projection $p^\circ_A: A \to \mathsf{T}^\circ(A)$. 

Turning our attention back to the adjunction, we define the unit $\eta_A: A \to \mathsf{T}^\circ (A) \ltimes \mathsf{T}^\circ (A)$ as follows: 
\begin{align*}
\eta_A(a) = (a, \mathsf{d}(a)),
\end{align*}
which is clearly a $\P$-algebra morphism. The counit $\varepsilon_A: \mathsf{T}^\circ(A \ltimes A) \to A$ is the $\P$-algebra morphism defined on generators $(a,b)$ and $\mathsf{d}(a,b)$ for all $(a,b) \in A \ltimes A$ as follows: 
\begin{align*}
\varepsilon_A(a,b) = a, && \varepsilon(\mathsf{d}(a,b))= b.
\end{align*}

\begin{lemma}\label{lem:operad-adj} $(\eta, \varepsilon): \mathsf{T}^\circ \dashv \mathsf{T}$ is an adjunction. 
\end{lemma}
\begin{proof} We leave it as an exercise for the reader to check that the adjunction triangle identities are satisfied. Alternatively, one could check that $\mathsf{ALG}_\P\left( \mathsf{T}^\circ(A) , A^\prime \right) \cong \mathsf{ALG}_\P(A, A^\prime \ltimes A^\prime)$. Explicitly, given a $\P$-algebra morphism $f: \mathsf{T}^\circ(A) \to A^\prime$, define the $\P$-algebra morphism $f^\flat: A \to A^\prime \ltimes A^\prime$ as $f^\flat(a) = (f(a), \mathsf{d}(f(a)))$, and conversely, given a $\P$-algebra morphism $g: A \to A^\prime \ltimes A^\prime$, with $g(a) = (g_1(a), g_2(a))$, define the $\P$-algebra morphism $g^\sharp: \mathsf{T}^\circ(A) \to A^\prime$ on generators as $g^\sharp(a) = g_1(a)$ and $g^\sharp(\mathsf{d}(a))= g_2(a)$. One can check that the mappings $f\mapsto f^\flat$ and $g\mapsto g^\sharp$ define mutually inverse bijections.
\end{proof}

For any operad $\P$, $\mathsf{ALG}_\P$ is cocomplete \cite[Proposition 6.4]{mandell:operadic-algebra-co-limits}, and so, it admits all coproducts and pushouts. Therefore, applying Theorem \ref{theorem:dual-tangent-structure} and Corollary \ref{corollary:dual-tangent-structure}, we obtain:

\begin{corollary} Let $\P$ be an operad and $A$ a $\P$-algebra. The Cartesian Rosick\'y tangent category $(\mathsf{ALG}_\P, \mathbb{T})$ defined in Theorem \ref{theorem:footT-semi-direct-product} has an adjoint tangent structure, and so $(\mathsf{ALG}^{op}_\P, \mathbb{T}^\circ)$ is a Cartesian Rosick\'y tangent category, where $\mathbb{T}^\circ$ is defined as in Theorem \ref{theorem:dual-tangent-structure}. 
\end{corollary}

We will now give a concrete description of the adjoint tangent structure. We can define all the necessary structure maps on generators. Let us first describe the generators of $\mathsf{T}^\circ_n(A)$ and ${\mathsf{T}^\circ}^2(A)$. On the one hand, $\mathsf{T}^\circ_n(A)$ is a quotient of $\mathsf{S}(\P, \prod\limits^{n+1}_{i=1} A )$ modulo the similar equations as above, and therefore, can be described in terms of generators $a$ and $\mathsf{d}_i(a)$ for all $a \in A$ and $1 \leq i \leq n$. By Lemma \ref{lemma:dual-tangent-structure}, $\mathsf{T}^\circ_n$ is indeed a left adjoint to $\mathsf{T}_n$. On the other hand, ${\mathsf{T}^\circ}^2(A)$ is generated by $a$, $\mathsf{d}(a)$, $\mathsf{d}^\prime(a)$, and $\mathsf{d}^\prime\mathsf{d}(a)$ for all $a \in A$, where $\mathsf{d}$ is for $\Omega_A$ and $\mathsf{d}^\prime$ is for $\Omega_{\mathsf{T}^\circ(A)}$. 

\begin{thm}
\label{theorem:operadic-tangent-categories} Let $\P$ be an operad. Consider: 
\begin{enumerate}[{\em (i)}]
\item The adjoint tangent bundle functor $\mathsf{T}^\circ: \mathsf{ALG}_{\P} \to \mathsf{ALG}_{\P}$ defined on objects by $\mathsf{T}^\circ(A) = \Free_A(\Omega_A)$ and on maps by $\mathsf{T}^\circ(f)$ (as defined above); 
\item The adjoint projection $p^\circ_A: A \to \mathsf{T}^\circ(A)$ defines as:
\[p^\circ_A(a) = a\] 
and where the $n$-fold pushout of $p^\circ_A$ is $\mathsf{T}^\circ_n(A)$, with injections $q^\circ_j: \mathsf{T}^\circ(A) \to \mathsf{T}^\circ_n(A)$ defined on generators as:
\begin{align*}
   q^\circ_j(a) = a &&  q^\circ_j(\mathsf{d}(a))=\mathsf{d}_j(a)
\end{align*}
\item The adjoint sum $s^\circ_A: \mathsf{T}^\circ(A) \to \mathsf{T}^\circ_2(A)$, defined on generators as:
\begin{align*}
    s^\circ_A(a) = a && s^\circ_A(\mathsf{d}(a)) = \mathsf{d}_1(a) + \mathsf{d}_2(a)
\end{align*}
\item The adjoint zero map $z^\circ_A: \mathsf{T}^\circ(A) \to A$, defined on generators as:
\begin{align*}
    z^\circ_A(a) = a && z^\circ_A(\mathsf{d}(a)) = 0
\end{align*}
\item The adjoint vertical lift $\ell^\circ_A: {\mathsf{T}^\circ}^2(A) \to \mathsf{T}^\circ(A)$, defined on generators as: 
\begin{align*}
l^\circ_A(a) = a, && l^\circ_A(\mathsf{d}(a)) =0, && l^\circ_A(\mathsf{d}^\prime(a)) = 0, && l^\circ_A(\mathsf{d}^\prime \mathsf{d}(a)) = \mathsf{d}(a).
\end{align*}
\item The adjoint canonical flip $c^\circ_A: {\mathsf{T}^\circ}^2(A) \to {\mathsf{T}^\circ}^2(A)$, defined on generators as: 
\begin{align*}
c^\circ_A(a) = a, && c^\circ_A(\mathsf{d}(a)) = \mathsf{d}^\prime(a), && c^\circ_A(\mathsf{d}^\prime(a)) = \mathsf{d}(a), && c^\circ_A(\mathsf{d}^\prime \mathsf{d}(a)) = \mathsf{d}^\prime \mathsf{d}(a).
\end{align*}
\item The adjoint negative map $n^\circ_A: \mathsf{T}^\circ(A) \to \mathsf{T}^\circ(A)$, defined on generators as:
\begin{align*}
   n^\circ_A(a) = a &&  n^\circ(\mathsf{d}(a)) = -\mathsf{d}(a)
\end{align*} 
\end{enumerate}
Then, $\mathbb{T}^\circ= (\mathsf{T}^\circ,  p^\circ, s^\circ, z^\circ, l^\circ, c^\circ, n^\circ)$ is a Rosick\'y tangent structure on $\mathsf{ALG}^{op}_\P$, and so, $(\mathsf{ALG}^{op}_\P, \mathbb{T}^\circ)$ is a Cartesian Rosick\'y tangent category. 
\end{thm}

 While the tangent bundle $\mathsf{T}$ is mostly the same for each operad, the adjoint tangent bundle $\mathsf{T}^\circ$ can vary quite drastically from operad to operad. So let us now consider the resulting tangent categories for our main examples of operads. We again stress that, while the first two examples recapture known examples, the last two examples are new examples of tangent categories. While the third example is not too surprising, it does provide a direct link between tangent categories and non-commutative algebraic geometry, which provides a novel application for the theory of tangent categories. On the other hand, the fourth example is a good example that demonstrates how operads provide many new (and surprising) examples of tangent categories that were not previously considered. 

\begin{example} Let $A$ be an $R$-algebra. For the operad $A^\bullet$, Theorem \ref{theorem:operadic-tangent-categories} recaptures precisely the adjoint biproduct tangent structure from Lemma \ref{lemma:biproduct-adjoint}. For an $A$-module $M$, every other $A$-module is an $M$-module in the operadic sense, where the evaluation maps are all zero. Similarly, algebras over $M$ in the operadic sense simply correspond to an $A$-module $N$ equipped with a chosen $A$-linear morphisms $N \to M$. In this case, $\mathsf{Free}_M(N) = M \times N$, which is an algebra over $M$ via the injection map. On the other hand, $\Omega_M = M$, with universal derivation being the identity $1_M: M \to M$. So, we indeed have that $\mathsf{T}^\circ(M) = M \times M$. Thus $(\mathsf{MOD}^{op}_A, \mathbb{B}^\circ)$ is a Cartesian Rosick\'y tangent category. 
\end{example}

\begin{example}
\label{example:otc-Com}
Recall that famously the opposite category of commutative $R$-algebras is isomorphic to the category of affine schemes over $R$. Therefore, the resulting tangent category for operad $\Com$ is equivalent to the tangent category of affine schemes as described in \cite[Section 2.3]{cruttwellLemay:AlgebraicGeometry}, providing a link between tangent categories and algebraic geometry. For a commutative $R$-algebra $A$, a module over $A$ in the operadic sense corresponds precisely to a (left) $A$-module $M$. Free $A$-algebras are constructed by the symmetric $A$-algebra functor:
\[\mathsf{Free}_A(M) = \mathsf{Sym}_A(M) = \bigoplus_{n \in \mathbb{N}} \left(M^{\otimes_A n}\right)_{\Sigma(n)}, \] 
where $\otimes_{A}$ is the tensor product over $A$. On the other hand, $\Omega_A$ is precisely the usual (left) module of K\"ahler differentials over $A$, that is, the free $A$-module over the set $\lbrace \mathsf{d}(a) \vert~ \forall a\in A \rbrace$ modulo the necessary derivation identities. Then $\mathsf{T}^\circ(A)$ is the free symmetric $A$-algebra over $\Omega_A$: 
\[ \mathsf{T}^\circ(A) := \mathsf{Sym}_A \left( \Omega_A \right) = \bigoplus \limits_{n=0}^{\infty} \left(\Omega_A^{\otimes_A n}\right)_{\Sigma(n)} = A \oplus \Omega_A \oplus \left( \Omega_A \otimes_A \Omega_A \right)_{\Sigma(2)} \oplus \hdots \]
In \cite[Definition 16.5.12.I]{grothendieck1966elements}, Grothendieck calls $\mathsf{T}^\circ(A)$ the ``fibré tangent'' (french for tangent bundle) of $A$, while in \cite[Section 2.6]{jubin2014tangent}, Jubin calls $\mathsf{T}^\circ(A)$ the tangent algebra of $A$. An arbitrary element of $\mathsf{T}^\circ(A)$ is a finite sum of monomials of the form $a \mathsf{d}(b_1) \hdots \mathsf{d}(b_n)$, and thus the $R$-algebra structure of $\mathsf{T}^\circ(A)$ is essentially the same as polynomials. In particular, this implies that as an $R$-algebra, $\mathsf{T}^\circ(A)$ is generated by $a$ and $\mathsf{d}(a)$ for all $a \in A$. On the other hand, $\mathsf{T}_n(A)$ can be described in terms of generators $a$ and $\mathsf{d}_i(a)$ for all $a\in A$ and $1 \leq i \leq n$, and ${\mathsf{T}^\circ}^2(A)$ in terms of four generators $a$, $\mathsf{d}(a)$, $\mathsf{d}^\prime(a)$, and $\mathsf{d}^\prime \mathsf{d}(a)$ for all $a \in A$, modulo all the necessary equations. The rest of the adjoint tangent structure is given as follows on generators: 
\begin{align*}
p^\circ_A(a) &= a, &&&&&&\\
q^\circ_j(a) &= a, & q_j^\circ(\mathsf{d}(a)) &= \mathsf{d}_j(a),&&&& \\
s^\circ_A(a) &= a, & s^\circ(\mathsf{d}(a)) &= \mathsf{d}_1(a) + \mathsf{d}_2(a),&&&& \\
z^\circ(a) &= a, & z^\circ(\mathsf{d}(a)) &= 0, &&&&\\
l^\circ_A(a) &= a, & l^\circ_A(\mathsf{d}(a)) &=0, & l^\circ_A(\mathsf{d}^\prime(a)) &= 0, & l^\circ_A(\mathsf{d}^\prime \mathsf{d}(a)) &= \mathsf{d}(a), \\
c^\circ_A(a) &= a, & c^\circ_A(\mathsf{d}(a)) &= \mathsf{d}^\prime(a), & c^\circ_A(\mathsf{d}^\prime(a)) &= \mathsf{d}(a), & c^\circ_A(\mathsf{d}^\prime \mathsf{d}(a)) &= \mathsf{d}^\prime \mathsf{d}(a), \\
n^\circ_A(a) &= a, & n^\circ(\mathsf{d}(a)) &= -\mathsf{d}(a).&&&& 
\end{align*}
So $(\mathsf{CALG}^{op}_R,\mathbb{T}^\circ)$ is a Cartesian Rosick\'y tangent category. 
\end{example}

\begin{example}
\label{example:otc-Ass}
For the operad $\Ass$, this results in a non-commutative version of the previous example. For an $R$-algebra $A$, a module over $A$ in the operadic sense corresponds precisely to an $A$-bimodule $M$. Free $A$-algebras are given by the $A$-tensor algebra:
\[\mathsf{Free}_A(M) = \mathsf{Ten}_A(M) = \bigoplus_{n \in \mathbb{N}} M^{\otimes_A n}.\] 
The $A$-module $\Omega_A$ is the non-commutative version of the module of K\"ahler differentials over $A$ \cite[Section 10]{Ginzburg:notes-on-NCGeometry} (which, it is important to note, is different from the commutative version). Therefore, 
\[\mathsf{T}^\circ(A) = \bigoplus\limits_{n \in \mathbb{N}} \Omega_A^{\otimes_{A} n} = A \oplus \Omega_A \oplus (\Omega_A \otimes_A \Omega_A) \oplus \hdots \]
More concretely, $\mathsf{T}^\circ(A)$ can be described as the free $A$-algebra over the set $\lbrace \mathsf{d}(a) \vert~ \forall a\in A \rbrace$ modulo $\mathsf{d}(ab)=a\mathsf{d}(b)+\mathsf{d}(a)b$ and $\mathsf{d}(ra +sb) = r\mathsf{d}(a) + s\mathsf{d}(b)$. So $(\mathsf{ALG}^{op}_R,\mathbb{T}^\circ)$ is a Cartesian Rosick\'y tangent category. In \cite[Definition 10.2.3]{Ginzburg:notes-on-NCGeometry}, Ginzburg calls $\mathsf{T}^\circ(A)$ the ``space of noncommutative differential forms of $A$''. To the best of our knowledge, this is the first mention of a tangent category that relates directly to non-commutative algebraic geometry. 
\end{example}

\begin{example}
\label{example:otc-Lie}
For the operad $\Lie$, we obtain a new example of a tangent category for Lie algebras. For a Lie algebra $\Lieg$, modules in the operadic sense correspond to representations of $\Lieg$, which we call $\Lieg$-representations and simply denote by their underlying $R$-module $V$. Algebras over $\Lieg$ in the operadic sense correspond to Lie algebras $\Lieg^\prime$ equipped with a Lie algebra morphism $\Lieg \to \Lieg^\prime$. So, $\mathsf{Free}_\Lieg(V)$ is the free Lie algebra over $\Lieg$ of a $\Lieg$-representation $V$. On the other hand, $\Omega_\Lieg$ is the free representation of $\Lieg$ over the set $\mathsf{d}(\Lieg) = \lbrace \mathsf{d}(x) \vert~ \forall x \in \Lieg \rbrace$ modulo the relations $\mathsf{d}(rx+sy)=s\mathsf{d}(x)+r\mathsf{d}(y)$ and $\mathsf{d}\left([x,y] \right)=[\mathsf{d}(x),y]+[x,\mathsf{d}(y)]$ for all $r,s\in R$ and $x,y\in\Lieg$. Hence,  $\mathsf{T}^\circ(\Lieg)$ can be concretely defined as the free Lie algebra over the underlying set of $\Lieg$ and the set $\mathsf{d}(\Lieg)$ modulo the same equalities as for $\Omega_A$, and such that $[x,y] \in \mathsf{T}^\circ(\Lieg)$ is identified to $[x,y] \in \Lieg$, which makes $\mathsf{T}^\circ(\Lieg)$ a Lie algebra over $\Lieg$. Therefore $(\mathsf{LIE}^{op}_R,\mathbb{T}^\circ)$ is a Cartesian Rosick\'y tangent category, which we stress is a new important example of a tangent category.  
\end{example}

We conclude this section by mentioning that the construction of Theorem \ref{theorem:operadic-tangent-categories} is functorial. Indeed, every operad morphism ${f: \P \to \P^\prime}$ induces a functor $\mathsf{ALG}^{op}_f: \mathsf{ALG}^{op}_{\P^\prime} \to \mathsf{ALG}^{op}_{\P}$ which is defined on objects and maps as $\mathsf{ALG}_f$. It is straightforward to check that $\mathsf{ALG}^{op}_f: (\mathsf{ALG}^{op}_{\P^\prime}, \mathbb{T}^\circ) \to (\mathsf{ALG}^{op}_\P, \mathbb{T}^\circ)$ is a Cartesian Rosick\'y tangent morphism. Now let $\mathsf{CRTAN}$ be the category of Cartesian Rosick\'y tangent categories and Cartesian tangent morphisms between them \cite[Section 4.3]{cockett2014differential}. Therefore we obtain a functor $\mathsf{ALG}^{op}: \mathsf{OPERAD}^{op}_R \to \mathsf{CRTAN}$ which sends an operad $\P$ to $(\mathsf{ALG}^{op}_\P, \mathbb{T}^\circ)$ and an operad morphism $f: \P \to \P^\prime$ to $\mathsf{ALG}^{op}_f: (\mathsf{ALG}^{op}_{\P^\prime}, \mathbb{T}^\circ) \to (\mathsf{ALG}^{op}_\P, \mathbb{T}^\circ)$. We note that the Cartesian tangent morphism $\mathsf{ALG}^{op}_f$ is not strong.

\subsection{Vector Fields of Algebras of an Operad}\label{sec:vf-operad}

In this section, we will explain how in the category of algebras of an operad, vector fields correspond precisely to derivations. Luckily, it turns out that it is already known that derivations are closely related to the semi-direct product, i.e., the tangent bundle. Indeed, one could apply \cite[Proposition 12.3.11]{LV} to get the desired result. However, let us give an alternative explanation using the coCartesian differential monad point of view. 

So let $\P$ be an operad and let $A$ be a $\P$-algebra. In Section \ref{sec:SDer} we explained how vector field of $A$ in $(\mathsf{ALG}_\P, \mathbb{T})$ correspond to $\mathsf{S}(\P,-)$-derivations of $A$. It turns out that $\mathsf{S}(\P,-)$-derivations on $A$ correspond precisely to $A$-derivations evaluated in itself. Indeed, $A$ is an $A$-module where the evaluation maps are induced from the $\P$-algebra structure $\theta: \mathsf{S}(\P,A) \to A$. Then we may consider $A$-derivations $D: A\to A$. 

\begin{lemma} For an operad $\P$ and a $\P$-algebra $A$, an $\mathsf{S}(\P,-)$-derivation $D: A \to A$ is precisely the same as a $A$-derivation $D: A \to A$. 
\end{lemma}
\begin{proof} Recall that an $R$-linear morphism $D: A \to A$ is an $A$-derivation if it satisfies: 
\begin{equation*} D(\mu(a_1,\dots,a_{n}))=\sum_{i=1}^{n}\mu(a_1,\hdots,D(a_i),\hdots,a_{n}).
\end{equation*}
On the other hand, an $R$-linear morphism $D: A\to A$ is an $\mathsf{S}(\P,-)$-derivation if $D \circ \theta = \theta \circ S\left(\P,\pi_1 + D \circ \pi_2 \right) \circ \partial_A$. Let us show that this equality is precisely the same as requiring $D$ be an $A$-derivation. For $\mu_A \in \P(0)$, since $\partial_A(\mu_A) = 0$, we have that $D(\mu_A) = 0$. For the rest, we compute: 
\begin{align*}
D(\mu(a_1,\dots,a_{n})) &= D\left( \theta(\mu; a_1, \hdots, a_n) \right) = \theta \left( S\left(\P,\pi_1 + D \circ \pi_2 \right) \left( \partial_A(\mu; a_1,\dots,a_{n}) \right) \right) \\
&=  \sum_{i=1}^{n}\theta \left( S\left(\P,\pi_1 + D \circ \pi_2 \right) \left(\mu; (a_1, 0),\hdots,(0, a_i),\hdots,(a_{n},0) \right) \right) \\
&= \sum_{i=1}^{n}\theta \left( \mu; a_1,\hdots,D(a_i),\hdots,a_{n} \right) = \sum_{i=1}^{n}\mu(a_1,\hdots,D(a_i),\hdots,a_{n}).
\end{align*}
We conclude that $\mathsf{S}(\P,-)$-derivations on $A$ and $A$-derivations evaluated in $A$ are indeed the same thing. 
\end{proof}

Therefore, by Lemma \ref{lemma:abstract-derivations-vector-fields}, we conclude that vector fields correspond to derivations as desired: 

\begin{proposition} For an operad $\P$ and a $\P$-algebra $A$, there is a bijective correspondence between vector fields of $A$ in $(\mathsf{ALG}_\P, \mathbb{T})$ and $A$-derivations $D: A \to A$. Therefore a vector field $v \in \mathsf{V}_\mathbb{T}(A)$ is precisely a $\P$-algebra morphism $v: A \to A \ltimes A$ such that $v(a) = (a, D_v(a))$ for all $a \in A$, where $D_v: A \to A$ is an $A$-derivation. Furthermore, the induced Lie bracket is given by $[v,w](a) = (a, D_v(D_w(a)) - D_w(D_v(a)))$. 
\end{proposition}

By Lemma \ref{lem:adjoint-vf}, we also have that vector fields in the opposite category of algebras also correspond precisely to derivations:

\begin{corollary} For an operad $\P$ and a $\P$-algebra $A$, there is a bijective correspondence between vector fields of $A$ in $(\mathsf{ALG}^{op}_\P, \mathbb{T}^\circ)$ and $A$-derivations $D: A \to A$. So a vector field $v \in \mathsf{V}_{\mathbb{T}^\circ}(A)$ is precisely a $\P$-algebra morphism $v: \Free_A(\Omega_A)  \to A$ which is defined on generators as $v(a)=a$ and $v(\mathsf{d}(a))=D_v(a)$ for all $a \in A$, where ${D_v: A \to A}$ is an $A$-derivation. Furthermore, the induced Lie bracket is given on generators by $[v,w](a) = a$ and $[v,w](\mathsf{d}(a))= D_v(D_w(a)) - D_w(D_v(a))$. 
\end{corollary}

Let us consider what vector fields are for our main examples of operads. 

\begin{example} For an $R$-algebra $A$ and an $A$-module $M$, an $M$-derivation evaluated in $M$
is just an $A$-linear endomorphism $f: M \to M$. Therefore a vector field of $M$ in $(\mathsf{MOD}_A, \mathbb{B})$ is an $A$-linear map $v: M \to M \times M$ such that $v(m)= (m,f_v(m))$ for some $A$-linear map $f_v: M \to M$. Similarly for vector fields of $M$ in $(\mathsf{MOD}^{op}_A, \mathbb{B}^\circ)$. 
\end{example}

\begin{example} For the operad $\Com$ and a commutative $R$-algebra $A$, an $A$-derivation evaluated in $A$ in the operadic sense is the same thing as a derivation in the classical sense, that is, $R$-linear morphism $D: A \to A$ which satisfies the product rule: $D(ab)= a D(b) + D(a) b$. Then a vector field of $A$ in $(\mathsf{CALG}_R, \mathbb{T})$ is an $R$-algebra morphism $v: A \to A[\epsilon]$ such that $v(a) = a + D_v(a) \epsilon$ for some derivation $D_v: A \to A$. Similarly, a vector field of $A$ in $(\mathsf{CALG}^{op}_R, \mathbb{T}^\circ)$ corresponds to an $R$-algebra morphism $v: \mathsf{Sym}_A \left( \Omega_A \right) \to A$ which is given on generators as $v(a) = a$ and $v(\mathsf{d}(a))=D_v(a)$ for some derivation $D_v: A \to A$. 
\end{example}

\begin{example} For the operad $\Ass$, derivations in the operadic sense again correspond to derivations in the classical sense as in the previous example. So vector fields in $(\mathsf{ALG}_R, \mathbb{T})$ or $(\mathsf{ALG}^{op}_R, \mathbb{T}^\circ)$ are given in essentially the same way as in the commutative case. 
\end{example}

\begin{example}
For the operad $\Lie$ and a Lie algebra $\Lieg$, a $\Lieg$-derivation evaluated in $\Lieg$ corresponds to an $R$-linear mormphsim $D: \Lieg \to \Lieg$ which satisfies $D([x,y])=[x,D(y)]+[D(x),y]$ for all $x,y \in \Lieg$. So a vector field of $\Lieg$ in $(\mathsf{LIE}_R, \mathbb{T})$ is a Lie algebra morphism $v: \Lieg \to \Lieg \ltimes \Lieg$ such that $v(x) = x + D_v(y) \epsilon$ for some $\Lieg$-derivation $D_v: \Lieg \to \Lieg$. Similarly, a vector field of $\Lieg$ in $(\mathsf{Lie}^{op}_R, \mathbb{T}^\circ)$ corresponds to an $R$-algebra morphism $v: \mathsf{T}^\circ(\Lieg) \to A$ which is given on generators as $v(a) = a$ and $v(\mathsf{d}(a))=D_v(a)$ for some $\Lieg$-derivation $D_v: \Lieg \to \Lieg$. 
\end{example}

\subsection{Differential Objects of an Operad}\label{subsec:diffobj-operad}

In this section, we will give precise characterizations of the differential objects in both the category of algebras of an operad and its opposite category. On the one hand, for the category of algebras, we will see that the differential objects are in some sense quite trivial algebras. On the other hand, we will show that the differential objects in the opposite category are quite rich and recapture a certain kind of module in the operadic sense.

Let us begin by taking a look at the differential objects in the category of algebras of an operad. 

\begin{proposition}\label{prop:diffobj-alg} Let $\P$ be an operad. Then a $\P$-algebra $A$ is a differential object in $(\mathsf{ALG}_\P, \mathbb{T})$ (in a necessarily unique way) if and only if $\mu(a_1, \hdots, a_{n}) =0$ for all $n\neq 1$, $\mu \in \P(n)$, and $a_i \in A$.
\end{proposition}
\begin{proof} Recall that the terminal object in $\mathsf{ALG}_\P$ is given by the zero $R$-module $\mathsf{0}$, whose $\P$-algebra structure is just given by $0$. Per the discussions in Section \ref{subsec:diffobj-ccdc}, a $\P$-algebra $A$ is a differential object if and only if $\pi_2: A \ltimes A \to A$, $\pi_1 + \pi_2: A \times A \to A$, and $0: \mathsf{0} \to A$ are all $\P$-algebra morphisms. Suppose that $A$ is a differential object. Then, as mentioned in  Section \ref{subsec:diffobj-ccdc}, we have that $A \times A = A \ltimes A$. This means that: 
\begin{align*}
 ( \mu(a_1, \hdots, a_{n}), \mu(b_1, \hdots, b_{n}) ) &= \mu((a_1, b_1),\hdots,(a_{n}, b_{n})) \\
 &=\left(\mu(a_1,\hdots,a_{n}),\sum_{i=1}^{n}\mu(a_1,\hdots,b_i,\hdots,a_{n})\right),
\end{align*}
which implies:
\begin{align*}
    \mu(b_1, \hdots, b_{n}) = \sum_{i=1}^{n}\mu(a_1,\hdots,b_i,\hdots,a_{n}), && \forall a_i, b_i \in A.
\end{align*}
The left-hand side does not depend on $a_i$, so we get: $\mu(b_1, \hdots, b_{n}) = \sum_{i=1}^{n}\mu(0,\hdots,b_i,\hdots,0)$. But by multilinearity, $\sum_{i=1}^{n} \mu(0,\hdots,b_i,\hdots,0) =0$ for $n \geq 2$. For $n=0$, we also get an empty sum. Therefore, $\mu(b_1, \hdots, b_{n}) = 0$ for all $n\neq 1$. Conversely, if in $A$ $\mu(b_1, \hdots, b_{n}) = 0$ for all $n\neq 1$, it is straightforward to show that the equalities in Lemma \ref{lem:diffobj-Salg} hold, and so, $A$ is a differential object. 
\end{proof}

Here are the differential objects for our main examples of operads: 

\begin{example} For any $R$-algebra $A$, $A^\bullet(1)=A$; every $A$-module $M$ is a differential object in $(\mathsf{MOD}_A, \mathbb{B})$, as per the discussion in Section \ref{subsec:diffobj-ccdc}. 
\end{example}

\begin{example} For the operad $\Com$, a commutative $R$-algebra $A$ is a differential object in $(\mathsf{CALG}_R, \mathbb{T})$ would imply that $A[\epsilon]\cong A \times A$ as $R$-algebras. However, the unit in $A[\epsilon]$ is $1$ while the unit in $A \times A$ is $(1,1)$. But then the isomorphism $A[\epsilon]\cong A \times A$ would imply that $1=0$. This is only the case for the zero $R$-algebra $\mathsf{0}$. So the only differential object in $(\mathsf{CALG}_R, \mathbb{T})$ is $\mathsf{0}$. 
\end{example}

\begin{example} For the operad $\Ass$, by the same argument as in the previous example, we have that the differential object in $(\mathsf{ALG}_R, \mathbb{T})$ is $\mathsf{0}$.
\end{example}

\begin{example} For the operad $\Lie$, it turns out that the differential objects are precisely the $R$-modules. Indeed, every $R$-module $V$ comes equipped with a trivial Lie bracket, $[v,w]=0$, which makes $V$ a Lie algebra and also that $V \ltimes V = V \times V$. Conversely, suppose that $\Lieg$ is a Lie algebra and a differential object, which in particular implies that $\Lieg \ltimes \Lieg = \Lieg \times \Lieg$. However, this implies that $\left( [x_1, x_2], [y_1, y_2] \right) = [(x_1, y_1), (x_2, y_2)] = \left( [x_1, x_2], [x_1, y_2] + [y_1, x_2] \right)$. Setting $x_i =0$, we get that $[y_1, y_2] =0$, which means that the Lie bracket of $\Lieg$ is trivial. So we have that the differential objects in $(\mathsf{LIE}_R, \mathbb{T})$ are precisely the $R$-modules with the trivial Lie bracket. 
\end{example}

Let us now turn our attention to differential objects in the opposite category of algebras of an operad $\P$. Luckily, as mentioned in Section \ref{sec:CDC-diffobj}, differential objects do not transfer through adjoint tangent structure. So even if $(\mathsf{ALG}_\P, \mathbb{T})$ may not have any non-trivial differential objects, we will show that $(\mathsf{ALG}^{op}_\P, \mathbb{T}^\circ)$ actually has many interesting differential objects. Recall that we mentioned that $\mathsf{ALG}_\P$ is cocomplete \cite[Proposition 6.4]{mandell:operadic-algebra-co-limits}, and therefore has coproducts. However, coproducts of $\P$-algebras are not straightforward and easy to work with. Luckily, there is an alternative but equivalent characterization of differential objects in a Cartesian Rosick\'y tangent category which does not involve the product $\times$. As such, this alternative description will allow us to describe differential objects in $(\mathsf{ALG}^{op}_\P, \mathbb{T}^\circ)$ without having to work with the coproduct in $\mathsf{ALG}_\P$. 

Firstly, it turns out that a differential object is in fact a special kind of differential bundle \cite[Definition 2.3]{cockettCruttwellDiffBundles}, which are analogues of smooth vector bundles in a tangent category. While differential bundles are beyond the scope of this paper (we invite interested readers to learn about them in \cite{cockettCruttwellDiffBundles,cruttwellLemay:AlgebraicGeometry,MacAdamVectorBundles}), it is enough to know that a differential object is the same thing as a differential bundle over the terminal object $\ast$ \cite[Proposition 3.4]{cockettCruttwellDiffBundles}. In \cite{MacAdamVectorBundles}, MacAdam provided an alternative description of a differential bundle in a Cartesian Rosick\'y tangent category, which in particular required less data and axioms than the original definition. Briefly, MacAdam showed that, in a Cartesian Rosick\'y tangent category, a differential bundle over an object $X$ can be characterized as an object $A$ with maps $q: A \to X$, called the projection, $\zeta: X \to A$, called the zero map, and $\ell: A \to \mathsf{T}(A)$, called the lift map. Furthermore, these need to satisfy: (1) four equalities, (2) that the pullback of $n$-copies of $q$ exists and is preserved by $\mathsf{T}^n$, and (3) $\ell$ has a universal property being part of a pullback square called the Rosick\'y's universality diagram, see \cite[Corollary 3]{MacAdamVectorBundles} or \cite[Proposition 3.8]{cruttwellLemay:AlgebraicGeometry} for full details. We have mentioned that differential objects are just differential bundles over the terminal object. So, in a Cartesian Rosick\'y tangent category, MacAdam's description is greatly simplified when $X = \ast$. Firstly, there is only one possible candidate for the projection: the unique map $t_A: A \to \ast$, so this map need not be specified. It follows that condition (2) is just saying that the product of $n$ copies of $A$ exists and is preserved by the tangent bundle, which is already true since we are in a Cartesian tangent category. So, condition (2) is automatically verified. So, one only needs the relations on the maps $\zeta: \ast \to A$ and ${\ell: A \to \mathsf{T}(A)}$. Now, one of the equalities in condition (1) is $t_A \circ \zeta = 1_\ast$, which is true by the universal property of the terminal object, so it is always satisfied. Lastly, the pullback square of condition (3) usually has $\mathsf{T}(\ast) \times A$ in the bottom corner. However, since $\mathsf{T}(\ast) \cong \ast$, we may rewrite Rosick\'y's universality diagram with $A$ in the bottom corner. Therefore, MacAdam's description allows us to provide a much simpler, yet equivalent, characterization of differential objects without referring to products. 

The following is just a combination of \cite[Proposition 3.4]{cockettCruttwellDiffBundles} with the rewritten version of \cite[Corollary 3]{MacAdamVectorBundles}, or \cite[Proposition 3.8]{cruttwellLemay:AlgebraicGeometry}, for the specific case of the terminal object. 

\begin{lemma} \label{lem:diffobjBen} Let $(\mathbb{X}, \mathbb{T})$ be a Cartesian Rosick\'y tangent category. Then, there is a bijective correspondence between: 
\begin{enumerate}[{\em (i)}]
\item Differential objects $(A, \hat{p}, \sigma, \zeta)$;
\item Triples $(A, \zeta, \ell)$ consisting of an object $A$, a map $\zeta: \ast \to A$, called the \textbf{zero map}, and a map $\ell: A \to \mathsf{T}(A)$, called the \textbf{differential lift}, such that the following equalities hold: 
\begin{align}\label{pre-diffeq}
    p_A \circ \ell = \zeta \circ t_A, && \ell \circ \zeta = z_A \circ \zeta, && \mathsf{T}(\ell) \circ \ell = l_A \circ \ell , 
\end{align}
and the following commutative diagram, called \textbf{Rosick\'y's universality diagram}, is a pullback square: 
 \begin{equation*}\begin{gathered} 
  \xymatrixcolsep{5pc}\xymatrix{ A \ar[r]^-{\ell}  \ar[d]_-{t_A} & \mathsf{T}(A) \ar[d]^-{p_A} \\
 \ast \ar[r]_-{\zeta}  & A }  \end{gathered}\end{equation*}
and for all $n \in \mathbb{N}$, $\mathsf{T}^n$ preserves this pullback. 
\end{enumerate}
\end{lemma}

In order to use Lemma \ref{lem:diffobjBen}, we must mention what the terminal object is in $\mathsf{ALG}^{op}_\P$. Firstly note that $\P(0)$ is a $\P$-algebra where the $\P$-algebra structure is just given by the composition operation of $\P$. Furthermore, $\P(0)$ is the initial object in $\mathsf{ALG}_\P$, that is, for every $\P$-algebra $A$ there is a unique $\P$-algebra morphism $t^\circ_A: \P(0) \to A$, which is defined as $t_A(\mu) = \mu_A$. Thus $\P(0)$ is the terminal object in $\mathsf{ALG}^{op}_\P$. 

So by Lemma \ref{lem:diffobjBen}, a differential object in $(\mathsf{ALG}^{op}_\P, \mathbb{T}^\circ)$ is a $\P$-algebra $A$ equipped with $\P$-algebra morphisms $\zeta^\circ: A \to \P(0)$ and $\ell^\circ: \mathsf{T}^\circ(A) \to A$ such that the dual of equalities in (\ref{pre-diffeq}) hold and the following square is a pushout diagram in $\mathsf{ALG}_\P$:
 \begin{equation*}\begin{gathered} 
  \xymatrixcolsep{5pc}\xymatrix{ A \ar[r]^-{p_A}  \ar[d]_-{\zeta^\circ} & \mathsf{T}^\circ(A) \ar[d]^-{\ell^\circ} \\
 \P(0) \ar[r]_-{t^\circ_A}  & A }  \end{gathered}\end{equation*}
Note that we do not need to mention that ${\mathsf{T}^\circ}^n$ preserves these pushout squares, since $\mathsf{T}^\circ$ is a left adjoint, and left adjoints always preserve colimits. 
Let us unpack this a bit more. For starters for every $a \in A$, $\zeta^\circ(a) \in \P(0)$ so we have $\zeta^\circ(a)_A \in A$. Therefore the dual of equalities in (\ref{pre-diffeq}) amount to the following identities on generators:
\begin{equation}\begin{gathered}\label{diffobg-algopeq}
     \ell^\circ(a) = \zeta^\circ(a)_A, \\
      \zeta^\circ(\zeta^\circ(a)_A)_A = \zeta^\circ(a)_A \qquad \zeta^\circ(\ell^\circ(\mathsf{d}(a)))_A = 0, \\
      \ell^\circ(\zeta^\circ(a)_{A}) = \zeta^\circ(a)_{A} \qquad   \ell^\circ( \mathsf{d}\left(\ell^\circ(\mathsf{d}(a)) \right) ) = \ell^\circ(\mathsf{d}(a)).
\end{gathered}
\end{equation}
Lastly, the pushout property says that for any $\P$-algebra $A^\prime$ and $\P$-algebra morphism $f: \mathsf{T}^\circ(A) \to A^\prime$ such that $f(a) = \zeta^\circ(a)_A$, there exists a unique $\P$-algebra morphism $f^\natural: A \to A^\prime$ such that $f^\natural(\ell(x)) = f(x)$ for all $x \in \mathsf{T}^\circ(A)$, so in particular on generators $f^\natural(\zeta^\circ(a)_A) = \zeta^\circ(a)_{A^\prime}$ and $f^\natural(\ell^\circ(\mathsf{d}(a))) = f(\mathsf{d}(a))$. 

By Proposition \ref{prop:diffobg-stanopp}, we already know that every free $\P$-algebra is a differential object: 

\begin{lemma}\label{lem:freeP-diffobj} Let $\P$ be an operad. Then for any $R$-module $V$, $\mathsf{S}(\P,V)$ is a differential object in $(\mathsf{ALG}^{op}_\P, \mathbb{T}^\circ)$ where in particular the zero $\zeta^\circ: \mathsf{S}(\P,V) \to \P(0)$ is defined as follows: 
\begin{align*}
   \zeta^\circ(\mu;v_1, \hdots, v_{n}) = 0 ,
\end{align*}
and the differential lift $\ell^\circ: \mathsf{T}^\circ(\mathsf{S}(\P,V)) \to \mathsf{S}(\P,V)$ is defined as follows on generators: 
\begin{align*}
   \ell^\circ(\mu; v_1, \hdots, v_{n}) &= 0, &  \ell^\circ\left(\mathsf{d}(\mu; v_1, \hdots, v_{n}) \right) &= (\mu; v_1, \hdots, v_{n}).
\end{align*}
Furthermore, $\mathsf{T}^\circ(\mathsf{S}(\P,V)) \cong \mathsf{S}(\P,V \times V)$ as $\P$-algebras.
\end{lemma}

While free $\P$-algebras always give differential objects, it is possible that there are other differential objects. That said, we are still able to precisely characterize the differential objects as $\P(0)$-modules (in the operadic sense). 

\begin{thm}\label{thm:diffobj-operad-op} Let $\P$ be an operad. Then there is a bijective correspondence between differential objects in $(\mathsf{ALG}^{op}_\P, \mathbb{T}^\circ)$ and $\P(0)$-modules. 
\end{thm}
\begin{proof} Let $(A, \zeta^\circ, \ell^\circ)$ be a differential object in $(\mathsf{ALG}^{op}_\P, \mathbb{T}^\circ)$. Let $\mathsf{D}_{\ell^\circ}(a) = \ell^\circ(\mathsf{d}(a))$ and let $\mathsf{D}_{\ell^\circ}(A)$ be the image of $\mathsf{D}_{\ell^\circ}$, so $\mathsf{D}_{\ell^\circ}(A)$ is an $R$-module. We claim that the following equips $\mathsf{D}_{\ell^\circ}(A)$ with a $\P(0)$-module structure with evaluation: 
\begin{align*}
    \psi_{n+1}\left(\mu;\nu_1, \hdots, \nu_n, \ell^\circ(\mathsf{d}(a)))\right) := \ell^\circ\left( \mathsf{d}\left( \mu( {\nu_1}, \hdots, {\nu_n}, a ) \right) \right).
\end{align*}
Indeed, by definition, since $\mathsf{d}$ is a derivation, one has: 
\begin{align*}
     \ell^\circ\left( \mathsf{d}\left( \mu( {\nu_1}, \hdots, {\nu_n}, a ) \right) \right)&= \ell^\circ\left( \sum_{i=1}^n \mu( {\nu_1}, \hdots,\mathsf d\nu_i,\hdots {\nu_n}, a ) \right)+\ell^\circ\left(  \mu( {\nu_1}, \hdots, {\nu_n}, \mathsf da ) \right)\\
         &=\mu\left( {\nu_1}, \hdots, {\nu_n}, \ell^\circ(\mathsf da) \right),
\end{align*}
which gives a $\P(0)$-module structure on $\mathsf{D}_{\ell^\circ}(A)$.

Conversely, let $M$ be a $\P(0)$-module. Consider the free $\P(0)$-algebra $\Free_{\P(0)}(M)$. As shown in Proposition \ref{prop:free-A-algebra-functor}, $\Free_{\P(0)}(M)$ is generated by $\mu \in \P(0)$ and $x \in M$. However, thanks to the relations on $\Free_{\P(0)}(M)$, the generators $\mu\in\P(0)$ correspond in fact to the units coming from the $\P$-algebra structure, so $\Free_{\P(0)}(M)$, as $\P(0)$-algebra is generated by $x\in M$. The $\P$-algebra  $\mathsf{T}^\circ(\Free_{\P(0)}(M))$ has generators  $x$, and $\mathsf{d}(x)$ for all $x \in M$. Now, define the differential lift $\ell^\circ: \mathsf{T}^\circ(\Free_{\P(0)}(M)) \to \Free_{\P(0)}(M)$ as the $\P$-algebra morphism defined as follows on generators: 
\begin{align*}
   \ell^\circ(x) &= 0, & \ell^\circ(\mathsf{d}(x)) &= x,
\end{align*}
which is indeed well-defined, since $\ell^\circ$ can be constructed using the universal properties. Next, we define the zero ${\zeta^\circ: \Free_{\P(0)}(M) \to \P(0)}$ as the $\P$-algebra morphism defined as follows on generators: 
\begin{equation*}
    \zeta^\circ(x) = 0.
\end{equation*}
It is straightforward to check that $\ell^\circ$ and $\zeta^\circ$ satisfy the equalities of (\ref{diffobg-algopeq}). Lastly for the pushout square, suppose that there was a $\P$-algebra morphism $f: \mathsf{T}^\circ(\Free_{\P(0)}(M)) \to A^\prime$ such that  $f(x) = 0$. Then define the $\P$-algebra morphism $f^\natural: \Free_{\P(0)}(M) \to A^\prime$ on generators as $f^\natural(x) = f(\mathsf{d}(x))$. By construction, $f^\natural$ satisfies the necessary identities and is the unique map that does so since it was defined on generators. So we conclude that $(\Free_{\P(0)}(M), \zeta^\circ, \ell^\circ)$ is indeed a differential object. 

We must now explain why these constructions are inverses of each other. So starting from a $\P(0)$-module $M$, by definition of $\ell^\circ$ we already have that $\mathsf{D}_{\ell^\circ}(\Free_{\P(0)}(M))=M$ (and this is an equality of $\P(0)$-modules). In the other direction, start with a differential object $(A, \zeta^\circ, \ell^\circ)$. Define the $\P$-algebra morphism $\phi: A \to \Free_{\P(0)}(\mathsf{D}_{\ell^\circ}(A))$ using the universal property of the pushout as the unique $\P$-algebra morphism such that: 
\begin{equation*}
\phi(\ell^\circ(\mathsf{d}(a))) = \ell^\circ(\mathsf{d}(a))  .
\end{equation*}
Then, $\phi$ is a $\P$-algebra isomorphism with inverse $\phi^{-1}: \Free_{\P(0)}(\mathsf{D}_{\ell^\circ}(A)) \to A$ defined on generators as:
\begin{equation*}
\phi^{-1}(\ell^\circ(\mathsf{d}(a))) = \ell^\circ(\mathsf{d}(a)).
\end{equation*}
We have that $A \cong \Free_{\P(0)}(\mathsf{D}_{\ell^\circ}(A))$ as $\P$-algebras, and it easy to see that $\phi$ preserves the differential object structure $\zeta^\circ$ and $\ell^\circ$. So, we conclude that there is indeed a bijective correspondence between differential objects and $\P(0)$-modules as desired.  
\end{proof}

Before considering the differential objects in our main examples, we point out that $\P(0)$-modules also have an alternative description in terms of modules in the usual sense. 

\begin{lemma}\cite[Lemma 1.4]{moerdijkEnvelopingOperads} For an operad $\P$, $\P(0)$-modules in the operadic sense correspond precisely to $\P(1)$-left modules in the usual sense.
\end{lemma}

Let us consider the differential objects in our main examples of operads. In particular, we note that the first example has differential objects which are not simply free $\P$-algebras. On the other hand, for the last three examples, the differential objects turn out to be precisely free $\P$-algebras.

\begin{example} For an $R$-algebra $A$, $A^\bullet(1)=A$; so every $A$-module is a differential object in $(\mathsf{MOD}^{op}_A, \mathbb{B}^\circ)$, as per the discussion in Section \ref{subsec:diffobj-ccdc}. 
\end{example}

\begin{example} For the operad $\Com$, $\Com(0)=R$ and $\Com(1)=R$. We also have that $\Free_R = \mathsf{Sym}$. Therefore for any $R$-module $V$, $\mathsf{Sym}(V)$ is a differential object in $(\mathsf{CALG}^{op}_R, \mathbb{T}^\circ)$. If $V$ is a free $R$-module with basis $X$, then recall that $\mathsf{Sym}(V) \cong R[X]$ and so $\mathsf{T}^\circ(\mathsf{Sym}(V)) \cong R[X, dX]$, as in Example \ref{example:CDCCom}. So in terms of polynomials, the differential object structure is defined as follows 
\begin{align*}
    \hat{p}^\circ( q(x_1, \hdots, x_n) ) &= q(dx_1, \hdots, dx_n), \\  \zeta^\circ(q(x_1, \hdots, x_n)) &= q(0, \hdots, 0), \\
    \sigma^\circ(q(x_1, \hdots, x_n)) &= q(x_1 + dx_1, \hdots, x_n + dx_n), \\
    \ell^\circ( q(x_1, \hdots, x_n, dy_{1}, \hdots, dx_{m}) ) &= q(0, \hdots, 0, dy_{1}, \hdots, dx_{m}).
\end{align*}
This recaptures \cite[Theorem 5.9]{cruttwellLemay:AlgebraicGeometry}.
\end{example}

\begin{example} For the operad $\Ass$, $\Ass(0)=R$ and $\Ass(1)=R$. We also have that $\Free_R = \mathsf{Ten}$. So for any $R$-module $V$, $\mathsf{Ten}(V)$ is a differential object in $(\mathsf{ALG}^{op}_R, \mathbb{T}^\circ)$. For the case of free $R$-modules, we may describe the differential object structure in terms of non-commutative polynomials as in the previous example. 
\end{example}

\begin{example} For the operad $\Lie$, $\Lie(0)=\mathsf 0$ and $\Lie(1)=R$. We also have $\Free_{\mathsf 0} = \mathsf{Lie}$. Thus for any $R$-module $V$, the free Lie algebra $\mathsf{Lie}(V)$ is a differential object in $(\mathsf{Lie}^{op}_R, \mathbb{T}^\circ)$.
\end{example}

\section{Future Work}\label{sec:future}

We conclude this paper by discussing some interesting future research projects that build upon the theory of tangent categories of algebras over an operad. 

\begin{enumerate}[{\em (i)}]
\item In \cite{cruttwellLemay:AlgebraicGeometry}, it was shown that the study of the tangent structure on the opposite category of commutative algebras provides many concepts from the algebraic geometry of affine schemes. Similarly, the tangent structure on the opposite category of associative algebras formalizes many constructions of Ginzburg \cite{Ginzburg:notes-on-NCGeometry} related to non-commutative geometry. It is natural to ask the question: what kind of geometry can be described using the opposite category of Lie algebras? Is this somehow related to the geometry of Lie algebras studied, for example, by Francis and Gaitsgory \cite{francis2012}? What are the geometries obtained using the operads of PreLie algebras or Poisson algebras? Can we describe the non-commutative Poisson geometry studied by Van den Bergh \cite{bergh08} using our techniques? Even more generally, one could use tangent categories to provide a new version of algebraic geometry relative to an operad.

\item Differential bundles in a tangent category \cite{cockettCruttwellDiffBundles} generalize the notion of smooth vector bundles. In \cite{cruttwellLemay:AlgebraicGeometry}, it was shown that differential bundles over a commutative algebra correspond precisely to modules over said commutative algebra. As such, we conjecture the differential bundles over an algebra of an operad will also correspond to modules over the algebra (in the operadic sense). This would be an extension of Theorem \ref{lemma:dual-tangent-structure} since a differential object can equivalently be described as a differential bundle over the terminal object. This will be investigated by the second named author for their PhD thesis. Beyond differential bundles, one should also study other interesting tangent category notions in the (opposite) category of algebras over an operad. For example, what are the tangent category versions of connections \cite{cockettCruttwellConnections}, or de Rham cohomology \cite{CruttwellLucychynCohomology}, or even solving differential equations \cite{cockett2021differential} in these tangent categories? 

\item The story of this paper was to explain how, from operad, one could build tangent categories and Cartesian differential categories. A natural question is if we can go in the other direction. So for what kinds of tangent categories or Cartesian differential categories is possible to construct an operad? Similarly, it would be of interest to precisely characterize which tangent categories are equivalent to the (opposite) category of algebras of an operad. 

\item In \cite{bauerburkeching21}, Bauer, Burke and Ching generalized the notion of tangent categories to the higher categorical setting. This new concept of tangent $\infty$-category allows one to study tangent structures “up to higher coherences’’. There is a well-established notion of $\infty$-operad \cite{cisinskimoerdijk}, and this encodes operations “up to homotopy”. It then seems natural to ask whether our theory can be generalized to produce a tangent $\infty$-structure on the (opposite) category of algebras over an $\infty$-operad. If so, there are plenty of potential applications: replacing the operad $\Com$ by the $\infty$-operad $E_\infty$ of commutative algebras up to infinity could give an insight into the well-defined notion of derived algebraic geometry \cite{toenvezzosi05}. Using a replacement for $\Lie$ could recapture notions from the geometry of Lie algebras of Francis and Gaitsgory, and from works of Harpaz, Nuiten and Prasma \cite{francis2012,harpazetal19}. Replacing $\Ass$ by the $\infty$-operad $A_\infty$ should recapture the theory of $A_\infty$-geometry of Kontsevich and Soibelman  \cite{kontsevich09}.

\item In Sections \ref{sec:operads} we discussed briefly the functoriality of the two constructions. In particular, we mentioned how every operad morphism provides strong/strict tangent morphisms relating to each construction. This fact should play an important role in better understanding the link between operad theory and tangent category theory.
\end{enumerate}

So there are many potential interesting paths to take for future work regarding operads and tangent categories.

\addcontentsline{toc}{section}{References}
\bibliographystyle{plain}
\bibliography{mainbib.bib}

\end{document}